\documentclass[12pt,oneside,a4paper]{amsart}

\usepackage{amssymb}
\usepackage{amsmath}
\usepackage{amsthm}
\usepackage{amscd}
\usepackage[all]{xy}
\usepackage{longtable}
\usepackage{mathrsfs}
\usepackage[dvipdfmx]{xcolor}
\usepackage[dvipdfmx]{pict2e}
\usepackage[dvipdfmx]{graphicx}

\usepackage{comment}
\usepackage{todonotes}

\usepackage[dvipdfmx]{hyperref}
\usepackage{hyperref}
\hypersetup{
	colorlinks=true,
	linkcolor=red,
	citecolor=blue}
	\usepackage[utf8]{inputenc}
\usepackage[T1]{fontenc}
\usepackage{color}

\usepackage{extarrows}
\usepackage{tikz}

\usepackage{geometry}
\theoremstyle{plain}
\newtheorem{theorem}{Theorem}[section]
\newtheorem{lemma}[theorem]{Lemma}
\newtheorem{corollary}[theorem]{Corollary}
\newtheorem{proposition}[theorem]{Proposition}
\newtheorem{remark}[theorem]{Remark}

\theoremstyle{definition}
\newtheorem{definition}[theorem]{Definition}

\newtheorem{conjecture}[theorem]{Conjecture}

\newtheorem*{observation}{Observation}

\newcommand{\kah}{K\"{a}hler }
\newcommand{\idd}{i\partial\overline{\partial}}
\newcommand{\dbar}{\overline{\partial}}

\newcommand{\cal}[1]{\mathcal{#1}}
\newcommand{\bb}[1]{\mathbb{#1}}
\newcommand{\scr}[1]{\mathscr{#1}}
\newcommand{\rom}[1]{\mathrm{#1}}
\newcommand{\fra}[1]{\mathfrak{#1}}

\newcommand{\wit}[1]{\widetilde{#1}}

\newcommand{\lla}[0]{{\langle\!\langle}}
\newcommand{\rra}[0]{{\rangle\!\rangle}}
\newcommand{\lara}[2]{\langle{#1},{#2}\rangle}

\newcommand{\iO}[1]{i\Theta_{#1}}


\begin{document}
\title
[Bigness and approximation theorems] 
{Bigness of adjoint linear subsystem and approximation theorems with ideal sheaves on weakly pseudoconvex manifolds}
\author{Yuta Watanabe}
\address{Department of Mathematics, Faculty of Science and Engineering, Chuo University.
1-13-27 Kasuga, Bunkyo-ku, Tokyo 112-8551, Japan}
\email{{\tt wyuta.math@gamil.com}, {\tt wyuta@math.chuo-u.ac.jp}}

\begin{abstract}
    Let $X$ be a weakly pseudoconvex manifold and $L\longrightarrow X$ be a holomorphic line bundle with a singular positive Hermitian metric $h$.
    In this article, 
    we provide a points separation theorem and an embedding for the adjoint linear subsystem including the multiplier ideal sheaf $\mathscr{I}(h^m)$, with respect to an appropriate set excluding a singular locus of $h$. 
    We also show that the adjoint bundle of $L$ is big, which constitutes a generalization to weakly pseudoconvex manifolds of Demailly's characterization of positivity in complex and algebraic geometry. 
    To handle analytical methods, an approximation of singular Hermitian metrics is first constructed based on Demailly's approximation, using the strong openness property, preserving the ideal sheaves and compatible with blow-ups.
    Using the blow-ups obtained from this approximation, the singular holomorphic Morse inequalities and the approximation theorem for holomorphic sections, each twisted by the ideal sheaves, are established. 
    This approximation theorem for sections provides the key to globalization, leading to global embeddings.
\end{abstract}


\maketitle

\vspace{-5mm}

\tableofcontents

\vspace*{-8mm}

\section{Introduction}

In complex geometry, the concepts of positivity and vanishing and embedding theorems play a crucial role. 
This is exemplified by Kodaira's vanishing and embedding theorems \cite{Kod53,Kod54}, which characterize compact Hodge manifolds, and as is well known, \textit{positive} and \textit{ample} coincide. 
Moreover, the important positivity concept of \textit{big} in algebraic geometry was shown to be equivalent to \textit{singular} \textit{positive} by Demailly \cite{Dem90}. 
As an extension of this result on projective manifolds, the equivalence of possessing a singular positive line bundle and being \textit{Moishezon} has been established for compact manifolds (see \cite{JS93,Tak94,Bon98}, \cite[Chapter\,2]{MM07}).

Whether a manifold is compact or not is clearly an important distinction that can be seen from the existence of global holomorphic functions.
We consider embeddings on weakly pseudoconvex manifolds, which form a class of manifolds as broad as possible, including both compact and Stein manifolds, and are of great importance. 
In fact, every complex Lie group is always weakly pseudoconvex (see \cite{Kaz73}), and weakly pseudoconvexity is preserved under inverse images of proper holomorphic maps such as modifications and blow-ups, which ensures that such manifolds exist in abundance.
Here, a complex manifold $X$ is said to be \textit{weakly} \textit{pseudoconvex} if there exists a smooth function $\varPsi:X\longrightarrow\bb{R}\cup\{-\infty\}$ which is plurisubharmonic and exhaustive.
We often say that $(X,\varPsi)$ is weakly pseudoconvex, and define the sublevel set as $X_c:=\{x\in X\mid\varPsi(x)<c\}$ for any $c\in\bb{R}$, which is then relatively compact.
Regarding positive line bundles, Kodaira's embedding theorem has been effectively formulated by Takayama for weakly pseudoconvex manifolds (see \cite[Theorem\,1.2]{Tak98}). 
Unlike in the compact case, a positive line bundle $L$ on $X$ does not directly become ample (see \cite{Ohs79}); instead, it is necessary to employ the theory of adjoint bundles. 
As a conclusion, the adjoint bundle \( K_X \otimes L^{\otimes m} \) is ample for any integer \( m\geq n(n+1)/2 \), where $n=\rom{dim}\,X$.

In this paper, for a singular positive line bundle $(L,h)$ on weakly pseudoconvex manifolds, 
we provide a points separation theorem and an embedding for the adjoint linear subsystem including the multiplier ideal sheaf $\scr{I}(h^m)$, with respect to an appropriate set excluding a singular locus of $h$. 
We also show that the adjoint bundle of $L$ is big.
However, since merely having a singular positive Hermitian metric makes it difficult to apply analytic techniques, we adopt an effective approach by using Demailly's approximation technique on relatively compact subsets (see Theorem \ref{Dem appro preserves ideal sheaves}, \cite{DPS01,Mat22}). 
Furthermore, to handle the singularities of $h$, Demailly's approximation can be effectively refined by using the strong openness property (see \cite{GZ15}) to construct an approximation of $h$ that preserves the multiplier ideal sheaf and for which the blow-ups along the singular sets of each approximation are compatible with the multiplier ideal sheaves (see Theorem \ref{Dem appro with log poles and ideal sheaves}).
We then use this Demailly-type approximation Theorem \ref{Dem appro with log poles and ideal sheaves} to establish the corresponding blow-ups (see Theorem \ref{Blow ups of Dem appro with log poles and ideal sheaves}).

\begin{observation}
By using this blow-ups Theorem \ref{Blow ups of Dem appro with log poles and ideal sheaves} and Takayama's embedding theorem (see \cite[Theorem\,1.2]{Tak98}), we establish the following for each weakly pseudoconvex sublevel set (see Theorem \ref{Thm bigness with ideal sheaves on each sublevel sets}): 
For any \( c<\sup_X\varPsi \), there exist a positive integer \( p_c \) and a singular positive Hermitian metric \( h_c \) on \( L|_{X_c} \) such that \( K_{X_c} \otimes L^{\otimes p_c} \otimes \mathscr{I}(h_c^{p_c}) \) is \textit{big} as in Definition \ref{Def of bigness}, and that \( X_c\setminus Z_c \) can be embedded into a projective space $\bb{P}^{2n+1}$, where $Z_c$ is an analytic subset of $X_c$ obtained as the singular locus of $h_c$.
\end{observation} 

The essential reason for employing the theory of adjoint bundles is to ensure that \( p_c \) can be bounded as \( c\to\sup_X\varPsi \) and to extend holomorphic sections defined on each sublevel set \( X_c \) to the entire space \( X \) using the approximation theorem for holomorphic sections of the adjoint bundle.
The approximation theorem regarding positive line bundles is established by Nakano, Kazama and Ohsawa, and is one of most important conclusion of the cohomology theory on weakly pseudoconvex manifolds (see \cite{Ohs82,Ohs83}). 
We first present the following approximation theorem 
including the multiplier ideal sheaf of the singular positive Hermitian metric, as a key main theorem. 

\begin{theorem}$(=\,$\textnormal{Theorem\,\ref{Approximation thm of hol section with ideal sheaves}}$)$\label{Approximation thm in Introduction}
    Let $(X,\varPsi)$ be a weakly pseudoconvex manifold and $L$ be a holomorphic line bundle with a singular Hermitian metric $h$. 
    If $h$ is singular positive, then for any sublevel set $X_c$, the natural restriction map
    \begin{align*}
        \rho_c:H^0(X,K_X\otimes L\otimes\scr{I}(h))\longrightarrow H^0(X_c,K_X\otimes L\otimes\scr{I}(h))
    \end{align*}
    has dense images with respect to the topology induced by the $L^2$-norm $||\bullet||_h$ and the topology of uniform convergence on all compact subsets in $X_c$. 
\end{theorem}

Throughout this paper, we let $X$ be an $n$-dimensional weakly pseudoconvex manifold, $\omega_X$ be a Hermitian metric and $L$ be a holomorphic line bundle on $X$ with a singular Hermitian metric $h$.
Here, if we take an arbitrary smooth Hermitian metric $h_0$ on $L$, the singular Hermitian metric $h$ can be written as $h=h_0e^{-2\varphi}$ for some locally integrable function $\varphi\in L^1_{loc}(X,\bb{R})$.
Unless otherwise stated, \( h \) is assumed to be \textit{singular} \textit{positive}.

Let us formulate our main results.
To provide embeddings and points separation on appropriate subsets, it is necessary to understand the zero subspace \( V(\mathscr{I}(h^m)) \) as \( m \to +\infty \).
Thus, considering the Lelong number's non-zero locus $E_{+}(h):=\{x\in X\mid \nu(-\log h,x)>0\}$, we have $E_{+}(h)=\bigcup_{m\in\bb{N}}V(\scr{I}(h^m))$.
On the other hand, regarding \( h_c \) and \( Z_c \) in the above observation, when globalizing, we handle the union of the analytic subsets \( Z_c \) on $X_c$, i.e., $\bigcup_{c\in\bb{N}}Z_c$, which is well-defined on $X$ by $Z_c\subset Z_{c+1}|_{X_c}$.
This union \( \bigcup_{c\in\bb{N}} Z_c \) is defined independently of the construction method, leading to \( \bigcup_{c\in\bb{N}} Z_c = E_{+}(h) \) (see Proposition \ref{Prop E0(h)=cup Zj}).
Using this appropriate singular loci subset \( E_{+}(h)=\bigcup_{m\in\bb{N}}V(\scr{I}(h^m)) \), we obtain the following global embedding. 

\begin{theorem}$(=\,$\textnormal{Theorem\,\ref{Thm global points separation}\,and\,\ref{Thm global embedding}}$)$\label{Thm global embedding in Introduction}
    We assume that $X$ is non-compact. For a Hermitian metric $\omega_X$ satisfying $\iO{L,h}\geq\omega_X$ in the sense of currents on $X$,
    if there exist a constant $C>0$ and a smooth Hermitian metric $h_0$ on $L$, which does not necessarily carry positivity, such that $C\omega_X\geq\iO{L,h_0}$ on $X$, then 
    we have the following main results. 
    \begin{itemize}
        \item [] \hspace{-10mm} $($\textnormal{\textbf{Global points separation}}$)$ Let $x_1,\ldots,x_r$ be $r$ distinct points on $X\setminus\bigcup_{k\in\bb{N}}\!V(\scr{I}(h^k))$. 
        Then for any positive integer $m\geq C\big(n(n+2r-3)/2+2-r\big)$, the space $H^0(X,K_X\otimes L^{\otimes m}\otimes\scr{I}(h^m))$ separates $\{x_j\}^r_{j=1}$, 
        i.e., the restriction map $H^0(X,K_X$ $\otimes\, L^{\otimes m}\otimes\scr{I}(h^m))\longrightarrow \bigoplus^r_{j=1}\cal{O}_X/\fra{m}_{X,x_j}$ is surjective. 

        In particular, if $m\geq C\big(n(n-1)/2+1\big)$ then the adjoint bundle $K_X\otimes L^{\otimes m}$ is generated by global sections on $X\setminus\bigcup_{k\in\bb{N}}V(\scr{I}(h^k))$. 
        \item [] \hspace{-10mm} $($\textnormal{\textbf{Global embedding}}$)$ Let $m$ be a positive integer with $m\geq C\cdot n(n+1)/2$. 
        Then, any open subset of $X\setminus\bigcup_{k\in\bb{N}}V(\scr{I}(h^k))$ 
        is holomorphically embeddable in to $\bb{P}^{2n+1}$ by a linear subsystem of $|(K_X\otimes L^{\otimes m})^{\otimes n+1+\ell}\otimes L^{\otimes p}\otimes\scr{I}(h^{p+1})|$ for any $\ell\geq1$ and $p\geq0$.
        In particular, 
        the adjoint bundle $K_X\otimes L^{\otimes m}$ is ample over any open subset of $X\setminus\bigcup_{k\in\bb{N}}V(\scr{I}(h^k))$, and 
        the tensor power of the adjoint bundle $(K_X\otimes L^{\otimes m})^{\otimes n+2}$ is very ample over any open subset of $X\setminus\bigcup_{k\in\bb{N}}V(\scr{I}(h^k))$.
    \end{itemize}
\end{theorem}

Note that Theorem \ref{Approximation thm in Introduction} and \ref{Thm global embedding in Introduction} hold without assuming K\"ahler-ness.
Here, if $h$ is smooth positive then we can take $C=1$ and $\omega_X=\iO{L,h}$. 
This and Remark \ref{Remark: in Introduction for add semi-positivity} below shows that Theorem \ref{Thm global embedding in Introduction} generalizes Takayama's result \cite[Theorem\,1.2]{Tak98}.

\begin{remark}$(=\,$\textnormal{Remark\,\ref{remark of only analytic singularities}}$)$
    In Theorem \ref{Thm global embedding in Introduction}, if \( h \) has only analytic singularities, then \( E_{+}(h) = \bigcup_{k\in\bb{N}}V(\scr{I}(h^k)) \) is an analytic subset, and under the same assumptions as Theorem \ref{Thm global embedding}, the complement  \( X\setminus E_{+}(h) \) is holomorphically embeddable in to \( \bb{P}^{2n+1} \).
\end{remark}

\begin{remark}\label{Remark: in Introduction for add semi-positivity}
    In Theorem \ref{Thm global embedding in Introduction}, furthermore, if $L\longrightarrow X$ is assumed to be semi-positive, 
    then the condition on $C$ is unnecessary and we may take $C = 1;$ see Theorem \ref{Thm global embedding with semi-positive}.
\end{remark}

Let us consider the points separation theorem necessary for Theorem \ref{Thm global embedding in Introduction}. 
It is sufficient to establish the following (see Theorem \ref{Ext [Tak98, Thm 4.2]}) from Approximation Theorem \ref{Approximation thm in Introduction}.
Here, a subvariety refers to a reduced and irreducible complex subspace, and for any point $x\in X_c$ we set 
\begin{align*}
    d_c(x):=\max\{\rom{dim}\,V\mid V \text{ is a } compact \text{ subvariety of } X_c \text{ passing through } x\}.
\end{align*}

\begin{theorem}\label{Thm points separation thm in Introduction}$($\textnormal{cf.\,Theorem\,\ref{Ext [Tak98, Thm 4.1]}\,and\,\ref{Ext [Tak98, Thm 4.2]}}$)$
    As in the above observation $($or $\S\ref{Setting of each sublevel set})$, we take the singular positive Hermitian metric $h_c$ on $L|_{X_c}$ and the analytic subset $Z_c$ obtained as the singular locus of $h_c$
    from Demailly-type approximation Theorem \ref{Dem appro with log poles and ideal sheaves} $($or blow-ups Theorem \ref{Blow ups of Dem appro with log poles and ideal sheaves}$)$.
    Let $x_1,\ldots,x_r$ be $r$ distinct points on a subset $X_c\setminus Z_c$ and let $d_r:=\max_{1\leq j\leq r} d_c(x_j)$.
    For any integer $m\geq d_r(d_r+2r-1)/2+1$, there exists a positive rational number $\vartheta_c$ smaller than $d_r(d_r+2r-1)/2+1$ such that 
    the space $H^0(X_c,K_X\otimes L^{\otimes m}\otimes\scr{I}(h_c^{m-\vartheta_c}))$ separates $\{x_j\}^r_{j=1}$, i.e., the following restriction map is surjective.
    \begin{align*}
        H^0(X_c,K_X\otimes L^{\otimes m}\otimes\scr{I}(h_c^{m-\vartheta_c})) \longrightarrow \bigoplus^r_{j=1}\cal{O}_X/\fra{m}_{X,x_j}.
    \end{align*}
\end{theorem}

Here, by assuming global conditions similar to Theorem \ref{Thm global embedding in Introduction} for a smooth metric \( h_0 \), we can take \( \vartheta_c = 0 \) in Theorem \ref{Thm points separation thm in Introduction}. 
Within this and the Approximation Theorem \ref{Approximation thm in Introduction}, it becomes possible to obtain the global points separation theorem (= Theorem \ref{Thm global points separation}).

Theorem \ref{Thm points separation thm in Introduction} is based on Takayama's arguments \cite{Tak98}, which localize the discussions from \cite{AS95} and employ Kawamata-Shokurov's concentration method \cite{KMM87}. 
To overcome the difficulties arising from singular positivity, we apply the following. 
\begin{itemize}
    \item Non-vanishing theorem (= Theorem \ref{Non-vanishing thm}) obtained from the singular holomorphic Morse inequality including multiplier ideal sheaves (= Theorem \ref{singular holomorphic Morse inequality on w.p.c}).
    \item Application of Kawamata-Shokurov's concentration method (see Lemma \ref{Ext [Lemma-Definition 4.4, Tak98]: singular positive var}), made possible through the appropriate use of blow-ups and vanishing theorems.
    \item Nadel-type vanishing theorem (= Theorem \ref{Thm Nadel type vanishing without Kahler}) without \kah metrics.
\end{itemize}
In this paper, unlike the case of positive line bundles, \( X \) and \( X_c \) generally admit a \kah current due to singular positivity but do not necessarily possess a \kah metric. 
Note that if a Moishezon manifold admits a \kah metric, it becomes projective (see \cite{Moi66}), 
thus the existence of a \kah metric is an important condition.

In relation to the first condition above, we discuss the volume and big-ness. 
In the compact case, as is well known from Siegel's Lemma, the Riemann-Roch theorem establishes that \( L \) being big is equivalent to \( \mathrm{Vol}_X(L) > 0 \). 
A more precise characterization of Moishezon-ness is obtained through Bonavero's type singular holomorphic Morse inequality (see \cite{Tak94,Bon98}), which generalizes the Riemann-Roch theorem.
We obtain the following as an extension of this Morse inequality to the non-compact case.

\begin{theorem}$(=\,$\textnormal{Theorem\,\ref{singular holomorphic Morse inequality on w.p.c}}$)$
    Let $E$ be a holomorphic vector bundle on $X$.
    There exists a nowhere dence subset $\Sigma\subset(\inf_X\varPsi,\sup_X\varPsi)$ such that for any $c\in (\inf_X\varPsi,\sup_X\varPsi)\setminus\Sigma$, 
    there exist a singular positive Hermitian metric $h_c$ on $L|_{X_c}$ obtained from Theorem \ref{Dem appro with log poles and ideal sheaves}, satisfying $\scr{I}(h)=\scr{I}(h_c)$ and a positive integer $t_c\in\bb{N}$, 
    and we obtain the singular holomorphic Morse inequality
    \begin{align*}
        \rom{dim}\,H^0(X_c,E\otimes L^{\otimes p}\otimes\scr{I}(h_c^p))\geq\rom{rank}\,E\cdot\frac{p^n}{t_c^n n!}\int_{X_c(\leq1)}\Bigl(\frac{i}{2\pi}\Theta_{L,h_c}\Bigr)^n+o(p^n).
    \end{align*}
    
    Furthermore, if there exists $c>\inf_X\varPsi$ such that $E_{+}(h)\subset X_c$, i.e., $E_{+}(h)$ is compact, or that $E_{+}(h)\bigcap X_c=\emptyset$, then we have a global singular holomorphic Morse inequality
    \begin{align*}
        \rom{dim}\,H^0(X,K_X\otimes L^{\otimes p}\otimes\scr{I}(\widetilde{h}^p))\geq\frac{p^n}{t_c^n n!}\int_{X_c(\leq1)}\Bigl(\frac{i}{2\pi}\Theta_{L,h_c}\Bigr)^n+o(p^n),
    \end{align*}
    where $\widetilde{h}$ is a singular positive Hermitian metric on $L$ and coincides with $h_c$ on $X_c$.
\end{theorem}

In the non-compact case, the equivalence between big-ness and the positivity of volume is not established. 
However, it follows from singular positivity that the adjoint bundle possesses both properties.

\begin{theorem}$(=\,$\textnormal{Theorem\,\ref{Thm global bigness theorem}\,and\,\ref{Thm global bigness and Vol if Zc is empty}}$)$
    There exists a positive integer $m_0$ such that the adjoint bundle $K_X\otimes L^{\otimes m}$ is singular semi-positive and big for any integer $m\geq m_0$.
    
    Furthermore, if there exists $c>\inf_X\varPsi$ such that $E_{+}(h)\bigcap X_c=\emptyset$, then  
    the adjoint bundle $K_X\otimes L^{\otimes m}$ is big for any integer $m\geq n(n+1)/2$. 
    In addition, if $m\geq n(n-1)/2+1$, then the adjoint bundle $K_X\otimes L^{\otimes m}$ is singular semi-positive, and we have the inequality $\rom{Vol}_X(K_X\otimes L^{\otimes m+1})>0$.
\end{theorem}

This result concerning positivity generalizes Takayama's result \cite{Tak98} to singular Hermitian metrics and generalizes Demailly's characterization of singular positivity and bigness on compact manifolds (see \cite{Dem90}, \cite[Chapter\,6]{Dem12}) to weakly pseudoconvex manifolds, which may be non-compact.

\section{Singular Hermitian metrics on line bundles}\label{Section 2: introduce several notations for sHm}

\subsection{Singular Hermitian metrics and multiplier ideal sheaves}

\begin{definition}
    A real function $\varphi:X\longrightarrow[-\infty,+\infty)$ is said to have \textit{analytic} \textit{singularities} if $\varphi$ has locally the form 
    \begin{align*}
        \varphi=\frac{c}{2}\log\sum_{j\in J}|f_j|^2+\psi,
    \end{align*} 
    where $J$ is at most countable, $f_j$ are non-vanishing holomorphic functions, $\psi$ is a locally bounded function and $c\in\bb{R}_{>0}$.
    Here, the singular support of $\varphi$ and $\idd\varphi$ is an analytic subset.
    If $c\in\bb{Q}_{>0}$, then we furthermore say that $\varphi$ has \textit{algebraic} \textit{singularities}.
\end{definition}

\begin{definition}$($\textnormal{cf.\,\cite{Dem12},\,\cite[Definition\,3.2]{Wat23}}$)$
    A \textit{singular} \textit{Hermitian} \textit{metric} $h$ on a holomorphic line bundle $L\longrightarrow X$ is a metric which is given in any trivialization $\tau:L|_U\stackrel{\simeq}{\longrightarrow}U\times\bb{C}$ by 
    \begin{align*}
        ||\xi||_h=|\tau(\xi)|e^{-\varphi(x)} \quad\text{for any}\,\, x\in U,\,\,\xi\in L_x,
    \end{align*}
    where $\varphi\in L^1_{loc}(U)$ is an arbitrary function, called the weight of the metric with respect to the trivialization $\tau$.
    We say that a singular Hermitian metric $h$ is \textit{singular} \textit{positive} if the weight of $h$ with respect to any trivialization coincide with some strictly plurisubharmonic function almost everywhere. 
    In other words, the curvature current \( \iO{L,h} \) is strictly positive, meaning \( \iO{L,h}\geq\omega_X \) in the sense of currents for some Hermitian metric \( \omega_X>0 \). 
    Moreover, if this coinciding function is merely plurisubharmonic, then $h$ is called \textit{singular} \textit{semi}-\textit{positive}, which coincides with \( \iO{L,h}\geq0 \) in the sense of currents.
\end{definition}

In other words, for a fiexd smooth Hermitian metric $h_0$, the singular Hermitian metric $h$ can be written as $h=h_0 e^{-2\varphi}$ for some locally integrable function $\varphi\in L^1_{loc}(X,\bb{R})$.
Furthermore, it is known that if \( X \) is compact \kah\!, singular positivity (resp. singular semi-positivity) coincides with being big (resp. pseudo-effective) \cite{Dem90,Dem12}.

\begin{definition}
    A function $\varphi:X\longrightarrow[-\infty,+\infty)$ is said to be a \textit{quasi}-\textit{plurisubharmonic} if $\varphi$ is locally the sum of a plurisubharmonic function and a smooth function, or equivalently, if $\idd\varphi$ is bounded below by a continuous real $(1,1)$-form.
    The \textit{multiplier} \textit{ideal} \textit{sheaf} $\scr{I}(\varphi)\subset\cal{O}_X$ is defined by 
    \begin{align*}
        \scr{I}(\varphi)(U):=\{f\in\cal{O}_X(U)\mid |f|^2 e^{-2\varphi}\in L^1_{loc}(U)\}
    \end{align*}
    for any open subset $U\subset X$. For a singular Hermitian metric $h$ on $L$ with the local weight $\varphi$, i.e., $h=h_0e^{-2\varphi}$, we define the multiplier ideal sheaf of $h$ by $\scr{I}(h):=\scr{I}(\varphi)$.
\end{definition}

Similarly, the $L^2$-\textit{subsheaf} $\scr{L}^2(L,h)=:\scr{L}^2(h)$ is defined by $\scr{L}^2(L,h)=L\otimes\scr{I}(h)$.
The multiplier ideal sheaf $\scr{I}(\varphi)$ can sometimes be defined using $e^{-\varphi}$ instead of $e^{-2\varphi}$, but in this paper, due to its compatibility with blow-ups, we define it using $e^{-2\varphi}$.

\subsection{Multivalued holomorphic sections and natural singular Hermitian metrics}\label{subsection: multivalued section and natural sHm}

We introduce a natural singular Hermitian metric defined by sections as follows.
Let $\sigma_0,\ldots,\sigma_N$ be a finite number of holomorphic sections of $L$, then we can define a \textit{natural} \textit{singular} \textit{Hermitian} \textit{metric} of $L$ by 
\begin{align*}
    h_\sigma:=\frac{h_0}{\sum^N_{j=0}|\sigma_j|^2_{h_0}}\xlongequal[]{\!local\!}\frac{1}{\sum^N_{j=0}|\sigma_j|^2},
\end{align*}
where $|\sigma|^2_{h_0}$ is a semi-positive smooth function on $X$ and \( h_\sigma \) is determined independently of the choice of \( h_0 \), thus we simply write it as \( 1/\sum^N_{j=0} |\sigma_j|^2 \).
The associated weight function is given locally by $\varphi=\frac{1}{2}\log(\sum^N_{j=0}|\sigma_j|^2)$,
and is plurisubharmonic. Thus, the metric $h_\sigma$ is singular semi-positive.
Let us denote by $V$ the linear system defined by $\sigma_0,\ldots,\sigma_N$ and by $\rom{Bs}_{|V|}=\bigcap\sigma_j^{-1}(0)$ its base locus.
We have the Kodaira map 
\begin{align*}
    \varPhi_V:X\dashrightarrow\bb{P}^N:x\longmapsto[\sigma_0(x):\cdots:\sigma_N(x)],
\end{align*}
which is holomorphic on $X\setminus \rom{Bs}_{|V|}$. Then the isomorphism $L\cong\varPhi^*_{V}\cal{O}_{\bb{P}^N}(1)$ is obtained on $X\setminus\rom{Bs}_{|V|}$, and 
the curvature current $\frac{i}{2\pi}\Theta_{L,h_\sigma}$ is equal to the pull-back over $X\setminus\rom{Bs}_{|V|}$ of the Fubini-Study metric $\omega_{FS}=\frac{i}{2\pi}\log\sum^N_{j=0}|z_j|^2$ of $\bb{P}^N$ by $\varPhi_V$.

Here, for an effective divisor $D=\sum^J_{j=1}a_jD_j$ with only simple normal crossing and $a_j\in\bb{N}$, let $s_D$ be a canonical section of $\cal{O}_X(D)$. 
Then the natural singular Hermitian metric $1/|s_D|^2$ on $\cal{O}_X(D)$ satisfies $\scr{I}(1/|s_D|^2)=\cal{O}_X(-D)$ and $\scr{L}^2(1/|s_D|^2)=\cal{O}_X$.

For a positive rational number $q$ and some locally finite open cover $\{U_\lambda\}_{\lambda\in\Lambda}$, the family $s=\{s_\lambda\mid s_\lambda^{r_\lambda}\in H^0(U,\cal{O}_X)\,\, \text{for a positive integer}\,\, r_\lambda \}_{\lambda\in\Lambda}$ is said to be a 
\textit{multivalued} \textit{holomorphic} \textit{section} of $L^{\otimes q}$ over $X$, if there exists a positive integer $p$ such that $pq$ begin an integer and that $s^p=\{s_\lambda^p\}_{\lambda\in\Lambda}$ defines an element of $H^0(X,L^{\otimes pq})$.
We denote
\begin{align*}
    |s|^2_{h_0^q}:=&\sqrt[p]{|s^p|^2_{h_0^{pq}}}\quad\text{the pointwise length},\\
    s^{-1}(0):=&\{x\in X\mid s_\lambda(x)=0\,\,\text{for some}\,\,\lambda\in\Lambda\}.
\end{align*}
Let $s_1,\ldots,s_k$ be a finite number of multivalued holomorphic sections of $L^{\otimes q}$ such that $s_j^p$ is holomorphic section of $L^{\otimes pq}$ for some $p\in\bb{N}$ with $pq$ being an integer. 
Then we can define a natural singular Hermitian metric of $L^{\otimes q}$ by 
\begin{align*}
    h_s:=\frac{h^q_0}{\sum^k_{j=1}|s_j|^2_{h^q_0}}
    =\frac{1}{\sum^k_{j=1}|s_j|^2}.
\end{align*}

Let $\scr{J}$ be the sheaf of ideal of $\cal{O}_X$ generated locally by $\{s_j^p\}^k_{j=1}$. 
Let $X_\sharp$ be a relatively compact open subset of $X$ and $\scr{J}^\sharp$ be the restriction of $\scr{J}$ on $X_\sharp$.
Then, there exists a proper modification $\pi:\widetilde{X}_\sharp\longrightarrow X_\sharp$ by a finite number of blow-ups with non-singular centers and 
a finite number of smooth exceptional divisors $E_\mu$ in $\widetilde{X}_\sharp$ with only simple normal crossing, satisfying the following conditions:
\begin{itemize}
    \item The sheaf $\pi^{-1}\scr{J}^\sharp\cdot\cal{O}_{\widetilde{X}_\sharp}=\rom{im}(\pi^*\scr{J}^\sharp\longrightarrow\cal{O}_{\widetilde{X}_\sharp})$ is equal to the ideal sheaf $\cal{O}_{\widetilde{X}_\sharp}(-\sum^J_{\mu=1}\xi^\sharp_\mu E_\mu)$ for some non-negative integers $\xi_\mu\in\bb{N}\cup\{0\}$.
    \item $K_{\widetilde{X}_\sharp}=\pi^*K_{X_\sharp}\otimes\cal{O}_{\widetilde{X}_\sharp}(\sum^J_{\mu=1}\zeta_\mu E_\mu)$ for some non-negative integers $\zeta_\mu\in\bb{N}\cup\{0\}$.
\end{itemize}
Let $\xi_\mu:=\xi_\mu^\sharp/p$. For any $t\geq0$, we set the multiplier ideal sheaf 
\begin{align*}
    \scr{I}(t):=\scr{I}(h_s^t)=\scr{I}\Bigl(\bigl(\sum^k_{j=1}|s_j|^2\bigr)^{-t}\Bigr).
\end{align*}
Then, we obtain $x\in V(\scr{I}(t))|_{X_\sharp}$ if and only if there exists an index $\mu$ such that $E_\mu$ intersects $\pi^{-1}(x)$ and that $t\xi_\mu-\zeta_\mu\geq1$.
It is known that
\begin{align*}
    \scr{I}(t)=\pi_*\cal{O}_{\widetilde{X}_\sharp}\Bigl(\sum^J_{\mu=1}(\zeta_\mu-\lfloor t\xi_\mu\rfloor)E_\mu\Bigr),
\end{align*}
where $\lfloor\bullet\rfloor$ denotes the integer part.
For any point $x\in V(\scr{I}(1))|_{X_\sharp}$, we see that 
\begin{align*}
    \sup&\{t\geq0\mid\scr{I}(t)_x=\cal{O}_{X,x}\}\\
    &=\min\{t\geq0\mid t\xi_\mu-\zeta_\mu\geq1\,\,\text{for}\,\,\mu\,\,\text{such that}\,\,E_\mu\,\,\text{intersects}\,\,\pi^{-1}(x)\}.
\end{align*}
This quantity, denoted by $\alpha(x)$, is always a rational number, and satisfies $0<\alpha(x)\leq 1$.

\subsection{Bigness and ampleness on general manifolds}

Let $X$ be a complex manifold which is not necessarily compact and $V$ be a finite dimensional linear subspace of $H^0(X,L)$.
We define the linear subsystem \(|V|\) corresponding to \(V\) by \(|V|=\bb{P}(V)\).
The base locus of the linear subsystem $|V|$ is given by 
\begin{align*}
    \rom{Bs}_{|V|}=\bigcap_{s\in V}\rom{div}\,s=\{x\in X\mid s(x)=0 \,\,\text{for all}\,\,s\in V\setminus\{0\}\}.
\end{align*}

If $X$ is compact, then the linear subsystem $|V|$ is naturally identified with $|V|=\{\rom{div}\,s\mid s\in V\setminus\{0\}\}$. 
In fact, due to the compactness of $X$, i.e., $H^0(X,\cal{O}^*_X)=\bb{C}^*$, it follows that $\rom{div}\,s=\rom{div}\,s'$ if and only if $s=\lambda s'$ for some $\lambda\in\bb{C}^*$. 
However, if $X$ is general, note that the projective space \(\bb{P}(V) = |V|\) is parametrized by $V^*/(H^0(X,\cal{O}^*_X)\setminus\bb{C}^*)$.
In fact, $\rom{div}\,s=\rom{div}\,s'$ if and only if $s=fs'$ for some $f\in H^0(X,\cal{O}_X^*)$.

Regarding ampleness, the following definition is well-known (see \cite{Tak98}).
A holomorphic line bundle $L$ on $X$ is \textit{very} \textit{ample}, if there exists a holomorphic embedding $\varPhi:X\longrightarrow\bb{P}^N$ such that $L\cong\varPhi^*\cal{O}_{\bb{P}^N}(1)$. 
$L$ is \textit{ample} if there exists a positive integer $m_X$ such that $L^{\otimes m_X}$ is very ample.

Regarding bigness, we introduce a definition involving the multiplier ideal sheaf.

\begin{definition}\label{Def of bigness}
    Let $X$ be a complex manifold (not necessarily compact) and $L$ be a holomorphic line bundle on $X$ equipped with a singular Hermitian metric $h$.
    Set 
    \begin{align*}
        \varrho_m^h:=\sup\{\rom{dim}\,\overline{\varPhi_V(X)}\mid V\subseteq H^0(X,L^{\otimes m}\otimes\scr{I}(h^m))\,\, \text{with}\,\, \rom{dim}\,V<+\infty\}
    \end{align*}
    if $H^0(X,L^{\otimes m}\otimes\scr{I}(h^m))\ne\{0\}$, and $\varrho_m^h=-\infty$ otherwise. 
    The \textit{Kodaira}-\textit{Iitaka} \textit{dimension} of $L\otimes\scr{I}(h)$ is defined by 
    \begin{align*}
        \kappa(L\otimes\scr{I}(h)):=\max\{\varrho_m^h\mid m\in\bb{N}\}=\max_{m\geq1}\Biggl\{\sup_{\substack{V\subseteq H^0(X,L^{\otimes m}\otimes\scr{I}(h^m)) \\ \rom{dim}\,V<+\infty}}\rom{dim}\,\overline{\varPhi_V(X)}\Biggr\}.
    \end{align*}
    The $L^2$-subsheaf $\scr{L}^2(h)=L\otimes\scr{I}(h)$ is called \textit{big} if $\kappa(L\otimes\scr{I}(h))=\rom{dim}\,X$. 

    Furthermore, considering $h$ to be smooth, \( \kappa(L) \) is also defined in a similar manner, and $L$ is called \textit{big} if $\kappa(L)=\rom{dim}\,X$.
\end{definition}

Clearly, if \( L\otimes\scr{I}(h) \) is big, then \( L \) is also big. 
This definition, as can be seen from the following observation, is a natural generalization.

\begin{proposition}\label{Prop of bigness on proj mfd}
    Let $X$ be a projective manifold and $L$ be a big line bundle on $X$. 
    There exist a singular positive Hermitian metric $h$ on $L$ and a large integer $p\in\bb{N}$ such that $L\otimes\scr{I}(h)$ is also big and $L^{\otimes p}\otimes\scr{I}(h^p)$ is very ample line bundle.
\end{proposition}

In the case of this proposition, it is appropriate to say that the \( L^2 \)-subsheaf \(\scr{L}^2(h) = L \otimes \scr{I}(h)\) is ample, here \(\scr{L}^2(h^p) = L^{\otimes p} \otimes \scr{I}(h^p)\).

\begin{proof}
    By Kodaira's lemma, there exist an ample line bundle $A$, an effective divisor $E$ and a large integer $p$ such that $L^{\otimes p}=A\otimes\cal{O}_X(E)$.
    Here, by choosing $p$ sufficiently large, we can assume that $A$ is very ample.
    Let $s_E$ be a canonical section of $\cal{O}_X(E)$ and $h_0$ be a smooth Hermitian metric on $L$. 
    From very ampleness of $A$, there exist $\sigma_0,\ldots,\sigma_N\in H^0(X,A)$ such that the Kodaira map $\varPhi_\sigma:X\longrightarrow\bb{P}^N$ induced by sections $\sigma_j$ becomes an embedding.
    Here, each $\tau_j:=\sigma_j^{1/p}\otimes s_E^{1/p}$ is multivalued holomorphic section of $L$.
    We define a natural singular Hermitian metric of $L$ by
    \begin{align*}
        h:=\frac{h_0}{\sum^N_{j=0}|\tau_j|^2_{h_0}}=\frac{1}{\sum^N_{j=0}|\sigma_j^{1/p}|^2}\cdot\frac{1}{|s_E^{1/p}|^2}=:h_\sigma\otimes h_s,
    \end{align*} 
    where $h_\sigma$ and $h_s$ are also natural singular Hermitian metrics on $A^{1/p}$ and $\cal{O}_X(E)^{1/p}$. 
    The curvature $\frac{i}{2\pi}\Theta_{A,h_\sigma^p}=\varPhi_\sigma^*\omega_{FS}$ is Hodge.
    By the inequality 
    \begin{align*}
        \frac{i}{2\pi}\Theta_{L,h}=\frac{1}{p}\Bigl(\frac{i}{2\pi}\Theta_{A,h_\sigma^p}+[E]\Bigr)\geq\frac{1}{p}\omega, \qquad \omega=\frac{i}{2\pi}\Theta_{A,h_\sigma^p},
    \end{align*}
    where $[E]$ is the closed positive current induced by integration over $E$, it follows that $h$ is singular positive.
    Since $h_\sigma$ is smooth, we obtain $\scr{I}(h^p)=\scr{I}(1/|s_E|^2)=\cal{O}_X(-E)$ and $L^{\otimes p}\otimes\scr{I}(h^p)=A$, here $h_s^p=1/|s_E|^2$.
\end{proof}

\section{Approximation, blow-ups and construction of canonical metrics}\label{Section 3: Demailly's approximate and blow-ups}

In this section, we refine Demailly's approximation to provide an approximation that ideal sheaves are preserved and compatible with blow-ups. Using this blow-ups 
we construct a canonical singular Hermitian metric.

\subsection{Demailly-type approximation theorems}\label{subsection: Demailly-type approximation}

We know the following Demailly's approximation as an extremely effective method.

\begin{theorem}[{=\,\cite[Theorem\,2.3]{DPS01},\,cf.\,\cite[Theorem\,2.9]{Mat22}}]\label{Dem appro preserves ideal sheaves}
    Let $X$ be a complex manifold equipped with a Hermitian metric $\omega$ and $T=\alpha+\idd\varphi$ be a closed $(1,1)$-current on $X$ 
    where $\alpha$ is a smooth closed $(1,1)$-form and $\varphi$ is a quasi-plurisubharmonic function. Assume that $T=\alpha+\idd\varphi\geq \gamma$ holds for a continuous real $(1,1)$-from $\gamma$ on $X$. 
    Then, for a relatively compact set $K\Subset X$, there exists a sequence of quasi-plurisubharmonic functions $\{\varphi_\nu\}_{\nu\in\bb{N}}$ on $K$ with the following properties:
    \begin{itemize}
        \item [($a$)] $\varphi_\nu$ is smooth on $K\setminus Z_\nu$, where $Z_\nu$ is an analytic subset of $K$ with $Z_\nu\subset Z_{\nu+1}$,
        \item [($b$)] $\{\varphi_\nu\}_{\nu\in\bb{N}}$ is a decreasing sequence of functions converging to $\varphi$,
        \item [($c$)] $\scr{I}(\varphi)=\scr{I}(\varphi_\nu)$ on $K$ for all $\nu\in\bb{N}$,
        \item [($d$)] $T_\nu=\alpha+\idd\varphi_\nu$ satisfies $T_\nu\geq\gamma-\varepsilon_\nu\omega$, where $\lim_{\nu\to+\infty}\varepsilon_\nu=0$.
    \end{itemize}
\end{theorem}

There are two types of Demailly approximations: one involves approximations with algebraic singularities or whose singular loci are Lelong number upper-level sets (see \cite[Theorem\,1.1\,and\,Proposition\,3.7]{Dem92}, \cite[Theorem\,13.12]{Dem12}), 
while the other preserves ideal sheaves as described above. 
The former is suitable for using blow-ups, while the latter works effectively due to its decreasing property when dealing with $L^2$-estimates.
We provide the following approximation that ideal sheaves are preserved and ensures compatible with blow-ups, i.e., having algebraic singularities, sacrificing both singular loci begin Lelong number upper-level sets and the decreasing property.

\begin{theorem}\label{Dem appro with log poles and ideal sheaves}
    Let $X$ be a complex manifold equipped with a Hermitian metric $\omega$ and $T=\alpha+\idd\varphi$ be a closed $(1,1)$-current on $X$ 
    where $\alpha$ is a smooth closed $(1,1)$-form and $\varphi$ is a quasi-plurisubharmonic function. Assume that $T=\alpha+\idd\varphi\geq \gamma$ holds for a continuous real $(1,1)$-form $\gamma$ on $X$. 
    Then, for a relatively compact subset $K\Subset X$, there exist an increasing sequence of positive integers $\{m_\nu\}_{\nu\in\mathbb{N}}$ and a sequence of quasi-plurisubharmonic functions $\{\varphi_{m_\nu}\}_{\nu\in\bb{N}}$ on $K$ such that the following are satisfied.
    \begin{itemize}
        \item [$(a)$] $\varphi_{m_\nu}$ has algebraic singularities. 
        That is, locally $\varphi_{m_\nu}$ can be expressed as
            \begin{align*}
                \varphi_{m_\nu}=\frac{1+2^{-\nu}}{2m_\nu}\log\sum_j|\sigma_{m_\nu,j}|^2+\varPhi_{m_\nu},
            \end{align*}
            where $\{\sigma_{m_\nu,j}\}_{j\in\bb{N}}$ is an orthonormal basis of the Hilbert space $\cal{H}(m_\nu(\varphi+c|z|^2))$ for some real number $c$ and local coordinates $z$ and $\varPhi_{m_\nu}$ is smooth. 
        \item [$(b)$] $\varphi_{m_\nu}$ is smooth on $K\setminus Z_\nu$, where the set $Z_\nu$ of logarithmic poles of $\varphi_{m_\nu}$ is an analytic subset of $K$, satisfying $Z_\nu\subset Z_{\nu+1}$, and is locally obtained by $\bigcap_j \sigma_{m_\nu,j}^{-1}(0)$. 
        \item [$(c)$] there exists a large $\nu_0$ such that $\scr{I}(\varphi)=\scr{I}(\varphi_{m_\nu})$ on $K$ for any $\nu\geq \nu_0$. 
        \item [$(d)$] $T_{m_\nu}:=\alpha+\idd\varphi_{m_\nu}$ satisfies $T_{m_\nu}\geq\gamma-\varepsilon_\nu\omega$, where $\{\varepsilon_\nu\}_{\nu_0\leq\nu\in\bb{N}}$ is a decreasing sequence of positive numbers with $\lim_{\nu\to+\infty}\varepsilon_\nu=0$.
        \item [$(e)$] we obtain the following inequality related to Lelong numbers 
        \begin{align*}
            \nu(\varphi,x)-\frac{n}{m_\nu}\leq\nu\Big(\frac{\varphi_{m_\nu}}{1+2^{-\nu}},x\Big)\leq\nu(\varphi,x) \quad\text{ for any } x\in K.
        \end{align*}
    \end{itemize}

    In particular, for any $t>0$ there exists an enough large integer $\nu(t)\in\bb{N}$ such that $\scr{I}(t\varphi)=\scr{I}(t\varphi_{m_\nu})$ for any $\nu\geq\nu(t)$.
\end{theorem}

\begin{proof}
    Clearly, we may assume that $\alpha=0$ and $T=\idd\varphi\geq\gamma$ by replacing $T$ with $T-\alpha$ and $\gamma$ with $\gamma-\alpha$.
    The construction method of the sequence of functions is almost entirely derived from the proof of Theorem \ref{Dem appro preserves ideal sheaves}.
    (From the proof of \cite[Theorem\,2.9]{Mat22}, we can assume that \( K \) is compact.) 
    As in Step 1 of the proof of \cite[Theorem\,2.2.1]{DPS01}, for a given \( \varepsilon > 0 \), we fix an open covering of $K$ by open balls \( B_j=\{|z^{(j)}|<r_j\} \) with coordinates $z^{(j)}=(z^{(j)}_1,\ldots,z^{(j)}_n)$ such that 
    \begin{align*}
        0\leq\gamma+c_j\idd|z^{(j)}|^2\leq\varepsilon\omega \qquad\rom{on} \quad B_j,
    \end{align*}
    for some real number $c_j$.
    Then, using the local Demailly's approximation on each open subset \( B_j \), we get a sequence of functions \( \{\psi_{\varepsilon,\nu,j}\}_{\nu\in\bb{N}} \) with logarithmic poles defined by 
    \begin{align*}
        \psi_{\varepsilon,\nu,j}=\frac{1}{2\nu}\log\sum_{k\in\bb{N}}|\sigma_{\nu,j,k}|^2-c_j|z^{(j)}|^2 \qquad\rom{on}\quad B_j,
    \end{align*}
    where $\{\sigma_{\nu,j,k}\}_{k\in\bb{N}}$ is an orthonormal basis of the Hilbert space $\cal{H}_{\nu,j}:=\cal{H}_{B_j}(\nu(\varphi+c_j|z^{(j)}|^2))$ of holomorphic functions on $B_j$ with finite $L^2$-norm  
    \begin{align*}
        ||u||^2_{\nu,j}=\int_{B_j}|u|^2e^{-2\nu(\varphi+c_j|z^{(j)}|^2)}d\lambda_{z^{(j)}}.
    \end{align*}

    We already know that there exist $C_l,\,l=1,2,\ldots, $ depending only on $K$ and $\gamma$, such that the following uniform estimates hold: 
    \begin{align*}
        \idd\psi_{\varepsilon,\nu,j}&\geq-c_j\idd|z^{(j)}|^2\geq\gamma-\varepsilon\omega  &\rom{on}\quad B'_j (\Subset B_j),\\
        \varphi(z)\leq\psi_{\varepsilon,\nu,j}(z)&\leq\sup_{|\zeta-z|\leq r}\varphi(\zeta)+\frac{n}{\nu}\log\frac{C_1}{r}+C_2r^2  &\forall\,z\in B'_j,\quad r<r_j-r'_j,\\
        |\psi_{\varepsilon,\nu,j}-\psi_{\varepsilon,\nu,k}|&\leq\frac{C_3}{\nu}+C_4\varepsilon\bigl(\min\{r_j,r_k\}\bigr)^2  &\rom{on}\quad B'_j\cap B'_k, 
    \end{align*}
    here $B''_j=\{|z^{(j)}|<r''_j:=r_j/4\}\Subset B'_j=\{|z^{(j)}|<r'_j:=r_j/2\}$ and we may also assume that $K\subset\bigcup_jB''_j$. 
    From the final inequality, it follows that $\psi_{\varepsilon,\nu,j}$ and $\psi_{\varepsilon,\nu,k}$ have the same logarithmic poles.
    
    Construct the globally defined function \( \psi_{\varepsilon, \nu} \) on \( K \) by gluing together all functions \( \psi_{\varepsilon, \nu, j} \) with respect to \( j \), as in the proof of \cite[Theorem\,2.2.1]{DPS01}.
    \begin{align*}
        \psi_{\varepsilon,\nu}(z)=\sup_{j,\,z\in B'_j}\Bigl(\psi_{\varepsilon,\nu,j}(z)+12 C_4\varepsilon(r'^2_j-|z^{(j)}|^2)\Bigr) \qquad\rom{on} \quad K.
    \end{align*}
    We already know that $\psi_{\varepsilon,\nu}$ is well-defined and continuous, 
    \begin{align*}
        \idd\psi_{\varepsilon,\nu}\geq\gamma-C_5\varepsilon\omega
    \end{align*}
    for $\nu\geq\nu'(\varepsilon)$ large enough and that $\lim_{\nu\to+\infty}\psi_{\varepsilon,\nu}(z)=\varphi(z)$. 
    In particular, the quasi-plurisubharmonic function \( \psi_{\varepsilon,\nu} \) locally possesses logarithmic poles that can be expressed similarly to \( \psi_{\varepsilon,\nu,j} \).
    Furthermore, from the above inequality related to the Lelong number of \( \psi_{\varepsilon,\nu,j} \), we obtain the Lelong numbers inequality
    \begin{align*}
        \nu(\varphi,x)-\frac{n}{\nu}\leq\nu(\psi_{\varepsilon,\nu},x)\leq\nu(\varphi,x) \quad\text{ for any }\quad x\in \bigcup_j B'_j.
    \end{align*}
    
    Here, $\psi_{\varepsilon,\nu}(z)\geq\psi_{\varepsilon,\nu,j}(z)$ for some $j$ dependent on $z$, by $12 C_4\varepsilon(r'^2_j-|z^{(j)}|^2)>0$.
    From this and $\varphi(z)\leq\psi_{\varepsilon,\nu,j}(z)$ on $B'_j$, we have $\varphi\leq\psi_{\varepsilon,\nu}$ on $K$ and
    \begin{align*}
        \scr{I}(\varphi)\subseteq\scr{I}(\psi_{\varepsilon,\nu}) \qquad\rom{for\,\,any}\quad \varepsilon>0\quad\rom{and}\quad \nu\in\bb{N}.
    \end{align*}

    From Steps 2 and 3 of the proof of \cite[Theorem\,2.2.1]{DPS01}, we can choose an index \( \nu=p(k) \) such that $\lim_{k\to+\infty}p(k)=+\infty$ and 
    \begin{align*}
        \int_K\bigl(e^{-2\varphi}-e^{-2\max\{\varphi,\,(1+2^{-k})\psi_{2^{-k},2^{p(k)}}\}}\bigr)dV_\omega\leq 2^{-k}.
    \end{align*}
    By setting $\widetilde{\varphi}_\nu(z)=\sup_{k\geq\nu}(1+2^{-k})\psi_{2^{-k},2^{p(k)}}(z)$, we obtain an approximate sequence as given in Theorem \ref{Dem appro preserves ideal sheaves}, and from the construction, \( \{\widetilde{\varphi}_\nu\}_{\nu\in\bb{N}} \) is a decreasing sequence and satisfies the estimates
    \begin{align*}
        \widetilde{\varphi}_\nu\geq\max\{\varphi, (1+2^{-\nu})\psi_{2^{-\nu},2^{p(\nu)}}\}, \qquad \idd\widetilde{\varphi}_\nu\geq\gamma-C_52^{-\nu}\omega.
    \end{align*}
    Let $Z_\nu$ be the set of logarithmic poles of $\psi_{2^{-\nu},2^{p(\nu)}}$, then $Z_\nu\subset Z_{\nu+1}$ and $\widetilde{\varphi}_\nu$ is continuous on $K\setminus Z_\nu$. 
    Here, $\scr{I}(\varphi)=\scr{I}(\widetilde{\varphi}_\nu)$ on $K$ for all $\nu\in\bb{N}$ by \cite[Theorem\,2.2.1]{DPS01}.

    Setting $m_k:=2^{p(k)}$, we define the desired quasi-plurisubharmonic function by 
    \begin{align*}
        \varphi_{m_k}:=(1+2^{-k})\psi_{2^{-k},2^{p(k)}}=(1+2^{-k})\psi_{2^{-k},m_k}.
    \end{align*}
    From the inequality \( \widetilde{\varphi}_\nu \geq(1+2^{-\nu})\psi_{2^{-\nu},m_\nu}=\varphi_{m_\nu} \), we obtain the inclusion relation
    \begin{align*}
        \scr{I}((1+2^{-\nu})\varphi)\subseteq\scr{I}((1+2^{-\nu})\psi_{2^{-\nu},m_\nu})=\scr{I}(\varphi_{m_\nu})\subseteq\scr{I}(\widetilde{\varphi}_\nu)=\scr{I}(\varphi)
    \end{align*}
    on $K$ concerning multiplier ideal sheaves. 

    By the strong openness property $\scr{I}(\varphi)=\bigcup_{\delta>0}\scr{I}((1+\delta)\varphi)$ (see \cite{GZ15}) and the strong Noetherian property of coherent sheaves (see \cite[ChapterII,\,(3.22)]{Dem-book}), there exists an enough large $\nu_0\gg1$ such that $\scr{I}((1+2^{-\nu_0})\varphi)=\scr{I}(\varphi)$ on $K$, then we have $\scr{I}(\varphi_{m_\nu})=\scr{I}(\varphi)$ on $K$ if $\nu\geq \nu_0$. 
    In a similar manner by using Remark \ref{Remark of ideal for Demailly approximation} below, for any $t>0$ there exists an enough large $\nu(t)\in\bb{N}$ such that the equation \( \scr{I}(t\varphi)=\scr{I}(t\varphi_{m_\nu}) \) holds for any $\nu\geq\nu(t)$.
    This theorem is proved, except that $\varphi_{m_\nu}$ is possibly just continuous instead of being smooth. This can be arranged by the following Lemma \ref{Lemma of Richberg's regularization theorem} which is the Richberg's regularization theorem (see \cite{Ric68}), at the expense of an arbitrary small loss in the Hessian form. 

    By the definition of $\psi_{\varepsilon,\nu}$, for any $z\in K$, there exists $j$ such that
    \begin{align*}
        \psi_{\varepsilon,\nu}(z)=\frac{1}{2\nu}\log\sum_{k\in\bb{N}}|\sigma_{\nu,j,k}|^2-c_j|z^{(j)}|^2+12C_4\varepsilon(r'^2_j-|z^{(j)}|^2).
    \end{align*}
    Therefore, there exist continuous functions \( G_{\varepsilon,\nu} \) and \( \varPsi_{\varepsilon,\nu} \) on $K$ such that \( \psi_{\varepsilon,\nu} \) can be written as \( \psi_{\varepsilon,\nu}=\frac{1}{2\nu}\log G_{\varepsilon,\nu}+\varPsi_{\varepsilon,\nu} \) on $K$.
    Clearly, \( \varPsi_{\varepsilon,\nu} \) is a quasi-plurisubharmonic function, and by Richberg's regularization theorem (= Lemma\,\ref{Lemma of Richberg's regularization theorem}), for a sufficiently small neighborhood \( U_{\varepsilon,\nu} \) of \( Z_{\varepsilon,\nu} \), which is the set of logarithmic poles of $\psi_{\varepsilon,\nu}$, and any continuous function $\lambda>0$ on $U_{\varepsilon,\nu}$, 
    there exists a smooth function $\widetilde{\varPsi}_{\varepsilon,\nu}$ on $U_{\varepsilon,\nu}$ that satisfies $\varPsi_{\varepsilon,\nu}<\widetilde{\varPsi}_{\varepsilon,\nu}<\varPsi_{\varepsilon,\nu}+\lambda$ with an arbitrarily small loss in the Hessian form.
    Finally, applying Richberg's regularization theorem to $\psi_{\varepsilon,\nu}$ on $K\setminus Z_{\varepsilon,\nu}$ suffices. 
    Since $\lambda$ can be chosen arbitrarily, by the proof of Richberg's regularization theorem (see \cite[Sketch of proof of Lemma\,2.15]{Dem92}), it is possible to regularize $\psi_{\varepsilon,\nu}$ on $K\setminus Z_{\varepsilon,\nu}$ including $\widetilde{\varPsi}_{\varepsilon,\nu}$, such that there exists a regularization $\widetilde{\psi}_{\varepsilon,\nu}$ which locally can be written as 
    \begin{align*}
        \widetilde{\psi}_{\varepsilon,\nu}=\frac{1}{2\nu}\log\sum_k|\sigma_{\nu,k}|^2+\varPhi_{\varepsilon,\nu},
    \end{align*}
    where $\varPhi_{\varepsilon,\nu}$ is smooth even on $Z_{\varepsilon,\nu}$. Here, we have $\widetilde{\psi}_{\varepsilon,\nu}\geq\psi_{\varepsilon,\nu}$ on $K$.
    
    We can again choose an index $\nu=p(k)$ that satisfies
    \begin{align*}
        \int_K\bigl(e^{-2\varphi}-e^{-2\max\{\varphi,\,(1+2^{-k})\widetilde{\psi}_{2^{-k},2^{p(k)}}\}}\bigr)dV_\omega\leq 2^{-k},
    \end{align*}
    and by redefining $\varphi_{m_\nu}:=(1+2^{-\nu})\widetilde{\psi}_{2^{-\nu},m_\nu}$ and $\varPhi_{m_\nu}:=\varPhi_{2^{-\nu},m_\nu}$, the proof is complete. 
    In particular, by similarly defining $\widetilde{\varphi}_\nu$ as the regularization of $\varphi_\nu:=\sup_{k\geq\nu}(1+2^{-k})\widetilde{\psi}_{2^{-k},m_k}=\sup_{k\geq\nu}\varphi_{m_k}$, 
    we obtain $\widetilde{\varphi}_\nu\geq\varphi_\nu\geq\varphi_{m_\nu}\geq(1+2^{-\nu})\psi_{2^{-\nu},m_\nu}\geq(1+2^{-\nu})\varphi$ on $K$ and the same inclusion relationship for multiplier ideal sheaves as above.
\end{proof}

\begin{lemma}[{cf.\,\cite{Ric68},\,\cite[Lemma\,2.15]{Dem92}}]\label{Lemma of Richberg's regularization theorem}
    Let $\psi$ be a quasi-plurisubharmonic function on a complex manifold $M$ such that $\idd\psi\geq\gamma$ where $\gamma$ is a continuous $(1,1)$-form. 
    For any Hermitian metric $\omega$ and any continuous function $\lambda>0$ on $M$, there is a smooth function $\widetilde{\psi}$ such that $\psi<\widetilde{\psi}<\psi+\lambda$ and $\idd\widetilde{\psi}\geq\gamma-\lambda\omega$ on $M$.
\end{lemma}

\begin{remark}$($\textnormal{cf.\,\cite[Remark\,2.2.13]{DPS01}}$)$\label{Remark of ideal for Demailly approximation}
    Regarding the function \( \widetilde{\varphi}_\nu \) within the proof, for any \( t > 0 \), there exists an enough large integer \( \nu(t)\in\bb{N} \) such that \( \scr{I}(t\varphi)=\scr{I}(t\widetilde{\varphi}_\nu) \) holds for any $\nu\geq\nu(t)$.
\end{remark}

\subsection{Blow-ups for the approximation}\label{subsection: blow-ups thm}

In this subsection, we provide the appropriate blow-ups and related properties for the approximation in Theorem \ref{Dem appro with log poles and ideal sheaves}.

\begin{theorem}\label{Blow ups of Dem appro with log poles and ideal sheaves}
    Let $X$ be a complex manifold equipped with a Hermitian metric $\omega$ and $L$ be a holomorphic line bundle with a singular Hermitian metric $h$. 
    If $\iO{L,h}\geq\gamma$ for a continuous real $(1,1)$-form $\gamma$ on $X$, then for a relatively compact open subset $V\Subset X$, there exist 
    a sequence of singular Hermitian metrics $\{h_\nu\}_{\nu\in\bb{N}}$ on $L|_V$ with algebraic singularities and proper modifications $\pi_\nu:\widetilde{V}\longrightarrow V$, given by a composition of finitely many blow-ups with smooth center, 
    such that the following conditions are satisfied.
    \begin{itemize}
        \item [($a$)] there exists a large $\nu_0$ such that $\scr{I}(h)=\scr{I}(h_\nu)$ on $V$ for any $\nu\geq\nu_0$.
        \item [($b$)] $\iO{L,h_\nu}\geq\gamma-\varepsilon_\nu\omega$ on $V$, where $\{\varepsilon_\nu\}_{\nu\in\bb{N}}$ is a decreasing sequence of positive numbers with $\lim_{\nu\to+\infty}\varepsilon_\nu=0$.
        \item [($c$)] there exists an increasing sequence of positive integers $\{m_\nu\}_{\nu\in\bb{N}}$ such that the local weight of $h_{\nu}$ has algebraic singularities. That is, locally the weight can be expressed in the following form
        \begin{align*}
            \varphi_\nu=\frac{1+2^{-\nu}}{2m_\nu}\log\sum_{k\in\bb{N}}|\sigma_{\nu,k}|^2+\varPhi_\nu,
        \end{align*}
        where $\{\sigma_{\nu,k}\}_{k\in\bb{N}}$ is an orthonormal basis of a Hilbert space 
        and $\varPhi_\nu$ is smooth. 
        \item [($d$)] $h_\nu$ is smooth on $V\setminus Z_\nu$, where the set $Z_\nu$ of logarithmic poles of $\varphi_\nu$ is an analytic subset of $V$ satisfying $Z_\nu\subset Z_{\nu+1}$. 
        \item [($e$)] $\pi_\nu:\widetilde{V}\setminus\pi_\nu^{-1}(Z_\nu)\longrightarrow V\setminus Z_\nu$ is biholomorphic. 
        \item [($f$)] there exists a simple normal crossing divisor $D_\nu=\sum^J_{j=1}a_jD_j$ with positive integers $a_j\in\bb{N}$ and $\rom{supp}\,D_\nu=\pi^{-1}_\nu(Z_\nu)$ such that 
        \begin{align*}
            \qquad\pi_\nu^*\varphi_\nu=\frac{1+2^\nu}{2^\nu m_\nu}\sum^J_{j=1}a_j\log|g_j|+\widetilde{\varPhi}_\nu \quad \rom{and} \quad \scr{I}(\pi^*_\nu h_\nu^p)=\cal{O}_{\widetilde{V}}\Bigl(-\sum^J_{j=1}\lfloor\frac{(1+2^\nu)a_jp}{2^\nu m_\nu}\rfloor D_j\Bigr)
        \end{align*}
        for any $p\in\bb{N}$, where $g_j$ are local generators of $\cal{O}_{\widetilde{V}}(D_j)$ and $\lfloor\bullet\rfloor$ denotes the integer part. Moreover, $\widetilde{\varPhi}_\nu$ is smooth and can be expressed as 
        \begin{align*}
            \widetilde{\varPhi}_\nu=\frac{1+2^\nu}{2^{\nu+1} m_\nu}\log\sum_{k\in\bb{N}}|\tau_{\nu,k}|^2+\pi_\nu^*\varPhi_\nu,
        \end{align*}
        where $\{\tau_{\nu,k}\}_{k\in\bb{N}}$ has no common zeros.
        \item [($g$)] the holomorphic line bundle $\widetilde{L}_{m_\nu}\!:=\!\pi_\nu^*L^{\otimes2^\nu m_\nu}\otimes\cal{O}_{\widetilde{V}}(-(1+2^\nu)D_\nu)\!=\!\scr{L}^2(\pi_\nu^*h_\nu^{2^\nu m_\nu})$ has a smooth Hermitian metric $\widetilde{h}_{m_\nu}$ 
        satisfying $\iO{\widetilde{L}_{m_\nu},\widetilde{h}_{m_\nu}}\!\geq2^\nu m_\nu\pi^*_\nu(\gamma-\varepsilon_\nu\omega)$. 
    \end{itemize}
    In particular, for any $k\in\bb{N}$, there exists an enough large integer $\nu(k)\in\bb{N}$ such that $\scr{I}(h^k)=\scr{I}(h_\nu^k)$ for any $\nu\geq\nu(k)$.
    Furthermore, we have the following
    \begin{itemize}
        \item [($\alpha$)] assume that $h$ is singular positive, i.e., $\gamma>0$, and fix $\nu$ with $\gamma-\varepsilon_\nu\omega>0$. 
        Then there exist an effective divisor $D_b=\sum^J_{j=1}b_jD_j$ with $b_j\in\bb{N}$, a smooth Hermitian metric $h^*_b$ on the corresponding line bundle $\cal{O}_{\widetilde{V}}(-D_b)$  
        and a positive integer $t_V\gg0$ such that for any integer $t\geq t_V$, the smooth Hermitian metric $\widetilde{h}_{m_\nu}^{\otimes t}\otimes h^*_b$ on $\scr{L}:=\widetilde{L}_{m_\nu}^{\otimes t}\otimes\cal{O}_{\widetilde{V}}(-D_b)$ is positive on $\widetilde{V}$. 
        \item [($\beta$)] for any holomorphic vector bundle $E$ on $V$ and for any integers $\displaystyle p\geq \frac{2^\nu m_\nu}{1+2^\nu}$ and $q\geq0$, we have the following cohomology isomorphism
        \begin{align*}
            H^q(V,K_V\otimes E\otimes L^{\otimes p}\otimes\scr{I}(h_\nu^p))\cong H^q(\widetilde{V},K_{\widetilde{V}}\otimes\pi^*_\nu (E\otimes L^{\otimes p})\otimes\scr{I}(\pi^*_\nu h_\nu^p)).
        \end{align*}
    \end{itemize}
\end{theorem}

\begin{proof}
    For a smooth Hermitian metric $h_0$ on $L$, there exists a locally integrable function $\varphi\in L^1_{loc}(X,\bb{R})$ such that $h$ can be written as $h=h_0e^{-2\varphi}$ on $X$ and $\varphi$ is quasi-plurisubharmonic by $\iO{L,h}\geq\gamma$.
    For this $\varphi$, from Theorem \ref{Dem appro with log poles and ideal sheaves}, there exist an increasing sequence of positive integers $\{m_\nu\}_{\nu\in\bb{N}}$ and a sequence of quasi-plurisubharmonic functions $\{\varphi_{m_\nu}\}_{\nu\in\bb{N}}$ that satisfy conditions $(a)$-$(d)$ of Theorem \ref{Dem appro with log poles and ideal sheaves}. 
    Setting $\varphi_\nu:=\varphi_{m_\nu}$ and constructing a sequence of singular Hermitian metrics $\{h_\nu\}_{\nu\in\bb{N}}$ by $h_\nu:=h_0e^{-2\varphi_\nu}$, it satisfies the conditions $(a)$-$(d)$ of this theorem.

    We introduce the ideal $\scr{J}_\nu$ of germs of holomorphic function $f$ such that $|f|\leq C \exp{\frac{m_\nu\varphi_\nu}{1+2^{-\nu}}}$ for some constant $C$. 
    This is a globally defined ideal sheaf. Under the notation of the proof of Theorem \ref{Dem appro with log poles and ideal sheaves},
    from the strong Noetherian property of coherent ideal sheaves (see \cite[ChapterII,\,(3.22)]{Dem-book}), the sequence of ideal sheaves generated by the holomorphic functions $\{\sigma_{\nu,j,k}(z)\overline{\sigma_{\nu,j,k}(\overline{w})}\}_{k\leq K_j}$ on $B_j\times B_j$ is locally stationary as $K_j$ increases, hence independant of $K_j$ on $B'_j\times B'_j\Subset B_j\times B_j$ for $K_j$ large enough.
    By uniform convergence of $\sum^{\infty}_{k=1}\sigma_{\nu,j,k}(z)\overline{\sigma_{\nu,j,k}(\overline{w})}$ on compact sets, this sum of the series is a section of the coherent ideal sheaf generated by $\{\sigma_{\nu,j,k}(z)\overline{\sigma_{\nu,j,k}(\overline{w})}\}_{k\leq K_j}$ over $B'_j\times B'_j$.
    Hence, for some $C_j>0$ we obtain (see the proof of \cite[Theorem\,2.2.1]{DPS01}) 
    \begin{align*}
        \exp\frac{m_\nu(\varphi_\nu-\varPhi_\nu)}{1+2^{-\nu}}=\sum^{+\infty}_{k=1}|\sigma_{\nu,j,k}(z)|^2\leq C_j\sum^{K_j}_{k=1}|\sigma_{\nu,j,k}(z)|^2 \quad\text{on}\quad B'_j,
    \end{align*}
    where $e^{-\varPhi_\nu}$ is bounded. 
    Then, $\scr{J}_\nu$ locally equal to the integral closure $\overline{\cal{J}_j}$ of the ideal sheaf $\cal{J}_j=(\sigma_{\nu,j,1},\ldots,\sigma_{\nu,j,K_j})$ by Brian\c{c}on-Skoda's theorem (see \cite[Theorem\,(11.17)]{Dem12}), 
    thus $\scr{J}_\nu$ is coherent, and we have $\mathrm{supp}\,\mathcal{O}_V/\scr{J}_\nu=Z_\nu$.

    Here, $\scr{J}_\nu$ is defined on an open cover $\bigcup_j B_j$ of $V$, and since $V\Subset\bigcup_j B_j$, restricting $\scr{J}_\nu$ to $V$ ensures that the irreducible components of $\mathrm{supp}\,\mathcal{O}_V/\scr{J}_\nu$ are finite. 
    By Hironaka's desingularization theorem \cite{Hir64}, there exists a proper modification $\pi_\nu:\widetilde{V}\longrightarrow V$ obtained by a finite sequence of blow-ups with smooth centers, satisfying 
    \begin{itemize}
        \item the sheaf $\pi^{-1}_\nu\!\!\scr{J}_\nu\cdot\cal{O}_{\widetilde{V}}=\rom{im}(\pi^*_\nu\scr{J}_\nu\rightarrow\cal{O}_{\widetilde{V}})$ 
        is equal to the ideal sheaf $\cal{O}_{\widetilde{V}}(-\sum^J_{j=1}\xi_jD_j)$ for some non-negative integers $\xi_j\in\bb{N}\cup\{0\}$.
        \item $K_{\widetilde{V}}=\pi^*_\nu K_V\otimes\cal{O}_{\widetilde{V}}(\sum^J_{j=1}\zeta_jD_j)$ for some non-negative integers $\zeta_j\in\bb{N}\cup\{0\}$. 
    \end{itemize} 
    For simplicity, we omit \( j \) locally and denote the local orthonormal basis by $\{\sigma_{\nu,k}\}_{k\in\bb{N}}$.
    Let $g$ be the local generator of the ideal generated by $\{\pi_\nu^*\sigma_{\nu,k}\}_{k\in\bb{N}}$, then
    there exists holomorphic functions $\tau_{\nu,k}$ such that $\pi^*_\nu\sigma_{\nu,k}=g\cdot \tau_{\nu,k}$, where $\tau_{\nu,k}$ have no common zeros (see \cite[Lemma\,2.3.19]{MM07}). 
    We consider the decomposition $g=\prod g_j^{a_j}$ of $g$ in irreducible factors, 
    then the local weight of $\pi^*_\nu h_\nu$ has the form 
    \begin{align*}
        \pi^*_\nu\varphi_\nu
        &=\frac{1+2^\nu}{2^\nu m_\nu}\sum_j a_j\log|g_j|+\frac{1+2^\nu}{2^{\nu+1}m_\nu}\log\sum_{k\in\bb{N}}|\tau_{\nu,k}|^2+\pi^*_\nu\varPhi_\nu,
    \end{align*}
    where $\pi^*_\nu h_\nu=\pi^*_\nu h_0\cdot e^{-2\pi^*_\nu\varphi_\nu}$ on $\widetilde{V}$. 
    We introduce global divisors $D_j$ given locally by generators $g_j$. 
    Thus, the divisor $D_\nu=\sum^J_{j=1}a_jD_j$ is simple normal crossing with $\rom{supp}\,D_\nu=\pi^{-1}_\nu(Z_\nu)$ and $\pi_\nu:\widetilde{V}\setminus\pi^{-1}(Z_\nu)\longrightarrow V\setminus Z_\nu$ is biholomorphic. 
    We already know (see \cite[Remark\,5.9]{Dem12})
    \begin{align*}
        \scr{I}(\pi^*_\nu h^p_\nu)=\cal{O}_{\widetilde{V}}\Bigl(-\sum^J_{j=1}\lfloor\frac{(1+2^\nu)a_jp}{2^\nu m_\nu}\rfloor D_j\Bigr).
    \end{align*}

    For a canonical section $s_{D_\nu}$ of $\cal{O}_{\widetilde{V}}(D_\nu)$, we define the natural singular Hermitian metric $h_{D_\nu}$ by $h_{D_\nu}:=1/|s_{D_\nu}|^2$, then we obtain $\scr{I}(h_{D_\nu})=\cal{O}_{\widetilde{V}}(-D_\nu)=\scr{I}(\pi_\nu^*h_\nu^{m_\nu})$.
    Since the local weight of \( h_{D_\nu} \) is \( \sum^J_{j=1} a_j\log|g_j| \), the Hermitian metric \( \widetilde{h}_{m_\nu} \) on the line bundle $\widetilde{L}_{m_\nu}\!:=\pi_\nu^*L^{\otimes 2^\nu m_\nu}\otimes\cal{O}_{\widetilde{V}}(-(1+2^\nu)D_\nu)=\scr{L}^2(\pi_\nu^*h_\nu^{2^\nu m_\nu})$, defined by $\widetilde{h}_{m_\nu}:=\pi^*_\nu h_\nu^{2^\nu m_\nu}\otimes h_{D_\nu}^{*\otimes 1+2^\nu}$, is smooth on $\widetilde{V}$.
    Here, $\iO{\cal{O}_{\widetilde{V}}(D_\nu),h_{D_\nu}}=0$ on $\widetilde{V}\setminus\pi_\nu^{-1}(Z_\nu)$. Clearly, we have 
    \begin{align*}
        \iO{\widetilde{L}_{m_\nu},\widetilde{h}_{m_\nu}}
        =2^\nu m_\nu\pi^*_\nu\iO{L,h_\nu}\geq 2^\nu m_\nu\pi^*_\nu(\gamma-\varepsilon_\nu\omega)
    \end{align*}
    on \( \widetilde{V}\setminus\pi_\nu^{-1}(Z_\nu) \), and since \( \widetilde{h}_{m_\nu} \) is smooth on $\widetilde{V}$, this inequality also holds on \( \widetilde{V} \).

    By using the fact that the dual of the line bundle associated with the exceptional divisor has a smooth Hermitian metric that is positive on the exceptional divisor, we obtain the following lemma.

    \begin{lemma}\label{lemma positivity of exc div to blow ups}
        Under this blow-up setting, holomorphic line bundles $L^*_{D_j}:=\cal{O}_{\widetilde{V}}(-D_j)$ corresponding to each exceptional divisor $D_j$ have smooth Hermitian metrics $h^*_j$ satisfying the following conditions.
        \begin{itemize}
            \item $\iO{L_{D_j}^*,h^*_j}$ is positive on $D_j\setminus\bigcup_{j\ne l}D_l$ for any $1\leq j<J$,
            \item $\iO{L_{D_J}^*,h^*_J}$ is positive on $D_J$,
            \item due to the relative compactness of $\widetilde{V}$, 
            the sum of the Chern curvatures 
            \begin{align*}
                \sum^J_{j=1}\delta_j\iO{L_{D_j}^*,h^*_j}=-\sum^J_{j=1}\delta_j\iO{L_{D_j},h_j}  
            \end{align*}
            becomes positive on $D=\sum^J_{j=1} D_j$ for small enough $\delta_j>0$.
        \end{itemize}
        Here, each $\delta_j$ has some degree of freedom and can also be made a rational number.
    \end{lemma}

    \begin{proof}
        We can assume that \( \widetilde{V} \) is obtained as a tower of blow-ups
        \begin{align*}
            \widetilde{V}=V_N\xrightarrow{\mu_N}V_{N-1}\xrightarrow{\mu_{N-1}}\cdots\xrightarrow{\mu_2}V_1\xrightarrow{\mu_1}V_0=V,
        \end{align*}
        where $\mu_{j+1}:V_{j+1}\longrightarrow V_j$ is a blow-up with smooth center $C_j\subset V_j$, $Z_0:=Z_\nu$, $Z_{j+1}$ is the strict transform of $Z_j$ by $\mu_{j+1}$ and $E_{j+1}$ is the exceptional divisor in $V_{j+1}$. 
        Therefore, let $D'_{j+1}:=\mu^{-1}_{j+1}(C_j)$ be a exceptional divisor in $V_j$, then $E_{j+1}=E'_j\bigcup D'_{j+1}$ where $E'_j$ denotes the set of strict transforms by $\mu_{j+1}$ of all divisors in $E_j$ and $E_1=D'_1$

        The line bundle $\cal{O}_{V_{j+1}}(-D'_{j+1})|_{D'_{j+1}}$ is equal to $\cal{O}_{\bb{P}(N_j)}(1)$, where $N_j:=N_{C_j/V_j}$ is the normal bundle to $C_j$ in $V_j$. 
        We choose any smooth Hermitian metric on \( N_j \), using this metric to induce the Fubini-Study metric on \( \mathcal{O}_{\bb{P}(N_j)}(1) \), 
        and extend this metric as a smooth Hermitian metric $\hbar^*_{j+1}$ on \( L^*_{j+1}:=\mathcal{O}_{V_{j+1}}(-D'_{j+1}) \). 
        This $\hbar^*_{j+1}$ has positive curvature along the tangent vectors of \( V_{j+1} \) which are tangent to the fibers of \( D'_{j+1} = \mathbb{P}(N_j) \longrightarrow C_j \). 
        We simply say that $\iO{L^*_{j+1},\hbar^*_{j+1}}$ \textit{is} \textit{positive} \textit{on} $D'_{j+1}$ in this situation. 

        Since $\mu_2:V_2\setminus D'_2\longrightarrow V_1\setminus C_1$ is biholomorphic, we obtain $\mu_2^*\iO{L^*_1,\hbar^*_1}=\iO{\mu_2^*L^*_1,\mu_2^*\hbar^*_1}$ is positive on $\mu_2^{-1}(D'_1)\setminus D'_2=E_2\setminus D'_2$ and $\iO{L^*_2,\hbar^*_2}$ is positive on $D'_2$.
        Due to the compactness of $\mu_2^{-1}(D'_1)\bigcap D'_2=E'_1\bigcap D'_2$, there exists a sufficiently small $\delta_1>0$ such that $\delta_1\mu_2^*\iO{L^*_1,\hbar^*_1}+\iO{L^*_2,\hbar^*_2}>0$ on $D_2'\bigcup\mu_2^{-1}(D'_1)=E_2$.
        Let $\widetilde{\mu}_q$ be the composition of blow-ups from $\mu_q$ to $\mu_N$, here $\widetilde{\mu}_N=\pi_\nu$.
        Then, letting $h^*_j:=\widetilde{\mu}_{j+1}^*\hbar^*_j$ and $D_j:=\widetilde{\mu}_j^*(C_{j-1})$, we obtain $\widetilde{\mu}_{j+1}^*L^*_j=\cal{O}_{\widetilde{V}}(-D_j)$ and the desired condition inductively.
    \end{proof}

    $(\alpha)$ We assume that $h$ is singular positive, i.e., $\gamma>0$. Fix an integer $\nu$ with $\gamma-\varepsilon_\nu\omega>0$, then there exists $\varepsilon_V>0$ such that $\gamma-\varepsilon_\nu\omega>\varepsilon_V\omega$ on $V$.
    Here, we know $\iO{\widetilde{L}_{m_\nu},\widetilde{h}_{m_\nu}}\geq 2^\nu m_\nu\varepsilon_V\pi^*_\nu\omega$ on $\widetilde{V}$, and $\pi^*_\nu\omega$ is semi-positive on $\widetilde{V}$ and positive on $\widetilde{V}\setminus D$.
    By Lemma \ref{lemma positivity of exc div to blow ups}, there exist positive integers $b_j$ and smooth Hermitian metrics $h^*_j$ on $L^*_{D_j}:=\cal{O}_{\widetilde{V}}(-D_j)$ such that $\Omega:=\sum^J_{j=1}b_j\iO{L^*_{D_j},h^*_j}>0$ on $D$.
    From the relatively compactness of $\widetilde{V}$ and smooth-ness of $h^*_j$, there exists a positive integer $t_V\gg0$ such that $t_V2^\nu m_\nu\varepsilon_V\pi^*_\nu\omega+\Omega$ is positive on $\widetilde{V}$.
    Thus, by setting the smooth Hermitian metric of $\cal{O}_{\widetilde{V}}(-D_b)$ as $h^*_b:=\bigotimes^J_{j=1}{h_j^*}^{\otimes b_j}$, for any $t\geq t_V$ we obtain
    \begin{align*}
        \iO{\widetilde{L}_{m_\nu}^{\otimes t},\widetilde{h}_{m_\nu}^{\otimes t}}+\iO{\cal{O}_{\widetilde{V}}(-D_b),h^*_b}=t\iO{\widetilde{L}_{m_\nu},\widetilde{h}_{m_\nu}}+\sum^J_{j=1}b_j\iO{L^*_{D_j},h^*_j}\geq t2^\nu m_\nu\varepsilon_V\pi^*_\nu\omega+\Omega>0 \quad\rom{on}\,\, \widetilde{V}.
    \end{align*}

    $(\beta)$ This follows from a similar local argument as in \cite[(2.3.45)]{MM07} by applying Leray theorem and Nadel vanishing theorem for weakly pseudoconvex manifolds. 
    Indeed, since the proper modification $\pi_\nu$ is obtained by a finite sequence of blow-ups with smooth centers, $(\beta)$ follows by applying Proposition \cite[Proposition\,2.3.25]{MM07} at each step of the blow-ups.
\end{proof}

\subsection{Construction of singular Hermitian metrics via approximation}\label{subsection: canonical sHm}

Let $X$ be a open complex manifold with a Hermitian metric $\omega$ and an exhaustive open covering $\{X_j\}_{j\in\bb{N}}$, i.e., $X_j\Subset X_{j+1}$ and $\bigcup_j X_j=X$.
Such an exhaustive open covering always exists due to the $\sigma$-compactness of complex manifolds.
Let $L$ be a holomorphic line bundle on $X$ with a singular Hermitian metric $h=h_0e^{-2\varphi}$. 
In this subsection, for an arbitrarily chosen integer $\ell\in\bb{N}$, we construct a new singular Hermitian metric $h_\natural:=h_\natural(\ell)$ that preserves the multiplier ideal sheaf, i.e., $\scr{I}(h^\ell)=\scr{I}(h_\natural^\ell)$ on $X$.

By applying Theorem \ref{Dem appro with log poles and ideal sheaves} and \ref{Blow ups of Dem appro with log poles and ideal sheaves} and its proof on each $X_j$, there exist increasing sequences of positive integers $\{\nu_j(\ell)\}_{j\in\bb{N}}$ and $\{m_{j,\nu}\}_{\nu\in\bb{N}}$, 
and sequences of quasi-plurisubharmonic functions $\{\varphi_{m_{j,\nu}}\}_{\nu\in\bb{N}}$ and $\{\widetilde{\varphi}_{j,\nu}\}_{\nu\in\bb{N}}$ on $X_j$ that satisfies the following. 
\begin{itemize}
    \item [$(\ref{subsection: canonical sHm}\,a)$] $\varphi_{m_{j,\nu}}$ has the same singularities as $1/2m_{j,\nu}$ times a logarithm of a sum of squares of holomorphic functions, i.e., algebraic singularities.
    \item [$(\ref{subsection: canonical sHm}\,b)$] $\varphi_{m_{j,\nu}}$ and $\widetilde{\varphi}_{j,\nu}$ are smooth on $X_j\setminus Z_{j,\nu}$, where $Z_{j,\nu}$ is an analytic subset of $X_j$ obtained as the singular locus of $\varphi_{m_{j,\nu}}$, satisfying $Z_{j,\nu}\subset Z_{j,\nu+1}$.
    \item [$(\ref{subsection: canonical sHm}\,c)$] $\widetilde{\varphi}_{j,\nu}$ is decreasing in $\nu$ and converges to $\varphi$.
    \item [$(\ref{subsection: canonical sHm}\,d)$] $\scr{I}(t\varphi)=\scr{I}(t\varphi_{m_{j,\nu}})=\scr{I}(t\widetilde{\varphi}_{j,\nu})$ for any $0<t\leq\ell$ and any integer $\nu\geq\nu_j(\ell)$.
    \item [$(\ref{subsection: canonical sHm}\,e)$] if $\iO{L,h}\geq\gamma$ for a continuous real $(1,1)$-form $\gamma$ on $X$, then there exists a decreasing sequence of positive numbers $\{\varepsilon_{j,\nu}\}_{\nu\in\bb{N}}$ with $\lim_{\nu\to+\infty}\varepsilon_{j,\nu}=0$ such that $\iO{L,h_{j,\nu}}\geq\gamma-\varepsilon_{j,\nu}\omega$ on $X_j$, where $h_{j,\nu}:=h_0e^{-2\varphi_{m_{j,\nu}}}$.
\end{itemize}

Here, if \( \gamma \) is positive, i.e., $h$ is singular positive, by retaking the sequence \( \{ \nu_j(\ell) \}_{j\in\bb{N}} \), we can have \( \iO{L,h_{j,\nu_j(\ell)}}\geq\gamma/2 \) and $\iO{L,\widetilde{h}_j(\ell)}\geq\gamma/2$ on each $X_j$, where \( \widetilde{h}_j(\ell) = h_0e^{-2\widetilde{\varphi}_{j,\nu_j(\ell)}} \).

\begin{proposition}\label{Prop of two quasi-psh by Dem-appro}
    Fix any integer $j\in\bb{N}$. We can take $\{m_{j+1,\nu}\}_{\nu\in\bb{N}}$ to be a subsequence of $\{m_{j,\nu}\}_{\nu\in\bb{N}}$. In other words, there exists an increasing sequence of positive integers $\{k_j(\nu)\}_{\nu\in\bb{N}}$ such that $m_{j+1,\nu}=m_{j,k_j(\nu)}$.
    For a given sequence of quasi-plurisubharmonic functions $\{\varphi_{m_{j,\nu}}\}_{\nu\in\bb{N}}$ on $X_j$, a sequence of quasi-plurisubharmonic functions $\{\varphi_{m_{j+1,\nu}}\}_{\nu\in\bb{N}}$ on $X_{j+1}$ can be constructed to satisfy $\varphi_{m_{j+1,\nu}}|_{X_j}=\varphi_{m_{j,k_j(\nu)}}$.
    Similarly, we obtain $\widetilde{\varphi}_{j+1,\nu}|_{X_j}\leq\widetilde{\varphi}_{j,k_j(\nu)}\leq\widetilde{\varphi}_{j,\nu}$ and $Z_{j+1,\nu}|_{X_j}=Z_{j,k_j(\nu)}$ for any $\nu\in\bb{N}$.
\end{proposition}

\begin{proof}
    For each $X_j$, taking functions $\psi_{j,\varepsilon,\nu}$ constructed similarly to the proof of Theorem \ref{Dem appro with log poles and ideal sheaves} (see \cite[Theorem\,2.3]{DPS01}), it depends on the choice of the open covering of $X_j$ due to its construction.  
    Therefore, when constructing $\psi_{j+1,\varepsilon,\nu}$ using a similar method on $X_{j+1}$, we obtain $\psi_{j+1,\varepsilon,\nu}|_{X_j}=\psi_{j,\varepsilon,\nu}$ by using the same cover as was used for constructing $\psi_{j,\varepsilon,\nu}$ on $X_j$.
    Since Richberg's regularization theorem (see Lemma\,\ref{Lemma of Richberg's regularization theorem}) also depends on the choice of open covering, $\psi_{j+1,\varepsilon,\nu}|_{X_j}=\psi_{j,\varepsilon,\nu}$ is preserved even if $\psi_{j,\varepsilon,\nu}$ is regularized. 
    Thus, we consider $\psi_{j,\varepsilon,\nu}$ to be regularized. For each $\nu$, there exists an integer $p_j(\nu)$ satisfying the integral inequality 
    \begin{align*}
        \int_{X_j}\Bigl(e^{-2\varphi}-e^{-2\max\{\varphi,(1+2^{-\nu})\psi_{j,2^{-\nu},2^{p_j(\nu)}}\}}\Bigr)dV_\omega\leq 2^{-\nu},
    \end{align*}
    and the quasi-plurisubharmonic function $\varphi_{m_{j,\nu}}$ is obtained as $\varphi_{m_{j,\nu}}:=(1+2^{-\nu})\psi_{j,2^{-\nu},2^{p_j(\nu)}}$ by setting $m_{j,\nu}:=2^{p_j(\nu)}$.
    From $\psi_{j+1,\varepsilon,\nu}|_{X_j}=\psi_{j,\varepsilon,\nu}$ and the integral inequality, for each $\nu$ there exists an integer $k_j(\nu)\geq\nu$ such that one can be choose $p_{j+1}(\nu)$ to satisfy $p_{j+1}(\nu)=p_j(k_j(\nu))$. 
    Then we obtain $m_{j+1,\nu}=m_{j,k_j(\nu)}$ and $\varphi_{m_{j+1,\nu}}|_{X_j}=\varphi_{m_{j,k_j(\nu)}}$.
    Here, by the definition 
    $\widetilde{\varphi}_{j,\nu}:=\sup_{\mu\geq\nu}\varphi_{m_{j,\mu}}$, the sequence $\{\widetilde{\varphi}_{j,\nu}\}_{\nu\in\bb{N}}$ is decreasing in $\nu$, i.e., the condition $(\ref{subsection: canonical sHm}\,c)$. 
    Hence, the following inequality is obtained on each $X_j$.
    \begin{align*}
        \widetilde{\varphi}_{j+1,k_j(\nu)}\leq\widetilde{\varphi}_{j+1,\nu}=\sup_{\mu\geq\nu}\varphi_{m_{j+1,\mu}}=\sup_{\mu\geq\nu}\varphi_{m_{j,k_j(\mu)}}\leq\sup_{\mu\geq k_j(\nu)}\varphi_{m_{j,\mu}}=\widetilde{\varphi}_{j,k_j(\nu)}\leq\widetilde{\varphi}_{j,\nu}.
    \end{align*}
    Since $Z_{j,\nu}$ is the set of logarithmic poles of $\varphi_{m_{j,\nu}}$, it follows that $Z_{j+1,\nu}|_{X_j}=Z_{j,k_j(\nu)}$.
\end{proof}

Using the decreasing properties of Demailly's approximation, i.e., the condition $(\ref{subsection: canonical sHm}\,c)$ related to $\{\widetilde{\varphi}_{j,\nu}\}_{\nu\in\bb{N}}$, 
we define a new singlular Hermitian metric in a canonical way such that the singularity becomes stronger as it approaches the boundary of \( X \).

\begin{definition}\label{def of canonical sHm}
    Let $X$ be a complex manifold with a exhaustive open covering $\{X_j\}_{j\in\bb{N}}$. 
    In the above setting, we construct a new singular Hermitian metric $h_\natural:=h_\natural(\ell)$ by 
    \begin{align*}
        h_\natural(\ell):=h_0e^{-2\varphi_\natural}\quad \rom{and} \quad \varphi_\natural:=\widetilde{\varphi}_{j,\nu_j(\ell)} \quad  \rom{on} \quad X_j\setminus X_{j-1} \text{ for  any } j\in\bb{N}.  
    \end{align*}
    We say that the constructed singular Hermitian metric $h_\natural$ is \textit{the} \textit{singular} \textit{Hermitian} \textit{metric} \textit{constructed} \textit{via} \textit{approximation} associated with $h$.
\end{definition}

By this construction and Proposition \ref{Prop of two quasi-psh by Dem-appro}, we obtain the following characterization.

\begin{theorem}\label{Thm characterizations of canonical sHm}
    Let $\{X_j\}_{j\in\bb{N}}$ be an exhaustive open covering of a complex manifold $X$, $\ell$ be a positive integer 
    and $L$ be a holomorphic line bundle on $X$ with a singular Hermitian metric $h$. 
    Then, the singular Hermitian metric $h_\natural:=h_\natural(\ell)$ on $L$, constructed via approximation, satisfies $\scr{I}(h^t)=\scr{I}(h_\natural^t)$ on $X$ for any $0<t\leq\ell$,
    and there exist a singular Hermitian metric $h_j$ on $L|_{X_j}$ with $\scr{I}(h^t)=\scr{I}(h_j^t)$ for any $0<t\leq\ell$ and an analytic subset $Z_j$ of $X_j$ such that the following conditions are satisfied.
    \begin{itemize}
        \item [$(a)$] $h_j$ is smooth on $X_j\setminus Z_j$ and has algebraic singularities,
        \item [$(b)$] $Z_j\subseteq Z_{j+1}|_{X_j}$ for any $j\in\bb{N}$,
        \item [$(c)$] $\scr{I}(h_j^p)\subseteq\scr{I}(h_\natural^p)$ on $X_j$ for any $p\in\bb{N}$,
        \item [$(d)$] $\bigcup_{p\in\bb{N}}V(\scr{I}(h_\natural^p))\subset Z_j$ on $X_j$.
    \end{itemize}
\end{theorem}


\begin{proof}
    By Proposition \ref{Prop of two quasi-psh by Dem-appro} and $(\ref{subsection: canonical sHm}\,c)$, if $k<j$ then $\widetilde{\varphi}_{j,\nu_j(\ell)}|_{X_k}\leq\widetilde{\varphi}_{k,\nu_j(\ell)}\leq\widetilde{\varphi}_{k,\nu_k(\ell)}$, where $\nu_k(\ell)<\nu_j(\ell)$.
    As a result, this yields $\widetilde{\varphi}_{j,\nu_j(\ell)}\leq\varphi_\natural$ on $X_j$ for any $j\in\bb{N}$.
    We define $h_j:=h_0e^{-2\varphi_{m_{j,\nu_j(\ell)}}}$ on $X_j$ and define $Z_j:=Z_{j,\nu_j(\ell)}$, then $(a)$ follows by $(\ref{subsection: canonical sHm}\,b)$, and $(c)$ follows by the inequality 
    \begin{align*}
        (1+2^{-\nu_j(\ell)})\varphi\leq\varphi_{m_j,\nu_j(\ell)}\leq\widetilde{\varphi}_{j,\nu_j(\ell)}\leq\varphi_\natural\quad\rom{on}\quad X_j.
    \end{align*}    
    From Proposition \ref{Prop of two quasi-psh by Dem-appro} and $(\ref{subsection: canonical sHm}\,b)$, 
    it follows that $Z_{j+1}|_{X_j}=Z_{j,k_j(\nu_{j+1}(\ell))}\supseteq Z_j$, where $k_j(\nu_{j+1}(\ell))\geq\nu_{j+1}(\ell)>\nu_j(\ell)$.
    By the conditions $(a)$ and $(c)$, we obtain $\scr{I}(h_\natural^p)=\cal{O}_X$ on $X_j\setminus Z_j$ for any $p\in\bb{N}$, which implies $(d)$.
    It follows from $(\ref{subsection: canonical sHm}\,d)$ that $\scr{I}(h^t)=\scr{I}(h_\natural^t)=\scr{I}(h_j^t)$ for any $0<t\leq\ell$.
\end{proof}

\begin{conjecture}
    If \( h \) is singular positive, does the singular Hermitian metric \( h_\natural \), constructed via approximation, also become singular positive?
    Here, clearly $h_\natural$ is singular positive on $X\setminus\bigcup_j\partial X_j$.
\end{conjecture}

By a result due to Siu (cf. \cite[Theorem\,2.10]{Dem12}), for any $c>0$, the super-level set
\begin{align*}
    E_c(h):=\{x\in X\mid\nu(\varphi,x)\geq c\}
\end{align*}
is a proper analytic subset of $X$, where $\nu(\varphi,x)$ is the Lelong number of $\varphi$ at $x\in X$.
Additionally, it is known that $E_n(h)\subseteq V(\scr{I}(h))\subseteq E_1(h)$ (see \cite[Lemma\,5.6]{Dem12}).
Clealy, $E_{n/m}(h)\subseteq V(\scr{I}(h^m))\subseteq E_{1/m}(h)$, and we obtain 
\vspace{1mm}
\begin{center}
    $E_{+}(h):=\{x\in X\mid\nu(\varphi,x)>0\}=\bigcup_{m\in\bb{N}}V(\scr{I}(h^m))$.
\end{center}
\vspace{1mm}
As one result, it is known that $\bb{B}_{-}(L)\subset E_{+}(h)$ if $X$ is projective and $h$ is singular semi-positive, i.e., $(L,h)$ is pseudo-effective, (see \cite[Lemma\,2.3.6]{PT18}).

In considering the embedding, we need to take into account the set $V(\scr{I}(h^m))$ for any $m\in\bb{N}$, 
but understanding the behavior of this analytic set as $m \to +\infty$ is difficult, and it is unclear whether it remains analytic.
Therefore, we will consider $h_\natural$ instead of $h$, and to handle the set $E_{+}(h_\natural)$, we consider the union $\bigcup_{j\in\bb{N}}Z_j$.
By $(b)$ and $(d)$ of Theorem \ref{Thm characterizations of canonical sHm}, the union $\bigcup_{j\in\bb{N}}Z_j$ is naturally defined on $X$, and we obtain 
\vspace{1mm}
\begin{center}
    $E_{+}(h_\natural)=\bigcup_{m\in\bb{N}}V(\scr{I}(h_\natural^m))\subset \bigcup_{j\in\bb{N}}Z_j$.
\end{center}
\vspace{1mm}
In particular, the union $\bigcup_{j\in\bb{N}}Z_j$ does not depend on the construction of \( h_\natural \) by the following.  
Here, the initial observation is as follows. Let $\bigcap^N_{j=1}f_j^{-1}(0)$ be an analytic subset for some $f_1,\ldots,f_N\in\cal{O}_X$, then 
the function $g=\frac{1}{2}\log\sum_j|f_j|^2$ has the property that $x\in \bigcap_jf_j^{-1}(0)$ if and only if $\nu(g,x)>0$. In other words, $\bigcap_jf_j^{-1}(0)=\bigcup_{c>0}E_c(g)$.

\begin{proposition}\label{Prop E0(h)=cup Zj}
    We have $\bigcup_{m\in\bb{N}}V(\scr{I}(h^m))=\{x\in X\mid \nu(\varphi,x)>0\}=\bigcup_{j\in\bb{N}}Z_j$. 
\end{proposition}

\begin{proof}
    For simplicity, let $\nu_j(\ell)=\nu_j$.
    Recall that from Theorem \ref{Thm characterizations of canonical sHm} and its proof, we have \( h_j = h_0e^{-2\varphi_{m_{j,\nu_j}}} \), and \( Z_j = Z_{j,\nu_j} \) is the set of logarithmic poles of $\varphi_{m_{j,\nu_j}}$.
    By $(a)$ of Theorem \ref{Dem appro with log poles and ideal sheaves}, locally $\varphi_{m_{j,\nu_j}}$ can be expressed as 
    \begin{align*}
        \varphi_{m_{j,\nu_j}}=\frac{1}{2m_{j,\nu_j}}\log\sum_{k\in\bb{N}}|\sigma_{j,k}|^2+\varPhi_{m_{j,\nu_j}},
    \end{align*}
    where $\{\sigma_{j,k}\}_{k\in\bb{N}}$ is an orthonormal basis of a Hilbert space. 
    Similarly to the proof of Theorem \ref{Blow ups of Dem appro with log poles and ideal sheaves}, there exists a positive integer \( K_j\in\bb{N} \) such that locally \( Z_j = \bigcap^{K_j}_{k=1}\sigma_{j,k}^{-1}(0) \), and we also obtain \( Z_j = \bigcup_{c>0}E_c(\varphi_{m_{j,\nu_j}}) \).
    From the condition $(e)$ of Theorem \ref{Dem appro with log poles and ideal sheaves}, i.e., 
    \begin{align*}
        \nu(\varphi,x)-\frac{n}{m_{j,\nu_j}}\leq\nu\Big(\frac{\varphi_{m_{j,\nu_j}}}{1+2^{-\nu_j}},x\Big)\leq\nu(\varphi,x),
    \end{align*}
    we have the inclusion $E_{c+n/{m_{j,\nu_j}}}(\varphi)\subset E_c\Big(\frac{\varphi_{m_{j,\nu_j}}}{1+2^{-\nu_j}}\Big)\subset E_c(\varphi)$ for any $c>0$.
    Thus, from the obtained inclusion $Z_j=\bigcup_{c>0}E_c\Big(\frac{\varphi_{m_{j,\nu_j}}}{1+2^{-\nu_j}}\Big)\subset\bigcup_{c>0}E_c(\varphi)=E_{+}(\varphi)$ on $X_j$ for any $j\in\bb{N}$, the inclusion $\bigcup_{j\in\bb{N}}Z_j\subset E_{+}(\varphi)=E_{+}(h)$ is first derived.
    Finally, from the inclusion \( Z_j\supset\bigcup_{c>0}E_{c+n/{m_{j,\nu_j}}}(\varphi) \) and $(b)$ of Theorem \ref{Thm characterizations of canonical sHm}, the following inclusion 
    \begin{align*}
        \bigcup_{j\in\bb{N}}Z_j\supset\bigcup_{j\in\bb{N}}\bigcup_{c>0}E_{c+n/{m_{j,\nu_j}}}(\varphi)=\bigcup_{c>0}E_c(\varphi)=E_{+}(h)
    \end{align*}
    is obtained by $\lim_{j\to+\infty}m_{j,\nu_j}=+\infty$.
\end{proof}

In other words, this proposition indicates that \( E_{+}(\varphi)=\{x\in K\mid\nu(\varphi,x)>0\} \) can be obtained as the limit of the analytic sets \( Z_\nu \), satisfying $Z_\nu\subset Z_{\nu+1}$, obtained from Demailly's approximation (= Theorem \ref{Dem appro preserves ideal sheaves} or \ref{Dem appro with log poles and ideal sheaves}).

\subsection{Setting on each sublevel set $X_c$}\label{Setting of each sublevel set}

Let $(X,\varPsi)$ be a weakly pseudoconvex manifold and $L$ be a holomorphic line bundle on $X$ with a singular positive Hermitian metric $h$.
We fix an arbitrary $c<\sup_X\varPsi$.
By blow-ups Theorem \ref{Blow ups of Dem appro with log poles and ideal sheaves}, there exist a singular Hermitian metric $h_c$ on $L|_{X_c}$ with algebraic singularities and $\scr{I}(h)=\scr{I}(h_c)$, an analytic subset $Z_c$ of $X_c$ and a proper modification $\pi_c:\widetilde{X}_c\longrightarrow X_c$ that satisfies the following 
\begin{itemize}
    \item [$(\ref{Setting of each sublevel set}\,a)$] $h_c$ is smooth on $X_c\setminus Z_c$ and $\pi_c:\widetilde{X}_c\setminus\pi_c^{-1}(Z_c)\longrightarrow X_c\setminus Z_c$ is biholomorphic,
    \item [$(\ref{Setting of each sublevel set}\,b)$] there exist $m_c\in\bb{N}$ and a simple normal crossing divisor $D_c=\sum^J_{j=1}a_jD_j$ with $a_j\in\bb{N}$ and $\rom{supp}\,D_c=\pi_c^{-1}(Z_c)$ such that 
    the holomorphic line bundle $\widetilde{L}_{m_c}:=\pi_c^*L^{\otimes m_c}\otimes\cal{O}_{\widetilde{X}_c}(-D_c)$ is semi-positive on $\widetilde{X}_c$ and positive on $\widetilde{X}_c\setminus D_c$,
    where $\scr{I}(\pi_c^*h_c^{m_c})\cong\cal{O}_{\widetilde{X}_c}(-D_c)$.
    \item [$(\ref{Setting of each sublevel set}\,c)$] there exist an effective divisor $D_b=\sum^J_{j=1}b_jD_j$ and enough large $t_c\in\bb{N}$ such that the holomorphic line bundle $\scr{L}:=\widetilde{L}_{m_c}^{\otimes t}\otimes\cal{O}_{\widetilde{X}_c}(-D_b)$ is positive on $\widetilde{X}_c$ for any integer $t\geq t_c$.
\end{itemize}

Note that within these conditions, \( m_c \) and \( a_j \) correspond to \( 2^\nu m_\nu \) and \( (1 + 2^\nu)a_j \), respectively, in Theorem \ref{Blow ups of Dem appro with log poles and ideal sheaves} for some $\nu\in\bb{N}$.

\begin{remark}\label{Remark of Setting on each sublevel set}
    From the proofs of Theorems \ref{Dem appro with log poles and ideal sheaves} and \ref{Blow ups of Dem appro with log poles and ideal sheaves}, for each $c$, it is possible to take a sufficiently small $0<\varepsilon_c\ll1$ such that all notations in the above setting are considered to hold on $X_{c+\varepsilon_c}$. 
    Henceforth, we will assume this without further notice, and for simplicity, we will omit this $\varepsilon_c$ and write as above. 
    By applying this to the canonical singular Hermitian metric $h_\natural$, the function $\widetilde{\varphi}_{j,\nu_j}$ in Definition \ref{def of canonical sHm} is considered as a function defined on each $X_{j+\varepsilon_j}$. 
    Therefore, $h_j$ and $Z_j$ in Theorem \ref{Thm characterizations of canonical sHm} are also on $X_{j+\varepsilon_j}$, and $(a)$-$(d)$ are actually valid on $X_{j+\varepsilon_j}$.
\end{remark}

\section{Approximation theorem for holomorphic sections\\ of adjoint bundles with multiplier ideal sheaves}\label{Section 4: Approximation thm}

Regarding positive line bundles, the approximation theorem for holomorphic sections of the adjoint bundle 
is one of most important conclusion of the cohomology theory on weakly pseudoconvex manifolds (see \cite{Ohs82,Ohs83}). 
In this section, we present the following approximation theorem 
including multiplier ideal sheaves 
as a generalization to singular positive Hermitian metrics.

\begin{theorem}\label{Approximation thm of hol section with ideal sheaves}$($\textnormal{Approximation theorem}$)$
    Let $(X,\varPsi)$ be a weakly pseudoconvex manifold and $L$ be a holomorphic line bundle on $X$ equipped with a singular Hermitian metric $h$. 
    If $h$ is singular positive, then for any sublevel set $X_c$, the natural restriction map
    \begin{align*}
        \rho_c:H^0(X,K_X\otimes L\otimes\scr{I}(h))\longrightarrow H^0(X_c,K_X\otimes L\otimes\scr{I}(h))
    \end{align*}
    has dense images with respect to the topology induced by the $L^2$-norm $||\bullet||_h$ and the topology of uniform convergence on all compact subsets in $X_c$.
\end{theorem}

\subsection{Regularization theorem preserving ideal sheaves and completeness}\label{subsection: Regularization theorem with ideal sheaves and completeness}

To prove Approximation Theorem \ref{Approximation thm of hol section with ideal sheaves} and the Nadel-type vanishing theorem without \kah metric (= Theorem \ref{Thm Nadel type vanishing without Kahler}), 
as a key result, we construct a singular Hermitian metric that preserves ideal sheaves and possesses curvature of complete \kah type.

\begin{theorem}\label{regularization with the equality of ideals and completeness}
    Let $(X,\varPsi)$ be a weakly pseudoconvex manifold and and $L$ be a holomorphic line bundle with a singular Hermitian metric $h$. 
    We take an arbitrary number $c<\sup_X\varPsi$.
    If $h$ is singular positive on $L|_{X_{c+\upsilon}}$ for some $\upsilon>0$, then there exist a singular Hermitian metric $\hbar_c$ on $L|_{X_c}$ 
    such that the following are satisfied.
    \begin{itemize}
        \item [$(\alpha)$] $\hbar_c$ is smooth on $X_c\setminus Z_c$, where $Z_c$ is an analytic subset of $X_c$.
        \item [$(\beta)$] $\scr{I}(h)=\scr{I}(\hbar_c)$ on $X_c$.
        \item [$(\gamma)$] $\iO{L,\hbar_c}$ is a complete \kah metric on $X_c\setminus Z_c$ and singular positive on $X_c$.
    \end{itemize}
\end{theorem}

\begin{proof}
    From the singular positivity of $h=h_0e^{-2\varphi}$, we obtain $\iO{L,h}\geq\varpi$ on $X_{c+\upsilon}$ for some continuous real positive $(1,1)$-form $\varpi>0$.
    By blow-ups Theorem \ref{Blow ups of Dem appro with log poles and ideal sheaves}, there exist a quasi-plurisubharmonic function $\varphi_c$ with algebraic singularities 
    and a proper modification $\pi:\widetilde{X}_c\longrightarrow X_c$ such that, using the notation from Setting \ref{Setting of each sublevel set} for $X_c$, the following conditons are satisfied. 
    \begin{itemize}
        \item $\iO{L,h_c}\geq\varpi/2>0$ and $\scr{I}(h)=\scr{I}(h_c)$ on $X_c$, where we define $h_c=h_0e^{-2\varphi_c}$.
        \item there is a rational number $m_c\in\bb{Q}_{>0}$ such that locally $\varphi_c$ can be expressed as 
        \begin{align*}
            \varphi_c=\frac{1}{2m_c}\log\Bigl(\sum_k|f_k|^2\Bigr)+\varPhi_c,
        \end{align*}
        where $\{f_k\}_{k\in\bb{N}}$ is the set of holomorphic functions 
        and $\varPhi_c$ is smooth on $X_c$.
        \item $\varphi_c$ is smooth on $X_c\setminus Z_c$, where $Z_c$ is an analytic subset of $X_c$. 
        \item there exists a simple normal crossing divisor $D_c=\sum^J_{j=1}a_jD_j$ with $a_j\in\bb{N}$ and $\rom{supp}\,D_c=\pi^{-1}(Z_c)$, and we have  
        \begin{align*}
            \pi^*\varphi_c=\frac{1}{m_c}\sum^J_{j=1}a_j\log|g_j|+\widetilde{\varPhi}_c \quad\rom{and}\quad \scr{I}(\pi^*\varphi_c)=\cal{O}_{\widetilde{X}_c}\Bigl(-\sum^J_{j=1}\lfloor\frac{a_j}{m_c}\rfloor D_j\Bigr),
        \end{align*}
        where $g_j$ are local generators of $\cal{O}_{\widetilde{X}_c}(D_j)$ and $\widetilde{\varPhi}_c$ is smooth on $\widetilde{X}_c$. 
    \end{itemize}

    Let $\alpha_j=\lfloor\frac{a_j}{m_c}\rfloor\in\bb{N}\cup\{0\}$, $\beta_j=\frac{a_j}{m_c}-\alpha_j\in [0,1)$ and $\sigma_j$ be the defining section of $D_j$, and  
    we define a simple normal crossing divisor $D:=\sum^J_{j=1}\alpha_jD_j$ and the natural singular metric $h_D=1/|\sigma_D|^2=1/\prod^J_{j=1}|\sigma_j|^{2\alpha_j}$ on $\cal{O}_{\widetilde{X}_c}(D)$ by $\sigma_D:=\bigotimes\sigma_j^{\alpha_j}$,
    then we have  $\scr{I}(\pi^*h_c)=\cal{O}_{\widetilde{X}_c}(-D)=\scr{I}(h_D)$. The weight of $h_D$ is $\sum^J_{j=1}\alpha_j\log|\sigma_j|^2$.
    In general, for any real number $\wp_j\in[0,1)$, we already know (see \cite[Remark\,5.9]{Dem12})
    \begin{align*}
        \scr{I}\Bigl(\sum^J_{j=1}(\alpha_j+\wp_j)\log|\sigma_j|^2\Bigr)=\cal{O}_{\widetilde{X}_c}(-D). \tag*{($\wp$)}
    \end{align*}
    Here, locally $\pi^*h_c$ can be expressed as
    \begin{align*}
        \pi^*h_c=\pi^*h_0\cdot e^{-2\pi^*\varphi_c}=\pi^*h_0\cdot e^{-2\widetilde{\varPhi}_c} \frac{1}{\prod^J_{j=1} |g_j|^{2(\alpha_j+\beta_j)}}.
    \end{align*}

    From Hironaka's resolution of singularities theorem (see \cite{Hir64}) and Lemma \ref{lemma positivity of exc div to blow ups}, 
    each line bundle $L_{D_j}^*:=\cal{O}_{\widetilde{X}_c}(-D_j)$ corresponding to each exceptional divisor $D_j$ has a smooth Hermitian metric $h^*_j$ and satisfies that 
    the sum of the Chern curvatures 
        \begin{align*}
            \sum^J_{j=1}\delta_j\iO{L_{D_j}^*,h^*_j}=-\sum^J_{j=1}\delta_j\iO{L_{D_j},h_j}  \tag*{($\delta$)}
        \end{align*}
    is positive on $D_c$ for enough small $\delta_j>0$.
    We construct a new metric $\widetilde{h}_{\varepsilon}$ on $\pi^*L$ by 
    \begin{align*}
        \widetilde{h}_{\varepsilon}:=\pi^*h_c\prod^J_{j=1}\bigl(\log(\varepsilon||\sigma_j||^2_{h_j})\bigr)^2
    \end{align*}
    where 
    $\varepsilon>0$ is an enough small positive constant to be determined later. 
    Clearly, \( \widetilde{h}_{\varepsilon} \) is a smooth Hermitian metric on \( \widetilde{X}_c\setminus D_c \).

\begin{lemma}\label{claim the equality of ideals}
    We have $\scr{I}(\pi^*h_c)=\cal{O}_{\widetilde{X}_c}(-D)=\scr{I}(\widetilde{h}_\varepsilon)$ on $\widetilde{X}_c$. 
    The curvature current $\iO{\pi^*L,\widetilde{h}_\varepsilon}$ is Poincar\'e type along $D_c$ and singular positive on $\widetilde{X}_c$, if $\varepsilon$ is taken small enough.
\end{lemma}

\begin{proof}
    The initial claim follows from the case where \( \wp_j > \beta_j \) in $(\wp)$, and the logarithmic part does not affect the integral. 
    In fact, this can be proven as follows by the following integral calculation. 
    For any positive number $\tau>0$, the integral 
    \begin{align*}
        \int_{0}^{1}r^\alpha(-\log r)^\tau dr<+\infty\quad\text{if and only if}\quad \alpha>-1.
    \end{align*}

    From the boundedness of this integral, it is clear that we obtain 
    \begin{align*}
        \scr{I}\Bigl(\prod^J_{j=1}\bigl(\log(\varepsilon||\sigma_j||^2_{h_j})\bigr)^2\Bigr) = \cal{O}_{\widetilde{X}_c}. 
    \end{align*}
    By the subadditivity theorem (see \cite[Theorem\,14.2]{Dem12}), the inclusion $\scr{I}(\widetilde{h}_{\varepsilon})\subseteq\scr{I}(\pi^*h_c)$ is derived. 
    Therefore, we prove the reverse inclusion. 
    Since $h_j$ and $\sigma_j$ is smooth, it suffices to consider points on \( D_c \) only. 
    Here locally, we may consider \( \sigma_j = g_j \) near $x\in D_c$. 
    We take \( f\in \scr{I}(\pi^*h_c)_x \), then there exists an open neighborhood \( \Delta_r=\{|z_j|<r,\,1\leq j\leq n\} \) for some $r<1$ such that 
    \begin{align*}
        I_f=\int_{\Delta_r}\frac{|f|^2}{\prod|g_j|^{2(\alpha_j+\beta_j)}}dV_{\bb{C}^n}<+\infty.
    \end{align*}
    Since locally the $g_j$ can be taken to be coordinate functions from a local coordinate $(z_1,\ldots,z_n)$, 
    the condition is that $f$ is divisible by $\prod g_j^{m_j}$, where $m_j\geq \alpha_j$ for each $j$. Let $F$ be a holomorphic function with $f=F\prod g_j^{\alpha_j}$. 
    By the strong openness property (see \cite{GZ15}), there exist $0<r'<r$ and $\tau \in (0,\min_j\frac{1}{\beta_j}-1)$ such that 
    \begin{align*}
        I_f=\int_{\Delta_r}\frac{|F|^2}{\prod|g_j|^{2\beta_j}}dV_{\bb{C}^n}<+\infty \quad\rom{and}\quad
        I_{f,\tau}=\int_{\Delta_{r'}}\frac{|F|^2}{\prod|g_j|^{2(1+\tau)\beta_j}}dV_{\bb{C}^n}<+\infty,
    \end{align*}
    where $(1+\tau)\beta_j<1$ for each $j$. 
    From the H\"older's inequality, we obtain 
    \begin{align*}
        \int_{\Delta_{r'}}|f|^2_{\widetilde{h}_{\varepsilon}}dV_{\bb{C}^n}&\leq C\int_\Delta\frac{|F|^2}{\prod|g_j|^{2c_j}}\prod(\log|g_j|^2)^2dV_{\bb{C}^n}\\
        &\leq C\Biggl(\int_{\Delta_{r'}}\frac{|F|^2}{\prod|g_j|^{2(1+\tau)c_j}}\Biggr)^{\frac{1}{1+\tau}}\cdot\Biggl(\int_{\Delta_{r'}}|F|^2\prod(-\log|g_j|^2)^{\frac{2(1+\tau)}{\tau}}\Biggr)^{\frac{\tau}{1+\tau}}\\
        &<+\infty,
    \end{align*}
    for some constant $C>0$.
    This shows $f\in\scr{I}(\widetilde{h}_{\varepsilon})_x$. Here, the following is obtained.
    \begin{align*}
        \int_{\{|z_j|<r'\}} (-\log|g_j|^2)^{\frac{2(1+\tau)}{\tau}} dV_{\bb{C}^n}=2\pi\int^{r'}_0r(-2\log r)^{\frac{2(1+\tau)}{\tau}}dr<+\infty.
    \end{align*}

    A straightforward calculation yields 
    \begin{align*}
        \iO{\pi^*L,\widetilde{h}_{\varepsilon}}&=\iO{\pi^*L,\pi^*h_c}+\sum^J_{j=1}\frac{\iO{L_{D_j},h_j}}{\log(\varepsilon||\sigma_j||^2_{h_j})}+i\sum^J_{j=1}\frac{\partial\log||\sigma_j||^2_{h_j}\wedge\dbar\log||\sigma_j||^2_{h_j}}{\bigl(\log(\varepsilon||\sigma_j||^2_{h_j})\bigr)^2}\\
        &\geq\frac{\pi^*\varpi}{2}+\sum^J_{j=1}\frac{\iO{L_{D_j},h_j}}{\log(\varepsilon||\sigma_j||^2_{h_j})}
    \end{align*}
    on $\widetilde{X}_c\setminus D$, where $\idd\log||\sigma_j||^2_{h_j}=\idd\log|\sigma_j|^2+\idd\log h_j=\iO{L_{D_j}^*,h^*_j}$ on $\widetilde{X}_c\setminus D_j$.
    Since \( \pi^*\varpi \) is non-negative on \( \widetilde{X}_c \) and positive on \( \widetilde{X}_c\setminus D_c \), the curvature $\iO{\pi^*L,\widetilde{h}_\varepsilon}$ becomes singular positive on $\widetilde{X}_c$ for sufficiently small \( \varepsilon>0 \), according to condition $(\delta)$.
    In fact, the right-hand side of the second line of the above equation becomes positive on $\wit{X}_c$.

    In the special coordinates $(z_1,\ldots,z_n)$ of $x\in D_c$, we can write $\varepsilon||\sigma_j||^2_{h_j}=|z_j|^2e^{-\varphi_j}$, where \( \varphi_j \) is a local weight function of \( h_j \). 
    Therefore, the following real $(1,1)$-form
    \begin{align*}
        i\sum_{j=1}^{k}\frac{\partial\log||\sigma_j||^2_{h_j}\wedge\dbar\log||\sigma_j||^2_{h_j}}{\bigl(\log(\varepsilon||\sigma_j||^2_{h_j})\bigr)^2}=i\sum_{j=1}^{k}\frac{1}{(\log|z_j|^2-\varphi)^2}\biggl(\frac{dz_j}{z_j}-\partial\varphi\biggr)\wedge\biggl(\frac{d\overline{z}_j}{\overline{z}_j}-\dbar\varphi\biggr)
    \end{align*}
    is of Poincar\'e type along $D_c$ (see \cite{Zuc79}).
\end{proof}

    Finally, we define the desired metric.
    Let $\chi\in\cal{C}^\infty(\bb{R},\bb{R})$ be a smooth convex increasing function satisfying $\int^{+\infty}_0\sqrt{\chi''(t)}dt=+\infty$, and define a smooth plurisubharmonic function $\psi_c:=\chi\circ\frac{1}{\varPsi-c}$.
    Here, we already know (see \cite[ChapterVIII,\,$\S5$]{Dem-book}) that for any Hermitian metric $\Omega$ on $X_c$, the Hermitian metric $\Omega+\idd\psi_c$ is complete.

    We construct a smooth Hermitian metric $\widetilde{h}_c$ on $\pi^*L$ over $\widetilde{X}_c\setminus D_c$ by 
    \begin{align*}
        \widetilde{h}_c:=\widetilde{h}_\varepsilon e^{-2\pi^*\psi_c}=\pi^*h_0\cdot e^{-2\pi^*(\varphi_c+\psi_c)}\prod_{j=1}^{k}\bigl(\log(\varepsilon||\sigma_j||^2_{h_j})\bigr)^2,
    \end{align*}
    then the \kah metric $\iO{\pi^*L,\widetilde{h}_c}$ is complete on $\widetilde{X}_c\setminus D_c$ by Poincar\'e type along $D_c$ (see \cite{Zuc79}). 
    Pulling back this constructed Hermitian metric by $\mu:=\pi^{-1}:X_c\setminus Z_c\longrightarrow\widetilde{X}_c\setminus\pi^{-1}(Z_c)$, we obtain the desired singular Hermitian metric by 
    \begin{align*}
        \hbar_c:=
        \begin{cases}
          \mu^*\widetilde{h}_c=h_c\cdot e^{-2\psi_c}\prod^J_{j=1}\mu^*\bigl(\log(\varepsilon||\sigma_j||^2_{h_j})\bigr)^2, & \rom{on} \,\,X_c\setminus Z_c, \\
          0,                 & \rom{on} \,\, Z_c.
        \end{cases}
    \end{align*}
    Here, 
    the \kah metric $\iO{L,\hbar_c}=\mu^*\iO{\pi^*L,\widetilde{h}_c}$ is also complete on $X_c\setminus Z_c$.

    By the proper modification and the relationship with ideal sheaves (see \cite[Proposition\,5.8]{Dem12}) and Lemma \ref{claim the equality of ideals}, we have
    \begin{align*}
        K_{X_c}\otimes L\otimes\scr{I}(h_c)&=\pi_*(K_{\widetilde{X}_c}\otimes\pi^*L\otimes\scr{I}(\pi^*h_c))=\pi_*(K_{\widetilde{X}_c}\otimes\pi^*L\otimes\cal{O}_{\widetilde{X}_c}(-D))\\
        &=\pi_*(K_{\widetilde{X}_c}\otimes\pi^*L\otimes\scr{I}(\widetilde{h}_\varepsilon))=\pi_*(K_{\widetilde{X}_c}\otimes\pi^*L\otimes\scr{I}(\pi^*\hbar_c))\\
        &=K_{X_c}\otimes L\otimes\scr{I}(\hbar_c).
    \end{align*}
    Hence, the identity $\scr{I}(h)=\scr{I}(h_c)=\scr{I}(\hbar_c)$ is obtained.
\end{proof}

\subsection{Construction and Setting}\label{subsection: Construction and Setting in Approximation Theorem}

Let $h$ be a singlular positive Hermitian metric on $L$.
We take any pair of real numbers $b<c<\sup_X\varPsi$ and a smooth increasing convex function $\chi_{c,b}:\bb{R}\longrightarrow\bb{R}$ with the following conditions.
\begin{itemize}
    \item $\chi_{c,b}(t)=0$ if $\displaystyle t\leq\frac{1}{c-b}$ and $\chi(t)>0$ when $\displaystyle t>\frac{1}{c-b}$,
    \item $\displaystyle \int^{+\infty}_0\sqrt{\chi_{c,b}''(t)}dt=+\infty$.
\end{itemize}
Let $\psi:=\chi_{c,b}\circ\frac{1}{c-\varPsi}$, then this smooth function is exhaustive and plurisubharmonic on $X_c$, $\psi\equiv0$ on $X_b$ and $\psi>0$ on $X_c\setminus\overline{X_b}$.

In the construction of $\hbar_c$ in Theorem \ref{regularization with the equality of ideals and completeness}, we may use this function \( \chi_{c,b} \), and can assume that \( \hbar_c \) is constructed using \( \chi_{c,b} \). 
That is, 
\( \hbar_c \) can be expressed as
\begin{align*}
    \hbar_c=h_0\cdot e^{-2\varphi_c-\psi}\prod^J_{j=1}\mu^*\Bigl(\log(\varepsilon||\sigma||^2_{h_j})\Bigr)^2=\mu^*\widetilde{h}_\varepsilon\cdot e^{-\psi}=\mu^*\widetilde{h}_\varepsilon\cdot e^{-\chi_{c,b}\circ\frac{1}{c-\varPsi}}.
\end{align*}

For each non-negative integers $m\geq0$, we define singular Hermitian metrics by 
\begin{align*}
    h_{c_m}:=\hbar_c\cdot e^{-(m-1)\psi}=\mu^*\widetilde{h}_\varepsilon\cdot e^{-m\psi},
\end{align*}
specifically setting $h_c:=h_{c_0}$, where $\hbar_c=h_{c_1}$. Note that this $h_c$ is different from the one in Setting \ref{Setting of each sublevel set}.
From $(\gamma)$ in Theorem \ref{regularization with the equality of ideals and completeness}, we take \( \omega:=\iO{L,h_{c_1}} \) as a complete \kah metric on \( X_c\setminus Z_c \).
Here, for any \((n,0)\)-form, the product of the norm and volume is independent of metrics. 
In other words, the following is immediately clear.

\begin{lemma}\label{inequality of (n,0)-forms}
    Let $\gamma_1$ and $\gamma_2$ be Hermitian metrics on $X$ with $\gamma_1\geq\gamma_2$. 
    Then we have 
    \begin{itemize}
        \item the equality $|u|^2_{\gamma_1} dV_{\gamma_1}=|u|^2_{\gamma_2} dV_{\gamma_2}$ holds for any $(n,0)$-form $u$, 
        \item the inequality $|u|^2_{\gamma_1} dV_{\gamma_1}\leq|u|^2_{\gamma_2} dV_{\gamma_2}$ holds for any $(n,q)$-form $u$ and any $q\geq1$.
    \end{itemize}
\end{lemma}

Let $L^2_{n,0}(X_c,L,h_c)$ be the set of $L$-valued $(n,0)$-form on $X_c$ whose $L^2$-norm with respect to $h_c$ is finite.
For simplicity, we write $Z:=Z_c$.
Here, this singular metric $h_c$ also satisfies $\scr{I}(h)=\scr{I}(h_c)$ on $X_c$ according to condition $(\beta)$, 
then we have \( L^2_{n,0}(X_c,L,h)=L^2_{n,0}(X_c,L,h_c)=L^2_{n,0}(X_c\setminus Z,L,h_c) \), 
and this norm does not depend on choice of Hermitian metrics on \( X_c\setminus Z \) by Lemma \ref{inequality of (n,0)-forms}.
This holds true even if we replace \( X_c \) with \( X_b \).
From Dolbeault-Grothendick's lemma (see \cite[Chapter1,\,Lemma\,3.29]{Dem-book}), we naturally obtain the following isomorphisms.
\begin{align*}
    H^0(X_c,K_X\otimes L\otimes\scr{I}(h))\cong\rom{ker}\,\dbar\cap L^2_{n,0}(X_c,L,h_c),\\
    H^0(X_b,K_X\otimes L\otimes\scr{I}(h))\cong\rom{ker}\,\dbar\cap L^2_{n,0}(X_b,L,h_c).
\end{align*}

These isomorphism depend on the fact that each sublevel set is relatively compact, 
and on the metric \( h_c \) does not contain \( e^{-\psi} \) from \( \int_{X_c}e^{-\psi}=+\infty \), in the case of \( X_c \).
In fact, we define \( \scr{L}^{n,q}_{L,h} \) as the sheaf of germs of $L$-valued $(n,q)$-forms $u$ with measurable coefficients such that both $|u|^2_h$ and $|\dbar u|^2_h$ are locally integrable. 
Then we obtain the following isomorphism over the entire \( X \). 
\begin{align*}
    H^0(X,K_X\otimes L\otimes\scr{I}(h))\cong\rom{ker}\,\dbar\cap\Gamma(X,\scr{L}^{n,0}_{L,h})
\end{align*}
It is not clear that this right-hand side coincides with \(\rom{ker}\,\dbar \cap L^2_{n,0}(X,L,h) \).

\begin{definition}
    For any non-negative integer $m\geq0$ and any $f,g\in L^2_{p,q}(X_c\setminus Z,L,h_{c_m},\omega)$ with $p,q\geq0$, we define the $L^2$-inner product by 
    \begin{align*}
        \lla f,g\rra_{m,\omega}:=\int_{X_c\setminus Z}\lara{f}{g}_{h_{c_m},\omega}\,dV_\omega=\int_{X_c\setminus Z}\lara{f}{g}_{h_c,\omega}e^{-m\psi}dV_\omega.
    \end{align*}
    For the case $(p,q)=(n,0)$, this $L^2$-inner product does not depend on \( \omega \) by Lemma \ref{inequality of (n,0)-forms}, thus we simply write it as \( ||\bullet||^2_m \). 
    We denote the adjoint operator of $\dbar$ in $L^2_{p,q}(X_c\setminus Z,L,h_{c_m},\omega)$ by $\dbar^*_{m,\omega}$.
\end{definition}

We denote the $L^2$-inner product of $L^2_{n,0}(X,L,h)$ with respect to $h$ by $||\bullet||^2_h$, and write the restriction to a subset $U\subset X$ as $||\bullet||^2_{h,U}:=||\bullet||^2_h|_U$, then $||\bullet||^2_{h,X_c}=||\bullet||^2_0$.
Let $\scr{D}^{p,q}(X_c,L)$ be the space of $\cal{C}^\infty$-section of $\bigwedge^{p,q}T^*_{X_c}\otimes L$ with compact support on $X_c$, and 
let $\cal{D}^{p,q}_{\dbar}$ and $\cal{D}^{p,q}_{\dbar^*_{m,\omega}}$ be the domein of operators $\dbar$ and $\dbar^*_{m,\omega}$ from $L^2_{p,q}(X_c\setminus Z,L,h_{c_m},\omega)$, respectively.

\begin{lemma}\label{lemma uniform estimate}
    Let $m\geq1$ and bi-degree $(p,q)$ be chosen arbitrarily with $q\geq1$.
    Then we have the following estimate 
    \begin{align*}
        ||\phi||^2_{m,\omega}\leq\frac{1}{p+q-n}(||\dbar\phi||^2_{m,\omega}+||\dbar^*_{m,\omega}\phi||^2_{m,\omega})
    \end{align*}
    for any $\phi\in\cal{D}^{p,q}_{\dbar}\cap\cal{D}^{p,q}_{\dbar^*_{m,\omega}}\subset L^2_{p,q}(X_c\setminus Z,L,h_{c_m},\omega)$, if $p+q>n$.
\end{lemma}

\begin{proof}
    Since $\omega$ is \kah metric, the Bochner-Kodaira-Nakano identity is given by 
    \begin{align*}
        \Delta''_{m,\omega}=\Delta'_{m,\omega}+[\iO{L,h_{m_c}},\Lambda_\omega] 
    \end{align*}
    on $\cal{C}^{\infty}(X_c\setminus Z,\bigwedge^{p,q}T^*_{X_c}\otimes L)=:\cal{E}^{p,q}(X_c\setminus Z,L)$. A straightforward calculation yields 
    \begin{align*}
        \lara{[\iO{L,h_{m_c}},\Lambda_\omega]v}{v}_\omega\geq(p+q-n)|v|^2_\omega
    \end{align*}
    for any $v\in\cal{E}^{p,q}(X_c\setminus Z)$ if $p+q>n$. Therefore, we have the following estimate 
    \begin{align*}
        (p+q-n)||\phi||^2_{m,\omega}\leq\int_{X_c\setminus Z}\lara{[\iO{L,h_{m_c}},\Lambda_\omega]\phi}{\phi}_{h_{c_m},\omega}\,dV_\omega\leq ||\dbar\phi||^2_{m,\omega}+||\dbar^*_{m,\omega}\phi||^2_{m,\omega} \tag*{($\ast$)}
    \end{align*}
    for any $\phi\in\scr{D}^{p,q}(X_c\setminus Z,L)$. 

    Since $\omega$ is complete, the estimate $(\ast)$ still holds provided $\phi\in \cal{D}^{p,q}_{\dbar}\cap\cal{D}^{p,q}_{\dbar^*_{m,\omega}}$.
\end{proof}

\begin{remark}
    Using the well-known Demailly's approximation (= Theorem \ref{Dem appro preserves ideal sheaves}), this estimate $(\ast)$ can be shown up to $\scr{D}^{p,q}(X_c\setminus Z,L)$. 
    However, extending it to $\cal{D}^{p,q}_{\dbar}\cap\cal{D}^{p,q}_{\dbar^*_{m,\omega}}$ requires the completeness of \( \omega=\iO{L,h_c} \), necessitating Theorems \ref{Dem appro with log poles and ideal sheaves} and \ref{regularization with the equality of ideals and completeness}.
\end{remark}

\subsection{Approximation proposition concerning the $L^2$-norm and semi-norms}

In this section, the aim is to provide approximation propositions between each pair of sublevel sets in order to prove Approximation Theorem \ref{Approximation thm in Introduction}.

\begin{proposition}\label{Prop of approximation thm with L2-norm}
    $($\textnormal{Approximation property for the $L^2$-norm}$)$ 
    The restriction map 
    \begin{align*}
        \rho_{c,b}:H^0(X_c,K_X\otimes L\otimes\scr{I}(h))\longrightarrow H^0(X_b,K_X\otimes L\otimes\scr{I}(h))
    \end{align*}
    has a dense image where the space $H^0(X_b,K_X\otimes L\otimes\scr{I}(h))$ is endowed with the topology induced by the $L^2$-norm $||\bullet||_h$.
    In other words, if $\phi\in H^0(X_b,K_X\otimes L\otimes\scr{I}(h))$ then for any $\varepsilon>0$ there exists $\widetilde{\phi}\in H^0(X_c,K_X\otimes L\otimes\scr{I}(h))$ such that 
    $||\widetilde{\phi}-\phi||^2_{h,X_b}<\varepsilon$.
\end{proposition}

\begin{proof}
    If this proposition is not correct, then we have 
    \begin{align*}
        \overline{H^0(X_c,K_X\otimes L\otimes\scr{I}(h))}\subsetneq H^0(X_b,K_X\otimes L\otimes\scr{I}(h)),
    \end{align*}
    where the bar over $H^0(X_c,K_X\otimes L\otimes\scr{I}(h))$ denotes the closure of $H^0(X_c,K_X\otimes L\otimes\scr{I}(h))$ in the Hilbert space $L^2_{n,0}(X_b,L,h)=L^2_{n,0}(X_b\setminus Z,L,h_c)$.
    Hence, there exists a non-zero element $\eta\in\overline{H^0(X_b,K_X\otimes L\otimes\scr{I}(h))}$ to be orthonormal to $H^0(X_c,K_X\otimes L\otimes\scr{I}(h))$. 
    Namely, for any $\zeta\in H^0(X_c,K_X\otimes L\otimes\scr{I}(h))$ we have 
    \begin{align*}
        \lla\zeta,\eta\rra_{h,X_b}&=\int_{X_b\setminus Z}\lara{\zeta}{\eta}_{h,\omega}dV_\omega \\
        &=0.
    \end{align*}
    Under this assumption, we derive a contradiction by showing that $\eta$ is orthogonal to $H^0(X_b,K_X\otimes L\otimes\scr{I}(h))$.

    Extend the definition of $\eta$ by setting $\eta=0$ outside $X_b$.
    By $\psi\equiv0$ on $X_b$ and the assumption, the following 
    \begin{align*}
        \eta\in L^2_{n,0}(X_c,L,h_{c_m})\cap(\rom{ker}\,\dbar_c)^{\perp}
    \end{align*}
    is derived for each $m\geq0$. Then we obtain $\eta\in\overline{\rom{im}\,\dbar^*_{m,\omega}}\subset L^2_{n,0}(X_c,L,h_{c_m})$.
    From \cite[Lemma\,4.1.1]{Hor90}, Lemma \ref{lemma uniform estimate} yields the identity
    \begin{align*}
        \overline{\rom{im}\,\dbar^*_{m,\omega}}=\rom{im}\,\dbar^*_{m,\omega}. 
    \end{align*}
    Therefore, from Lemma \ref{lemma uniform estimate} and \cite[Lemma\,4.1.2]{Hor90}, we acquire $\eta=\dbar^*_{m,\omega}\phi_m$ for some $\phi_m\in L^2_{n,1}(X_c\setminus Z,L,h_{c_m},\omega)$ with estimate 
    \begin{align*}
        ||\phi_m||^2_{m,\omega}\leq||\eta||^2_m\leq||\eta||^2_0=||\eta||^2_h.
    \end{align*}
    By setting $w_m=\phi_m e^{-m\psi}$, we obtain $\dbar^*_{h_c,\omega}w_m=\dbar^*_{m,\omega}\phi_m=\eta$ and 
    \begin{align*}
        ||w_m||^2_{h_c,\omega}=||w_m||^2_{0,\omega}\leq||w_m||^2_{-m,\omega}=||\phi_m||^2_{m,\omega}\leq||\eta||^2_h.
    \end{align*}
    Therefore, $\{w_m\}_{m\in\bb{N}}$ has a subsequence that weakly converges in $L^2_{n,1}(X_c\setminus Z,L,h_c,\omega)$, and denote its weak limit by $w$.
    
    From the inequality $||w_m||^2_{-m,\omega}\leq||\eta||^2_h$, we have the inequality
    \begin{align*}
        e^{\varepsilon m}\int_{\{x\in X_c\mid\varepsilon<\psi\}}|w_m|^2_{h_c,\omega}dV_\omega\leq\int_{\{x\in X_c\mid\varepsilon<\psi\}}|w_m|^2_{h_c,\omega} e^{m\psi}dV_\omega\leq||\eta||^2_h,
    \end{align*}
    for any $\varepsilon>0$ and each $m\geq1$. 
    Thus, the integral of $|w_m|^2_{h_c,\omega}$ over $\{x\in X_c\mid\varepsilon<\psi\}$ tends to zero, implying that $w_m\to0$ almost everywhere on $\{x\in X_c\mid\varepsilon<\psi\}$.
    As a result, the weak limit $w=0$ on $\{x\in X_c\mid\varepsilon<\psi\}$ for any $\varepsilon>0$, and we obtain 
    \begin{align*}
        \rom{supp}\,w\subseteq\overline{X}_b \qquad\rom{and}\qquad \dbar^*_{h_c,\omega}w=\eta \quad \rom{on}\quad X_c.
    \end{align*}

    We take an arbitrary section $\xi\in H^0(X_b,K_X\otimes L\otimes\scr{I}(h))\cong L^2_{n,0}(X_b\setminus Z,L,h_c)\cap\rom{ker}\,\dbar$. 
    By Theorem \ref{regularization with the equality of ideals and completeness} (or \cite[Th\'eor\`em\,1.5]{Dem82}), the sublevel set $X_b\setminus Z$ carries a complete metric $\omega_b$.
    From the characterization of complete metrics (see \cite[ChapterVIII,\,Lemma\,2.4]{Dem-book}), 
    there exists an exhaustive sequence $\{K_\nu\}_{\nu\in\bb{N}}$ of compact subsets of $X_b\setminus Z$ ans functions $\rho_\nu\in\cal{C}^\infty(X_b\setminus Z,\bb{R})$ such that 
    \begin{align*}
        \rho_\nu\equiv1& \quad \rom{in\,\, a\,\, neighborhood\,\, of}\,\,K_\nu,\quad \\
        0\leq\rho_\nu\leq1&, \quad \rom{supp}\,\rho_\nu\subset K_{\nu+1}^{\circ} \quad \rom{and}\quad |d\rho_\nu|_{\omega_b}\leq2^{-\nu}.
    \end{align*}
    In particular, $\{\rho_\nu\xi\}_{\nu\in\bb{N}}$ converges to $\xi$ in $L^2_{n,0}(X_b\setminus Z,L,h_c,\omega_b)$. 
    From Lemma \ref{inequality of (n,0)-forms} and the condition of $\rho_\nu$, we obtain the following equation
    \begin{align*}
        \lla\eta,\xi\rra_{h,X_b}&=\int_{X_b\setminus Z}\lara{\eta}{\xi}_{h_c,\omega_b}dV_{\omega_b}=\lim_{\nu\to+\infty}\int_{X_b\setminus Z}\lara{\eta}{\rho_\nu\xi}_{h_c,\omega_b}dV_{\omega_b}\\
        &=\lim_{\nu\to+\infty}\int_{X_c}\lara{\eta}{\rho_\nu\xi}_{h_c,\omega}dV_\omega\\
        &=\lim_{\nu\to+\infty}\int_{X_c}\lara{\dbar^*_{h_c,\omega}w}{\rho_\nu\xi}_{h_c,\omega}dV_\omega\\
        &=\lim_{\nu\to+\infty}\int_{X_c}\lara{w}{\dbar(\rho_\nu\xi)}_{h_c,\omega}dV_\omega\\
        &=\lim_{\nu\to+\infty}\int_{X_b}\lara{w}{\dbar\rho_\nu\cdot\xi+\rho_\nu\dbar\xi}_{h_c,\omega_b}dV_{\omega_b}\\
        &=\lim_{\nu\to+\infty}\int_{X_b}\lara{w}{\dbar\rho_\nu\cdot\xi}_{h_c,\omega_b}dV_{\omega_b}\\
        &=0,
    \end{align*}
    which shows that $\eta$ is orthogonal to $H^0(X_b,K_X\otimes L\otimes\scr{I}(h))$. 
\end{proof}

\begin{definition}
    Let $\gamma$ be a Hermitian metric on $X$ and $h_0$ be any smooth Hermitian metric of $L$ on $X$. 
    For a compact subset $K$ in $X_c$, we put 
    \begin{align*}
        |\phi|_K:=\sup_{x\in K}|\phi|_{h_0,\gamma}(x)
    \end{align*}
    for $\phi\in H^0(X_c,K_X\otimes L)$. Then $\{|\bullet|_K\}_{K}$ gives a system of semi-norms in $H^0(X_c,K_X\otimes L)$.
\end{definition}

\begin{proposition}\label{Prop of approximation thm with semi-norms}
    $($\textnormal{Approximation property for semi-norms}$)$ 
    The restriction map
    \begin{align*}
        \rho_{c,b}:H^0(X_c,K_X\otimes L\otimes\scr{I}(h))\longrightarrow H^0(X_b,K_X\otimes L\otimes\scr{I}(h))
    \end{align*}
    has dense image where the space $H^0(X_b,K_X\otimes L\otimes\scr{I}(h))$ is endowed with the topology of uniform convergence on all compact subsets in $X_b$.
    That is, for any compact subset $K$ in $X_b$, any positive number $\varepsilon>0$ and any $\phi\in H^0(X_b,K_X\otimes L\otimes\scr{I}(h))$, 
    there exists $\widetilde{\phi}\in H^0(X_c,K_X\otimes L\otimes\scr{I}(h))$ such that $|\widetilde{\phi}-\phi|_K<\varepsilon$.
\end{proposition}

\begin{proof}
    Fix an element of $\phi$ in $H^0(X_b,K_X\otimes L\otimes\scr{I}(h))$. Note that $\phi$ is zero on $Z$.
    We can find a positive constant $M_b$ satisfying $h_0\leq M_bh_c$ on $X_b$, here $M_b$ depends on $X_b$. 
    Using Cauchy's integral formula in each local coordinate $(U_j,z^{(j)}_1,\ldots,z^{(j)}_n)$ with $U_j\cap K\ne\emptyset$, for some constant $M_j>0$ we have 
    \begin{align*}
        |\phi|^2_{U_j\cap K}&\leq M_j\int_{U_j\cap K}|\phi|^2dV_{z^{(j)}}=M_j\int_{U_j\cap K}|\phi|^2_\omega dV_\omega\\
        &\leq\frac{M_j}{\min_K h_0}\int_{U_j\cap K}|\phi|^2_{\omega,h_0} dV_\omega\\
        &\leq\frac{M_jM_b}{\min_K h_0}\int_{U_j\cap K}|\phi|^2_{\omega,h} dV_\omega\\
        &=\frac{M_jM_b}{\min_K h_0}||\phi||^2_{h,U_j\cap K},
    \end{align*}
    from Lemma \ref{inequality of (n,0)-forms}.
    Hence, there exists a positive constant $M$ depends on $X_b$ such that $|\phi|_K\leq M||\phi||_{h,X_b}$.
    By this inequality and Proposition \ref{Prop of approximation thm with L2-norm}, the proposition is proven.
\end{proof}

\subsection{Proof of Theorem \ref{Approximation thm in Introduction}}\label{subsection: Proof of thm 4.1}

\begin{proof}[Proof of Theorem \ref{Approximation thm of hol section with ideal sheaves}]
    By Propositions \ref{Prop of approximation thm with L2-norm} and \ref{Prop of approximation thm with semi-norms}, the restriction map 
    \begin{align*}
        \rho_k:H^0(X_{c+k+1},K_X\otimes L\otimes\scr{I}(h))\longrightarrow H^0(X_{c+k},K_X\otimes L\otimes\scr{I}(h))
    \end{align*}
    has dence image with the topology induced by the $L^2$-norm $||\bullet||_h$ and the topology of uniform convergence on all compact subsets in $X_{c+k}$ for every $k=0,1,2,\ldots$. 
    Hence, for any compact subset $K$ in $X_c$, any $\varepsilon>0$ and any $\phi\in H^0(X_c,K_X\otimes L\otimes\scr{I}(h))$, we can find a sequence $\phi_k\in H^0(X_{c+k},K_X\otimes L\otimes\scr{I}(h))$ such that 
    \begin{align*}
        |\phi_1-\phi|_K<\frac{\varepsilon}{2}&, \quad\quad |\phi_{k+1}-\phi_k|_{\overline{X}_{c+k-1}}<\frac{\varepsilon}{2^k},\\
        ||\phi_1-\phi||_{h,X_c}<\frac{\varepsilon}{2}& \quad \rom{and} \quad ||\phi_{k+1}-\phi_k||_{h,X_{c+k}}<\frac{\varepsilon}{2^k}.
    \end{align*}

    We define the section $\widetilde{\phi}$ of $K_X\otimes L$ on $X$ by
    \begin{align*}
        \widetilde{\phi}:=\phi_1+\sum^\infty_{k=1}(\phi_{k+1}-\phi_k)=\phi_\mu+\sum^\infty_{k=\mu}(\phi_{k+1}-\phi_k).
    \end{align*}
    For each $\mu\in\bb{N}\cup\{0\}$, by the following inequality
    \begin{align*}
        |\widetilde{\phi}-\phi_{\mu+1}|_{\overline{X}_{c+\mu}}\leq\sum^\infty_{k=\mu+1}|\phi_{k+1}-\phi_k|_{\overline{X}_{c+\mu}}<\sum^\infty_{k=\mu+1}\frac{\varepsilon}{2^k}=\frac{\varepsilon}{2^\mu}
    \end{align*}
    on $\overline{X}_{c+\mu}$, the holomorphic sequence $\{\phi_{c+\mu+k}\}_{k\in\bb{N}}$ converges to $\widetilde{\phi}$ uniformly on all compact subset of $X_{c+\mu}$.
    Hence, the section $\widetilde{\phi}$ is holomorphic on each $X_{c+\mu}$, and we obtain $\widetilde{\phi}\in H^0(X,K_X\otimes L)$. 
    Furthermore, we derive the density $|\widetilde{\phi}-\phi|_K<\varepsilon$ by setting $\mu=1$.

    Similarly, for each $\mu\in\bb{N}$, from inequality 
    \begin{align*}
        ||\widetilde{\phi}-\phi_\mu||_{h,X_{c+\mu}}\leq\sum_{k=\mu}^{\infty}||\phi_{k+1}-\phi_k||_{h,X_{c+\mu}}<\sum^\infty_{k=\mu}\frac{\varepsilon}{2^k}=\frac{\varepsilon}{2^{\mu-1}},
    \end{align*}
    we obtain the $L^2$-inequality
    \begin{align*}
        \biggl(\int_{X_{c+\mu}}|\widetilde{\phi}|^2_{h,\gamma_X}dV_{\gamma_X}\biggr)^{\frac{1}{2}}=||\widetilde{\phi}||_{h,X_{c+\mu}}<||\phi_\mu||_{h,X_{c+\mu}}+\frac{\varepsilon}{2^{\mu-1}}<+\infty,
    \end{align*}
    for some Hermitian metric $\gamma_X$ on $X$. 
    Therefore, we have $\widetilde{\phi}\in K_{X,x}\otimes L_x\otimes\scr{I}(h)_x$ for any point $x$ in each sublevel set $X_{c+\mu}$, and thus $\widetilde{\phi}\in H^0(X,K_X\otimes L\otimes\scr{I}(h))$ is obtained.
    Finally, we obtain $||\widetilde{\phi}-\phi||_{h,X_c}<3\varepsilon/2$ representing density with respect to the $L^2$-norm similar to the semi-norm.
\end{proof}

\begin{corollary}\label{Corollary equation of dim H0}
    Let $(X,\varPsi)$ be a weakly pseudoconvex manifold and $L$ be a holomorphic line bundle on $X$ equipped with a singular Hermitian metric $h$. 
    If $h$ is singular positive, then for any sublevel set $X_c$, we have 
    \begin{align*}
        \rom{dim}\,H^0(X,K_X\otimes L\otimes\scr{I}(h))=\rom{dim}\,H^0(X_c,K_X\otimes L\otimes\scr{I}(h)).
    \end{align*}
    Furthermore, if $\rom{dim}\,H^0(X_c,K_X\otimes L\otimes\scr{I}(h))<+\infty$, then the following restriction map 
    \begin{align*}
        \rho_c:H^0(X,K_X\otimes L\otimes\scr{I}(h))\longrightarrow H^0(X_c,K_X\otimes L\otimes\scr{I}(h)).
    \end{align*}
    is an isomorphism, i.e., each section on $X_c$ can be uniquely extended.
\end{corollary}

\begin{proof}
    By the Identity theorem, the linear map $\rho_c$ is injective. 
    If $\rom{dim}\,H^0(X_c,K_X\otimes L\otimes\scr{I}(h))<+\infty$, then $\rom{Im}\,\rho_c$ is a finite-dimension subspace, hence closed; together with density, this implies that $\rom{Im}\,\rho_c$ coincides with the whole $H^0(X_c,K_X\otimes L\otimes\scr{I}(h))$.
    Therefore, $\rho_c$ is surjective and hence an isomorphism. 

    If $\rom{dim}\,H^0(X_c,K_X\otimes L\otimes\scr{I}(h))=+\infty$, then any dense subspace contained in it is also infinite-dimensional, which implies $\rom{dim}\,H^0(X,K_X\otimes L\otimes\scr{I}(h))=\rom{dim}\,\rom{Im}\,\rho_c=+\infty$.
\end{proof}

\section{Moishezon-ness and singular holomorphic Morse inequality}\label{Section 5: Moishezon-ness and Morse inequality}

Moishezon manifolds are characterized in terms of big line bundles, i.e., integral \kah currents, as first shown by Ji-Shiffman \cite{JS93}.
This characterization was first obtained for smooth metrics by resolving the Grauert-Riemenschneider conjecture, which is characterized by quasi-positive analytic sheaves. 
This conjecture was proven by Siu \cite{Siu85} and Demailly \cite{Dem85} using the Riemann-Roch type theorem and Demailly's holomorphic Morse inequality, respectively.
After that, by applying Demailly's approximation (see \cite{Dem92}) and Bonavero's singular holomorphic Morse inequalities (see \cite{Bon98}), a more precise characterization of Moishezon manifolds in terms of big line bundles, i.e., integral \kah currents, was established (see \cite{Bon98,Tak94}, \cite[Theorem\,2.3.28\,and\,2.3.30]{MM07}).
These discussions pertain to compact manifolds, revealing a relationship between Moishezon-ness and holomorphic Morse inequalities. 
In this section, the aim is to extend these concepts to non-compact manifolds.

\subsection{The generalization of Moishezon manifolds}

Let $\cal{M}_X$ be the sheaf of meromorphic functions.
It is straightforward to see that $\cal{M}_X(U)$ is a field if $U$ is connected.
If $X$ is connected, then $\scr{M}(X):=\Gamma(X,\cal{M}_X)$ is called the \textit{function field} of $X$. 
Clearly, there is a natural inclusion $\cal{O}_X\hookrightarrow\cal{M}_X$.

\begin{definition}[{cf.\,\cite[Definition\,2.2.8]{MM07}}]
    Let $X$ be a compact connected complex manifold of dimension $n$.
    We say that $f_1,\ldots,f_k\in\scr{M}(X)$ are \textit{algebraically dependent}, if there exists a non-trivial polynomial $P\in\bb{C}[z_1,\ldots,z_k]$ such that $P(f_1,\ldots,f_k)=0$ wherever it is defined.
    The \textit{transcendence degree} of $\scr{M}(X)$ is the maximal number of algebraically independent meromorphic functions on $X$, denoted by $a(X)$, and this is called
    the \textit{algebraic dimension} of $X$, i.e., $a(X):=\rom{tr.deg.}_{\bb{C}}\scr{M}(X)$.
    
    Furthermore, regardless of the connectedness of $X$, we say that $f_1,\ldots,f_k\in\scr{M}(X)$ are \textit{analytically dependent} if $df_1\wedge\cdots\wedge df_k\ne0$ for any point $x\in X$ where all $f_1,\ldots,f_k$ are holomorphic.
\end{definition}

\begin{theorem}[{cf.\,\cite[Theorem\,2.2.9]{MM07}}]\label{Thm [Theorem 2.2.9, MM07]}
    Let $X$ be a compact connected complex manifold. The functions $f_1,\ldots,f_k$ $\in\scr{M}(X)$ are analytically dependent if and only if they are algebraically dependent.
\end{theorem}

As is well known, a compact complex manifold $X$ is called \textit{Moishezon} if 
$a(X_j)=\rom{dim}\,X$ for each connected component $X_j$ of $X$.
As a generalization of Moishezon-ness to non-compact cases, a generalization to Andreotti-pseudoconcave manifolds using the Siegel-type lemma is known (see \cite{Mar96}, \cite[Chapter\,3]{MM07}).
In this paper, we define Moishezon-ness for a broader class of manifolds using Theorem \ref{Thm [Theorem 2.2.9, MM07]} as a generalization.

\begin{definition}
    We say that a complex manifold $X$ is generalized Moishezon if each connected component possesses $\rom{dim}\,X$ analytically independent meromorphic functions.
\end{definition}

Similarly to \cite[Theorem\,2.2.15]{MM07}, we can obtain the following characterization even in the non-compact (connected) case.

\begin{theorem}\label{Ext [MM07, Theorem 2.2.15]}
    A complex manifold $X$ is generalized Moishezon if and only if it carries a big line bundle.
\end{theorem}

\begin{proof}
    If $X$ is Moishezon, then there exist $n:=\rom{dim}\,X$ analytically independent meromorphic functions $f_1,\ldots,f_n\in\scr{M}(X)$. 
    Thus, we can find a holomorphic line bundle $L$ such that these functions have the form $f_j=s_j/s_0$ with $s_0\not\equiv0$ where $s_0,\ldots,s_n\in H^0(X,L)$ (see \cite{And63}).
    Here, $d(s_1/s_0)\wedge\cdots\wedge d(s_n/s_0)\ne0$ on the set where the left-hand side is defined. Let $V$ be a sublinear space of $H^0(X,L)$ spanned by $s_0, \dots, s_n$, then the Kodaira map $\varPhi_V:X\dashrightarrow \bb{P}(V)$ has maximal rank, and hence $L$ is big.

    Conversely, if $L$ is big, there exist $p>0$ and $s_0,\ldots,s_n \in H^0(X,L^{\otimes p})$ 
    such that $d(s_1/s_0)\wedge\cdots\wedge d(s_n/s_0)=0$ outside a nowhere dense analytic set. 
    This means that $s_1/s_0,\ldots,s_n/s_0$ are $n$ analytically independent.
\end{proof}

As an important and ideal example, we provide generalized Moishezon-ness and embeddings with respect to each relatively compact weakly pseudoconvex sublevel set $X_c$.

\begin{theorem}\label{Thm bigness with ideal sheaves on each sublevel sets}
    Let $(X,\varPsi)$ be an $n$-dimensional weakly pseudoconvex manifold and $L$ be a holomorphic line bundle on $X$ with a singular Hermitian metric $h$.
    If $h$ is singular positive, then for any $c<\sup_X\varPsi$ there exist a proper modification $\pi_c:\widetilde{X}_c\longrightarrow X_c$, a singular Hermitian metric $h_c$ on $L|_{X_c}$ with $\scr{I}(h)=\scr{I}(h_c)$ and a large integer $p_c$ 
    such that $K_{X_c}\otimes L^{\otimes p_c}\otimes\scr{I}(h_c^{p_c})$ is big. 
    Furthermore, there exist a very ample line bundle $A\longrightarrow\widetilde{X}_c$ and an effective divisor $E$ on $\widetilde{X}_c$ with only simple normal crossing such that 
    \begin{align*}
        K_{X_c}^{\otimes n+2}\otimes L^{\otimes p_c(n+2)}\otimes\scr{I}(h_c^{p_c(n+2)})&={\pi_c}_*(A+E),\\
        \pi_c^*(K_{X_c}^{\otimes n+2}\otimes L^{\otimes p_c(n+2)})\otimes\scr{I}(\pi_c^*h_c^{p_c(n+2)})&=A+E.
    \end{align*}

    In particular, each sublevel set $X_c$ is generalized Moishezon and $X_c\setminus Z_c$ is holomorphically embeddable into $\bb{P}^{2n+1}$, where $Z_c$ is an analytic subset of $X_c$ obtained as the singlular locus of $h_c$. 
\end{theorem}

Here, the second equation corresponds to Fujita's approximation (see \cite[II,\,Section\,11.4]{Laz04}) 
for an open manifold, and it can be written as $\scr{I}(\pi_c^*h_c^{p_c(n+2)})=\cal{O}_{\widetilde{X}_c}(-D)$ by an effective divisor $D$.
Moreover, using any smooth Hermitian metric \( h_{K_{\!X}} \) on \( K_X \), the above two formulas can be easily written as follows in terms of the \( L^2 \)-subsheaf.
\begin{align*}
    \scr{L}^2(h_{K_{\!X}}^{n+2}\otimes h_c^{p_c(n+2)})={\pi_c}_*(A+E), \quad \scr{L}^2(\pi_c^*(h_{K_{\!X}}^{n+2}\otimes h_c^{p_c(n+2)}))=A+E.
\end{align*}

\begin{proof}
    We use the notation from Setting \ref{Setting of each sublevel set}. Here, $K_{\widetilde{X}_c}=\pi^*_cK_{X_c}\otimes\cal{O}_{\widetilde{X}_c}(\sum^J_{j=1}\zeta_j D_j)$ for some non-negative integers $\zeta_j$. 
    Let $D_\zeta:=\sum^J_{j=1}\zeta_j D_j$ and $D_\eta:=\sum^J_{j=1}\eta_jD_j$ be simple normal crossing divisors on $\widetilde{X}_c$ with $\eta_j:=\max\{1,\zeta_j\}$. By the proof of Theorem \ref{Blow ups of Dem appro with log poles and ideal sheaves}, there exists a positive integer $t_\eta\gg0$ such that 
    the holomorphic line bundle $\scr{L}_\eta:=\widetilde{L}_{m_c}^{\otimes t_\eta}\otimes\cal{O}_{\widetilde{X}_c}(-D_\eta)$ is also positive. For any integer $q>n(n+1)/2$, $K_{\widetilde{X}_c}\otimes\scr{L}_\eta^{\otimes q}$ is ample and there exists $\sigma_j\in H^0(\widetilde{X}_c,(K_{\widetilde{X}_c}\otimes\scr{L}_\eta^{\otimes q})^{\otimes(n+2)})$ $(0\leq j\leq2n+1)$
    that induces the holomorphic embedding map $\widetilde{\varphi}:\widetilde{X}_c\longrightarrow\bb{P}^{2n+1}$ (see \cite[Theorem\,1.2]{Tak98}). In other words, $K_{\widetilde{X}_c}^{\otimes n+2}\otimes\scr{L}_\eta^{\otimes q(n+2)}$ is very ample. From $(\beta)$ of Theorem \ref{Blow ups of Dem appro with log poles and ideal sheaves}, we have 
    \begin{align*}
        H^0(X_c,&K_{X_c}^{\otimes n+2}\otimes L^{\otimes p(n+2)}\otimes\scr{I}(h_c^{p(n+2)}))\\
        &\cong H^0(\widetilde{X}_c,K_{\widetilde{X}_c}^{\otimes n+2}\otimes\cal{O}_{\widetilde{X}_c}(-(n+1)D_\zeta)\otimes\pi_c^*L^{\otimes p(n+2)}\otimes\scr{I}(\pi^*_ch_c^{p(n+2)})),
    \end{align*}
    for any $p\geq m_c$. Here, $\scr{L}_\eta^{\otimes q}=\pi^*_cL^{\otimes m_ct_\eta q}\otimes\scr{I}(\pi_c^*h_c^{m_ct_\eta q})\otimes\cal{O}_{\widetilde{X}_c}(-qD_\eta)$. 
    Define the effective divisor by $D_q:=q(n+2)D_\eta-(n+1)D_\zeta$. By taking $p=m_ct_\eta q$, we obtain 
    \begin{align*}
        \pi_c^*(K_{X_c}^{\otimes n+2}\otimes L^{\otimes p(n+2)})\otimes\scr{I}(\pi_c^*h_c^{p(n+2)})&=K_{\widetilde{X}_c}^{\otimes n+2}\otimes\scr{L}_\eta^{\otimes q(n+2)}\otimes\cal{O}_{\widetilde{X}_c}(D_q),\\
        H^0(X_c,K_{X_c}^{\otimes n+2}\otimes L^{\otimes p(n+2)}\otimes\scr{I}(h_c^{p(n+2)}))
        &\cong H^0(\widetilde{X}_c,K_{\widetilde{X}_c}^{\otimes n+2}\otimes\scr{L}_\eta^{\otimes q(n+2)}\otimes\cal{O}_{\widetilde{X}_c}(D_q)).
    \end{align*}
    By taking the canonical section \( s_q \) of \( \cal{O}_{\widetilde{X}_c}(D_q) \), we construct a Kodaira map \( \widetilde{\varPhi} \) from sections \( \widetilde{\sigma}_j:=\sigma_j\otimes s_q\in H^0(\widetilde{X}_c,K_{\widetilde{X}_c}^{\otimes n+2}\otimes\scr{L}_\eta^{\otimes q(n+2)}\otimes\cal{O}_{\widetilde{X}_c}(D_q)) \). 
    From the isomorphism given by \( \pi_c \), there exist sections \( \tau_j\in H^0(X_c,K_{X_c}^{\otimes n+2}\otimes L^{\otimes p(n+2)}\otimes\scr{I}(h_c^{p(n+2)})) \) satisfying $\widetilde{\sigma}_j=\pi_c^*\tau_j$, and let $\varPhi$ be a Kodaira map induced by $\tau_j$. 
    It follows that $\pi_c^*\varPhi=\widetilde{\varPhi}$ is satisfied, and the restriction to $X_c\setminus Z_c$ is clearly an embedding. Hence, $K_{X_c}\otimes L^{\otimes p}\otimes\scr{I}(h_c^p)$ is big. 
    Furthermore, 
     we obtain $K_{X_c}^{\otimes n+2}\otimes L^{\otimes p(n+2)}\otimes\scr{I}(h_c^{p(n+2)})={\pi_c}_*(K_{\widetilde{X}_c}^{\otimes n+2}\otimes\scr{L}_\eta^{\otimes q(n+2)}\otimes\cal{O}_{\widetilde{X}_c}(D_q))$ (see \cite[Proposition\,5.8]{Dem12}).
\end{proof}

Here, Theorem \ref{Thm bigness with ideal sheaves on each sublevel sets} does not require the entire \( X \) to be weakly pseudoconvex, but it holds for weakly pseudoconvex relatively compact subsets.

\begin{conjecture}
    Similar to the compact case, is the existence of integral \kah currents, i.e., singular positive line bundles, 
    equivalent to generalized Moishezon-ness for relatively compact weakly pseudoconvex subsets?
    In other words, can the existence of a big line bundle imply the existence of integral \kah currents?
\end{conjecture}

This seems to hinge on the extension of Moishezon's theorem (see \cite{Moi66}, \cite[Theorem\,2.2.16]{MM07}) to non-compact settings.
Furthermore, if the codimension of the locus where the rank of a Kodaira map induced by the big line bundle is not maximal is at least $2$, 
then the existence of integral \kah currents can be established in a manner analogous to the compact case by using blow-ups and Lemma \ref{lemma positivity of exc div to blow ups}.

\subsection{Singular holomorphic Morse inequality}\label{subsection: singular Morse inequality}

The Bonavero's singular holomorphic Morse inequality (see \cite[Th\'eor\`em\,1.1]{Bon98}) obtained on compact manifolds is a generalization of Demailly's holomorphic Morse inequality (see \cite[Th\'eor\`em\,0.1]{Dem85}) for singular Hermitian metrics with only algebraic singularities.
In this subsection, we generalize the Bonavero's Morse inequality to the non-compact case for singular positive Hermitian metrics by using blow-ups Theorem \ref{Blow ups of Dem appro with log poles and ideal sheaves} and Approximation Theorem \ref{Approximation thm in Introduction}.

Let $L$ be a holomorphic line bundle on an $n$-dimensional complex manifold $X$, and let $V$ be an open subset of $X$ and $\hbar$ be a singular Hermitian metric on $L|_V$ with only analytic singularities.
Let $Z_\hbar$ be the singular locus of $\hbar$ which is an analytic subset of $V$. We introduce the $q$-index set of $(L,\hbar)$ as
\begin{align*}
    V(q)=V(q,\hbar)=\left\{ x\in V\setminus Z_\hbar \, \middle\vert \, \iO{L,\hbar}(x) \,\, \text{ has } \,\,\, 
    \begin{matrix}
        q & \text{ negative eigenvalues} \\
        n-q & \text{ positive eigenvalues} 
     \end{matrix}
    \right\}
\end{align*}
for $0\leq q\leq n$, and we set $V(\leq q):=\bigcup^q_{j=0}V(j)$.

\begin{theorem}\label{Generalized singular holomorphic Morse inequality}
    Let $X$ be a complex manifold of dimension $n$ and $L$ be a holomorphic line bundle on $X$ with a singular Hermitian metric $h$. 
    Let $U$ be a relatively compact open subset 
    represented as $U = \{ x\in X\mid \rho(x)< 0 \}$ using a smooth defining function $\rho\in\cal{C}^\infty(X,\bb{R})$, 
    and $E$ be a holomorphic vector bundle on $U$. 
    If $h$ is singular positive on a neighborhood of $U$, then there exist an open subset $V$ of $U$ represented as $V=\{\rho<-\delta\}$ using a sufficiently small $\delta>0$, 
    a singular positive Hermitian metric $h_V$ on $L|_{V}$ and a integer $t_V\in\bb{N}$, and we have the following singular holomorphic Morse inequality
    \begin{align*}
        \rom{dim}\,H^0(V,E\otimes L^{\otimes p}\otimes\scr{I}(h_V^p))\geq\rom{rank}\,E\cdot\frac{p^n}{t_V^n n!}\int_{V(\leq1)}\Bigl(\frac{i}{2\pi}\Theta_{L,h_V}\Bigr)^n+o(p^n).
    \end{align*}
    
    Here, the metric $h_V$ is satisfying $\scr{I}(h)=\scr{I}(h_V)$ on $V$, is smooth on $V\setminus Z$ where $Z$ is an analytic subset of $V$, and has algebraic singularities which give $Z$.    
\end{theorem}

\begin{proof}
    Let $\omega$ be a Hermitian metric on $X$. By blow-ups Theorem \ref{Blow ups of Dem appro with log poles and ideal sheaves}, there exists a singular metric $h_U$ and a proper modification $\pi_U:\widetilde{U}\longrightarrow U$ that satisfy the appropriate conditions.
    Then, $\widetilde{U}=\{\pi_U^*\rho<0\}$, and by Sard's theorem, there exists a nowhere dense subset $\Sigma\subset(-\infty,0)$ such that if \( -\delta \in(-\infty,0)\setminus \Sigma \), then \( \widetilde{V}_{\delta}:=\{y\in\widetilde{U}\mid\pi^*_U\rho(y)<-\delta\} \) has a smooth boundary \( \partial\widetilde{V}_\delta \).
    Fix a sufficiently small \( \delta > 0 \) and let \( \widetilde{V} := \widetilde{V}_\delta \). Define the open set $V:=\{x\in X\mid\rho(x)<-\delta\}$, then $\widetilde{V}=\pi^{-1}_U(V)$.
    Let $\pi:\widetilde{V}\longrightarrow V$ be the restriction of $\pi_U$ to $\widetilde{V}$, and let $h_V:=h_U|_V$. These satisfy the following appropriate conditions induced from blow-ups Theorem \ref{Blow ups of Dem appro with log poles and ideal sheaves}.
    \begin{itemize}
        \item [$(a)$] $\scr{I}(h)=\scr{I}(h_V)$ on $V$, $\iO{L,h_V}\geq\varepsilon\omega$ on $V$ for some small $\varepsilon>0$ and $h_V$ is smooth on $V\setminus Z$, where $Z$ is an analytic subset of $V$.
        \item [$(b)$] $\pi:\widetilde{V}\setminus\pi^{-1}(Z)\longrightarrow V\setminus Z$ is biholomorphic and there exist $m_V\in\bb{N}$ and a simple normal crossing divisor $D_V=\sum^J_{j=1}a_jD_j$ with $\rom{supp}\,D_V=\pi^{-1}(Z)$ such that $\scr{I}(\pi^*h_V^q)=\cal{O}_{\widetilde{V}}(-\sum^J_{j=1}\lfloor\frac{qa_j}{m_V}\rfloor)$ for any $q\in\bb{N}$.
        \item [$(c)$] the holomorphic line bundle $\widetilde{L}_{m_V}:=\pi^*L^{\otimes m_V}\otimes\cal{O}_{\widetilde{V}}(-D_V)$ has a smooth Hermitian metric $\widetilde{h}_{m_V}$ satisfying $\iO{\widetilde{L}_{m_V},\widetilde{h}_{m_V}}\geq m_V\varepsilon\pi^*\omega$ on $\widetilde{V}$.
        \item [$(d)$] there exist an effective divisor $D_b=\sum^J_{j=1}b_jD_j$, a smooth Hermitian metric $h^*_b$ on the corresponding line bundle $\cal{O}_{\widetilde{V}}(-D_b)$ and an integer $t_V'\gg0$ such that the smooth Hermitian metric $\widetilde{h}_{m_V}^{\otimes t_V'}\otimes h^*_b$ on $\widetilde{L}_{m_V}^{\otimes t_V'}\otimes\cal{O}_{\widetilde{V}}(-D_b)$ is positive on $\widetilde{V}$.
        \item [$(e)$] if $F$ is a holomorphic vector bundle, then for any $q\geq m_V$ we have 
        \begin{align*}
            H^0(V,K_V\otimes F\otimes L^{\otimes q}\otimes\scr{I}(h_V^q))\cong H^0(\widetilde{V},K_{\widetilde{V}}\otimes\pi^*(F\otimes L^{\otimes q})\otimes\scr{I}(\pi^*h_V^q)).
        \end{align*}
    \end{itemize}
    We write $t_V:=t_V'+1$ and $p=m_Vt_Vq+r$, where $q,r\in\bb{N}$ and $0\leq r<m_Vt_V$. Let 
    \begin{align*}
        \widetilde{E}_r:=&\pi^*E\otimes K_{\widetilde{V}}\otimes \pi^*K_V^{-1}\otimes\scr{I}(\pi^*h_V^r)=\pi^*E\otimes K_{\widetilde{V}/V}\otimes\cal{O}_{\widetilde{V}}\Bigl(-\sum^J_{j=1}\lfloor\frac{ra_j}{m_V}\rfloor D_j\Bigr), \\
        \scr{L}:=&\widetilde{L}_{m_V}^{\otimes t_V}\otimes\cal{O}_{\widetilde{V}}(-D_b)=\pi^*L^{\otimes t_Vm_V}\otimes\cal{O}_{\widetilde{V}}(-t_VD_V-D_b).
    \end{align*}
    Through the inclusion map $0\longrightarrow \cal{O}_{\widetilde{V}}(-qD_b)\longrightarrow\cal{O}_{\widetilde{V}}$ associated with the effective divisor $qD_b$ twisted by $\widetilde{E}_r\otimes\widetilde{L}_{m_V}^{\otimes \,t_Vq}$, we obtain 
    \begin{align*}
        \rom{dim}\,H^0(\widetilde{V},\widetilde{E}_r\otimes\widetilde{L}_{m_V}^{\otimes t_Vq})\geq\rom{dim}\,H^0(\widetilde{V},\widetilde{E}_r\otimes\scr{L}^{\otimes q}).
    \end{align*}

    Since $\scr{L}$ has a smooth Hermitian metric $\widetilde{h}_{m_V}^{\otimes \,t_V}\otimes h^*_b$ with the positive curvature and $\widetilde{V}$ has the smooth boundary $\partial\widetilde{V}$ (which is necessary for Dirichlet boundary condition), we obtain the following Morse inequality (see \cite[Theorem\,3.2.16]{MM07})
    \begin{align*}
        \rom{dim}\,H^0(\widetilde{V},\widetilde{E}_r\otimes\scr{L}^{\otimes q})\geq\rom{rank}\,\widetilde{E}_r\cdot\frac{q^n}{n!}\int_{\widetilde{V}(\leq1)}\Bigl(\frac{i}{2\pi}\Theta_{\scr{L},\widetilde{h}_{m_V}^{\otimes t_V}\otimes h^*_b}\Bigr)^n+o(q^n).
    \end{align*}
    The curvature inequality $\iO{\scr{L},\widetilde{h}_{m_V}^{\otimes t_V}\otimes h^*_b}>\iO{\widetilde{L}_{m_V},\widetilde{h}_{m_V}}\geq0$ on $\widetilde{V}$ leads to the inequality
    \begin{align*}
        \int_{\widetilde{V}(\leq1)}\Bigl(\frac{i}{2\pi}\Theta_{\scr{L},\widetilde{h}_{m_V}^{\otimes t_V}\otimes h^*_b}\Bigr)^n&>\int_{\widetilde{V}(\leq1)}\Bigl(\frac{i}{2\pi}\Theta_{\widetilde{L}_{m_V},\widetilde{h}_{m_V}}\Bigr)^n=\int_{\widetilde{V}(\leq1)\setminus D_V}\Bigl(\frac{i}{2\pi}\Theta_{\widetilde{L}_{m_V},\widetilde{h}_{m_V}}\Bigr)^n\\
        &=m_V^n\int_{\widetilde{V}(\leq1)\setminus D_V}\Bigl(\frac{i}{2\pi}\Theta_{\pi^*L,\pi^*h_V}\Bigr)^n\\
        &=m_V^n\int_{V(\leq1)}\Bigl(\frac{i}{2\pi}\Theta_{L,h_V}\Bigr)^n\\
        &>0,
    \end{align*}
    where $\iO{\widetilde{L}_{m_V},\widetilde{h}_{m_V}}=m_V\iO{\pi^*L,\pi^*h_V}$ on $\widetilde{V}(\leq1)\setminus D_V$ which is biholomorphic to $V(\leq1, h_V)$. 
    
    From the previous discussion and $(e)$, we can obtain the desired Morse inequality
    \begin{align*}
        \rom{dim}\,H^0(V,E\otimes L^{\otimes p}\otimes\scr{I}(h_V^p))
        &=\rom{dim}\,H^0(\widetilde{V},\widetilde{E}_r\otimes\widetilde{L}_{m_V}^{\otimes t_Vq})\\
        &\geq\rom{dim}\,H^0(\widetilde{V},\widetilde{E}_r\otimes\scr{L}^{\otimes q})\\
        &\geq\rom{rank}\,\widetilde{E}_r\cdot\frac{q^n}{n!}\int_{\widetilde{V}(\leq1)}\Bigl(\frac{i}{2\pi}\Theta_{\scr{L},\widetilde{h}_{m_V}^{\otimes t_V}\otimes h^*_b}\Bigr)^n+o(q^n)\\
        &>\rom{rank}\,E\cdot\frac{m_V^nq^n}{n!}\int_{V(\leq1)}\Bigl(\frac{i}{2\pi}\Theta_{L,h_V}\Bigr)^n+o(q^n)\\
        &=\rom{rank}\,E\cdot\frac{p^n}{t_V^n n!}\int_{V(\leq1)}\Bigl(\frac{i}{2\pi}\Theta_{L,h_V}\Bigr)^n+o(p^n),
    \end{align*}
    where $\rom{rank}\,E=\rom{rank}\,\widetilde{E}_r$.
\end{proof}

\begin{theorem}\label{singular holomorphic Morse inequality on w.p.c}
    Let $(X,\varPsi)$ be a weakly pseudoconvex manifold of dimension $n$, $E$ be a holomorphic vector bundle on $X$ and $L$ be a holomorphic line bundle on $X$ equipped with a singular positive Hermitian metric $h$. 
    There exists a nowhere dence subset $\Sigma\subset(\inf_X\varPsi,\sup_X\varPsi)$ such that for any $c\in (\inf_X\varPsi,\sup_X\varPsi)\setminus\Sigma$, 
    there exist a singular positive Hermitian metric $h_c$ on $L|_{X_c}$ obtained from Theorem \ref{Dem appro with log poles and ideal sheaves}, satisfying $\scr{I}(h)=\scr{I}(h_c)$ and a positive integer $t_c\in\bb{N}$, 
    and we obtain the singular holomorphic Morse inequality
    \begin{align*}
        \rom{dim}\,H^0(X_c,E\otimes L^{\otimes p}\otimes\scr{I}(h_c^p))\geq\rom{rank}\,E\cdot\frac{p^n}{t_c^n n!}\int_{X_c(\leq1)}\Bigl(\frac{i}{2\pi}\Theta_{L,h_c}\Bigr)^n+o(p^n).
    \end{align*}
    
    Furthermore, if there exists $c>\inf_X\varPsi$ such that $E_{+}(h)\subset X_c$, i.e., $E_{+}(h)$ is compact, or that $E_{+}(h)\bigcap X_c=\emptyset$, then we have a global singular holomorphic Morse inequality
    \begin{align*}
        \rom{dim}\,H^0(X,K_X\otimes L^{\otimes p}\otimes\scr{I}(\widetilde{h}^p))\geq\frac{p^n}{t_c^n n!}\int_{X_c(\leq1)}\Bigl(\frac{i}{2\pi}\Theta_{L,h_c}\Bigr)^n+o(p^n),
    \end{align*}
    where $\widetilde{h}$ is a singular positive Hermitian metric on $L$ and coincides with $h_c$ on $X_c$.
\end{theorem}

\begin{proof}
    The first claim follows immediately from Theorem \ref{Generalized singular holomorphic Morse inequality}.

    We consider the case where $E_{+}(h)\subset X_c$. We have $\scr{I}(h^p)=\cal{O}_X$ on $X\setminus\overline{X}_c$ for any $p\in\bb{N}$. In other words, $h$ has no essential singularities on $X\setminus\overline{X}_c$.
    Then, with respect to the weight function of $h$, by Richberg's regularization theorem, i.e., Lemma \ref{Lemma of Richberg's regularization theorem}, there exists a smooth Hermitian metric $\widehat{h}$ on $X\setminus\overline{X_c}$ such that both the loss in the Hessian form and the difference from $h$ can be made arbitrarily small. 
    Here, this metric $\widehat{h}$ is also positive on $X\setminus\overline{X_c}$.
    From the construction of $h_c$, it is defined on $X_{c+3\delta}$ for some small $\delta>0$, such that both the difference from the weight function of $h$ and the loss in the Hessian form are arbitrarily small on $X_{c+2\delta}\setminus \overline{X_{c+\delta}}$.
    As in the proof of Richberg's regularization theorem, by using the regularized max function with respect to the weight functions, $\widehat{h}$ and $h_c$ can be smoothly glued together while keeping the loss of curvature positivity arbitrarily small.
    Denote this glued metric by $\widetilde{h}$, then $\widetilde{h}$ is smooth on $X \setminus Z_c$, singular positive on $X$, and coincides with $h_c$ on $X_c$. 
    Hence, we obtains $\scr{I}(h_c^p) = \scr{I}(\widetilde{h}^p)$ on $X_c$ and $\scr{I}(h^p) = \scr{I}(\widetilde{h}^p)=\cal{O}_X$ on $X\setminus X_c$ for any $p\in\bb{N}$.

    If $E_{+}(h)\bigcap X_c=\emptyset$, then Proposition \ref{Prop E0(h)=cup Zj} implies that $Z_c=\emptyset$ and $\scr{I}(h_c^p)=\scr{I}(h^p)=\cal{O}_X$ on $X_c$ for any $p\in\bb{N}$.
    Similarly, $h_c$ can be glued with $h$ on $X_{c+\delta} \setminus \overline{X_c}$, and this glued metric is denoted by $\widetilde{h}$.
    Therefore, in both of these situations, Approximation Theorem \ref{Approximation thm of hol section with ideal sheaves} and Corollary \ref{Corollary equation of dim H0} imply that 
    \begin{align*}
        \rom{dim}\,H^0(X,K_X\otimes L^{\otimes p}\otimes\scr{I}(\widetilde{h}^p))=\rom{dim}\,H^0(X_c,K_X\otimes L^{\otimes p}\otimes\scr{I}(h_c^p))
    \end{align*}
    for any $p\in\bb{N}$, and together with the first claim, the proof is complete. 
\end{proof}

Let $h$ be a singular Hermitian metric on $L$.
For an open subset \( V\subseteq X \), the volume of $L$ with the multiplier ideal sheaf \( \scr{I}(h) \) is defined as follows.
\begin{align*}
    \rom{Vol}_V(L\otimes\scr{I}(h)):=n!\cdot\liminf_{k\to+\infty}\frac{\rom{dim}\,H^0(V,L^{\otimes k}\otimes\scr{I}(h^{\otimes k}))}{k^n}.
\end{align*}
It is clear that $\rom{Vol}_V(L)\geq\rom{Vol}_V(L\otimes\scr{I}(h))$. 
Here, we know that if \( X \) is compact and $L$ is either nef and big or semi-positive and singular positive, then \( \rom{Vol}_X(L)=c_1(L)^n \).

Immediately as a corollary, the following estimate for the volume follows.

\begin{corollary}
    Under the same assumptions as in Theorem \ref{singular holomorphic Morse inequality on w.p.c}, take the same nowhere dence subset \( \Sigma \).
    For any $c\in (\inf_X\varPsi,\sup_X\varPsi)\setminus\Sigma$, there exist a singular positive Hermitian metric $h_c$ on $L|_{X_c}$ obtained from Theorem \ref{Dem appro with log poles and ideal sheaves}, satisfying $\scr{I}(h)=\scr{I}(h_c)$ on $X_c$ and a positive integer $t_c\in\bb{N}$ such that the following estimate is obtained.
    \begin{align*}
        \rom{Vol}_{X_c}(L\otimes\scr{I}(h_c))\geq\frac{1}{t_c^n}\int_{X_c(\leq1)}\Bigl(\frac{i}{2\pi}\Theta_{L,h_c}\Bigr)^n.
    \end{align*}
    In particular, by adding a suitably exhaustive weight function to \( h_c \), the value on the right-hand side can be made infinite.
\end{corollary}

\section{Vanishing and Non-vanishing theorems and volumes on subvarieties}\label{Section: vanishing and non-vanishing}

This section provides results such as (non)-vanishing theorems in preparation for the subsequent section.
We will need the following Serre-Nakano type vanishing theorem.

\begin{theorem}[{cf.\,\cite[Theorem\,N']{Fuj75}}]\label{Vanishing thm for positive line bdl}
    Let $(X,\varPsi)$ be a weakly pseudoconvex manifold with a positive holomorphic line bundle $L$.
    Then for any coherent analytic sheaf $\scr{F}$ on $X$ and any sublevel set $X_c$, there exists a positive integer $m_c$ such that the following cohomolog groups vanishing
    \begin{align*}
        H^q(X_c,\scr{F}\otimes L^{\otimes m})=0
    \end{align*}
    for any $q\geq1$ and any $m\geq m_c$.
\end{theorem}

In general, if $L$ is only singular positive, it does not imply that \( X \) carries a \kah metric. 
As a key result for considering points separation and embeddings without assuming the existence of a \kah metric, 
we present the following Nadel-type vanishing theorem that does not require a \kah metric.

\begin{theorem}\label{Thm Nadel type vanishing without Kahler}$($\textnormal{Nadel-type vanishing}$)$
    Let $(X,\varPsi)$ be a weakly pseudoconvex manifold and $L$ be a holomorphic line bundle on $X$. 
    We take an arbitrary number $c<\sup_X\varPsi$. If there exist $\delta>0$ and a singular Hermitian metric $h$ on $L|_{X_{c+\delta}}$ which is singular positive, then we have the following cohomology vanishing 
    \begin{align*}
        H^1(X_c,K_X\otimes L\otimes\scr{I}(h))=0.
    \end{align*}
\end{theorem}

\begin{proof}
    By Theorem \ref{regularization with the equality of ideals and completeness}, there exist a Hermitian metric $\gamma$ and a singular Hermitian metric $\hbar$ on $L|_{X_c}$ such that the following conditions are satisfied.
    \begin{itemize}
        \item $\hbar$ is smooth on $X_c\setminus Z_c$, where $Z_c$ is an analytic subset of $X_c$.
        \item $\scr{I}(h)=\scr{I}(\hbar)$ on $X_c$.
        \item $\iO{L,\hbar}$ is a complete \kah metric on $X_c\setminus Z_c$, and $\iO{L,\hbar}\geq\gamma$ on $X_c$ as currents.
    \end{itemize}
    For the sheaf $\scr{L}^{n,q}_{L,\hbar}$ defined in $\S\ref{subsection: Construction and Setting in Approximation Theorem}$, it is known that the isomorphism 
    \begin{align*}
        H^q(X_c,K_X\otimes L\otimes\scr{I}(\hbar))\cong H^q(\Gamma(X_c,\scr{L}^{n,\bullet}_{L,\hbar}))
    \end{align*}
    holds for any $q\geq1$ (see \cite[Theorem\,5.3]{Wat23}).
    We take any smooth increasing convex function $\chi:\bb{R}\longrightarrow\bb{R}$, and let $\psi:=\frac{1}{c-\varPsi}$ and $\hbar_\chi:=\hbar e^{-\chi\circ\psi}$. 
    
    For any $\dbar$-closed $f\in \Gamma(X_c,\scr{L}^{n,1}_{L,\hbar})$, the integral 
    \begin{align*}
        \int_{X_c}|f|^2_{\hbar_\chi,\gamma} dV_{\gamma}=\int_{X_c}|f|^2_{\hbar,\gamma}e^{-\chi\circ\psi}dV_{\gamma}
    \end{align*}
    becomes convergent if $\chi$ grows fast enough. Define the metric $\omega:=\iO{L,\hbar_\chi}=\iO{L,\hbar}+\idd\chi\circ\psi$ on $X_c\setminus Z_c$, then this metric $\omega$ is also complete \kah and satisfies $\omega\geq\gamma$ on $X_c\setminus Z_c$.
    Using Lemma \ref{inequality of (n,0)-forms}, we have 
    \begin{align*}
        \int_{X_c\setminus Z_c}|f|^2_{\hbar_\chi,\omega}dV_\omega\leq\int_{X_c\setminus Z_c}|f|^2_{\hbar_\chi,\gamma}dV_\gamma=\int_{X_c}|f|^2_{\hbar_\chi,\gamma} dV_{\gamma}<+\infty.
    \end{align*}
    By H\"ormander's $L^2$-estimate (cf. \cite[ChapterVIII,\,Theorem\,4.5]{Dem-book}), we obtain a solution $u\in L^2_{n,0}(X_c\setminus Z_c,L,\hbar_\chi,\omega)$ of $\dbar u=f$ on $X_c\setminus Z_c$ satisfying 
    \begin{align*}
        \int_{X_c\setminus Z_c}|u|^2_{\hbar_\chi,\omega}dV_\omega\leq\int_{X_c\setminus Z_c}\lara{[\iO{L,\hbar_\chi},\Lambda_\omega]^{-1}f}{f}_{\hbar_\chi,\omega}dV_\omega=\int_{X_c\setminus Z_c}|f|^2_{\hbar_\chi,\omega}dV_\omega.
    \end{align*}
    The $\dbar$-equation extends over $X_c$, meaning $\dbar u=f$ on $X_c$ (cf. \cite[Lemma\,6.9]{Dem82}).
    By using Lemma \ref{inequality of (n,0)-forms}, it follows that $u\in L^2_{n,0}(X_c,L,\hbar_\chi,\gamma)$.
    This means $u\in \Gamma(X_c,\scr{L}^{n,0}_{L,\hbar})$, and the vanishing $H^1(X_c,K_X\otimes L\otimes \scr{I}(\hbar))=0$ is obtained.
\end{proof}

By applying similar methods as in \cite[Theorem\,2-1-1]{KMM87} and \cite[Lemm\,3.7]{Tak98}, we obtain the following non-vanishing theorem.

\begin{theorem}\label{Non-vanishing thm}$($\textnormal{Non-vanishing}$)$
    Let $(X,\varPsi)$ be a weakly pseudoconvex manifold of dimension $n$ and $L$ be a holomorphic line bundle on $X$ equipped with a singular positive Hermitian metric $h$.
    We take the singular positive Hermitian metric $h_c$ satisfying $\scr{I}(h)=\scr{I}(h_c)$ obtained from blow-ups Theorem \ref{Blow ups of Dem appro with log poles and ideal sheaves}.
    Let $x_1,\cdots,x_r$ be $r$ distinct points on a sublevel set $X_c$. Let $p$ and $q$ be positive integers satisfying the inequality $p^n\rom{Vol}_{X_c}(L\otimes\scr{I}(h_c))>rq^n$.
    Then there exist a positive integer $m$ and a holomorphic section $\tau\in H^0(X_c,L^{\otimes mp}\otimes\fra{m}_{x_1}^{mq}\cdots\fra{m}_{x_r}^{mq})$
    such that $\rom{dim}\,\rom{div}\,\tau<n$. 
\end{theorem}

\begin{proof}
    By Remark \ref{Remark of Setting on each sublevel set} and the canonical restriction map, 
    we can obtain 
    \begin{align*}
        \rom{Vol}_{X_{c-\varepsilon}}(L\otimes\scr{I}(h_c))\geq\rom{Vol}_{X_c}(L\otimes\scr{I}(h_c))\geq\rom{Vol}_{X_{c+\varepsilon}}(L\otimes\scr{I}(h_c))
    \end{align*}
    for any \( 0<\varepsilon<\varepsilon_c \). We take the nowhere dense subset \( \Sigma\subset(-\infty,\sup_X\varPsi) \) from within Theorem \ref{singular holomorphic Morse inequality on w.p.c}.
    Then there exists a sufficiently small \( 0<\varepsilon<\varepsilon_c \) such that $c+\varepsilon\notin\Sigma$ and $p^n\rom{Vol}_{X_{c+\varepsilon}}(L\otimes\scr{I}(h_c))>rq^n$.
    To verify this theorem, we will count the dimensions of the following exact sequence
    \begin{align*}
        0\longrightarrow H^0(X_{c+\varepsilon},L^{\otimes mp}\otimes\fra{m}^{mq}_{x_1}\cdots\fra{m}^{mq}_{x_r})&\\
        \longrightarrow H^0(X_{c+\varepsilon},L^{\otimes mp})\longrightarrow &L^{\otimes mp}\otimes\cal{O}_X/\fra{m}^{mq}_{x_1}\cdots\fra{m}^{mq}_{x_r},
    \end{align*}
    where $\cal{O}_X/\fra{m}^{mq}_{x_1}\cdots\fra{m}^{mq}_{x_r}$ is a skyscraper sheaf of rank 
    \begin{align*}
        r\begin{pmatrix}
            mq+n-1\\
            n
        \end{pmatrix}
        =\frac{rq^n}{n!}m^n+o(m^n).
    \end{align*}
    On the other hand, we have
    \begin{align*}
        \liminf_{m\to+\infty}\frac{\rom{dim}\,H^0(X_{c+\varepsilon},L^{\otimes mp})}{m^n}
        \geq\frac{p^n}{n!}\rom{Vol}_{X_{c+\varepsilon}}(L\otimes\scr{I}(h_c)).
    \end{align*}
    By the inequality $p^n\rom{Vol}_{X_{c+\varepsilon}}(L\otimes\scr{I}(h_c))>rq^n$, there exist a positive integer $m$ and a section $\sigma\in H^0(X_{c+\varepsilon},L^{\otimes mp}\otimes\fra{m}_{x_1}^{\otimes mq}\cdots\fra{m}_{x_r}^{\otimes mq})$ 
    such that $\rom{dim}\,(\rom{div}\,\sigma\cap X_c)<n$. Hence, we get the desired section $\tau:=\sigma|_{X_c}$.
\end{proof}

We introduce the concept of volume on subvarieties.
Let 
$Y$ be a $k$-dimensional closed subvariety of $X$. 
By a \textit{subvariety}, we mean a reduced and irreducible complex subspace.
Before that if $L$ is positive, we define formal intersection numbers as follows \cite{Tak98}:  
\begin{align*}
    (L^k\cdot Y):=
    \begin{cases}
        c_1(\mu^*(L|_Y))^k & \rom{if}\,\,Y\,\,\text{is compact}, \\
        +\infty                 & \rom{otherwise},
      \end{cases}
\end{align*}
where $\mu:\widetilde{Y}\longrightarrow Y$ is a desingularization.
We can see that this definition is well-defined, i.e., it does not depend on desingularizations (if $Y$ is compact).

We use the notation from Setting \ref{Setting of each sublevel set} for $X_c$, where $c<\sup_X\varPhi$. 
Let $(L,h)$ be a singular positive line bundle on $X$ and $Y$ be a $k$-dimensional $(0<k<n)$ closed subvariety of a sublevel set $X_c$ satisfying $Y\not\subset Z_c$.
The strict transform of $Y$ by $\pi_c$ is denoted by $Y_{st}$.
We formally define the volume of $(L,h_c)$ over $Y$ as
\begin{align*}
    \rom{Vol}_Y(L\otimes\scr{I}(h_c)):=&\rom{Vol}_{Y_{st}}(\pi_c^\ast L\otimes\scr{I}(\pi_c^* h_c))\\
    :=&\frac{1}{m_c^k}\rom{Vol}_{Y_{st}}(\pi_c^\ast L^{\otimes m_c}\otimes\cal{O}_{\widetilde{X}_c}(-D_c))\\
    =&\frac{1}{m_c^k}(\widetilde{L}_{m_c}^k\cdot Y_{st}),
\end{align*}
if $Y$ is compact, and $\rom{Vol}_Y(L\otimes\scr{I}(h_c))=+\infty$ if otherwise. 
Here, $\widetilde{L}_{m_c}:=\pi_c^\ast L^{\otimes m_c}\otimes\cal{O}_{\widetilde{X}_c}(-\widetilde{D}_c)$ is nef and big on $Y_{st}$ by $(\ref{subsection: canonical sHm}\,b)$. 
Clearly, the inequality $\rom{Vol}_Y(L\otimes\scr{I}(h_c))\leq(\pi_c^*L^k\cdot Y_{st})$ holds, and if $Y\cap Z_c=\emptyset$ then equality is achieved, resulting in $\rom{Vol}_Y(L\otimes\scr{I}(h_c))=\rom{Vol}_{Y_{st}}(\pi^*_cL)$, and $\pi_c^*L$ is ample on $Y_{st}$.
Additionally, we define 
\begin{align*}
    \widetilde{\rom{Vol}}_Y(L\otimes\scr{I}(h_c)):=m_c^k\cdot\rom{Vol}_Y(L\otimes\scr{I}(h_c))=(\widetilde{L}_{m_c}^k\cdot Y_{st}),
\end{align*}
ensuring that its value is an integer, i.e., $\widetilde{\rom{Vol}}_Y(L\otimes\scr{I}(h_c))\in\bb{N}$.

Furthermore, with respect to the singular metric $h_\natural$ constructed via approximation of $h$, the volume of $(L,h_\natural)$ over a closed subvariety $Y\subset X$ is defined by 
\begin{align*}
    \rom{Vol}_Y(L\otimes\scr{I}(h_\natural)):=\min\{\rom{Vol}_Y(L\otimes\scr{I}(h_c))\mid c\in\bb{N} \text{ with } Y\subset X_c \text{ and } Y\not\subset Z_c\},
\end{align*}
if $Y$ is compact, and $\rom{Vol}_Y(L\otimes\scr{I}(h_\natural))=+\infty$ if otherwise.
Similarly, we also define $\widetilde{\rom{Vol}}_Y(L\otimes\scr{I}(h_\natural))$ by using $\widetilde{\rom{Vol}}_Y(L\otimes\scr{I}(h_c))$, here $\widetilde{\rom{Vol}}_Y(L\otimes\scr{I}(h_\natural))\in\bb{N}\cup\{+\infty\}$.

By using \cite[Lemma\,3.10]{Tak98}, which is a localization of \cite[Lemma\,4.1]{AS95}, we obtain the following similarity.

\begin{lemma}\label{Ext [Lemma 3.10, Tak98]: see [Lemma 4.1, AS95]}
    Let $Y$ be a $k$-dimensional $(0<k<n)$ closed subvariety of $X_{c+\varepsilon}$ passing through $r$ distinct points $x_1,\ldots,x_r$ on $X_c\setminus Z_c$ with $0<\varepsilon\ll1$.
    Let $\Delta'_j$ be a local holomorphic curve on $Y$ passing through $x_j$ such that $\rom{Sing}\,\Delta'_j\subset\{x_j\}$, $\Delta'_j\cap\rom{Sing}\,Y\subset\{x_j\}$ and 
    that the normalization $\sigma_j:\Delta\longrightarrow\Delta'_j$ is an injective holomorphic map from the open unit disk $\Delta\subset\bb{C}$ with $\sigma_j(0)=x_j$ $(1\leq j\leq r)$. 
    Let $p$ and $q$ be positive integers $(p>2)$ such that $(p-2)^k\widetilde{\rom{Vol}}_Y(L\otimes\scr{I}(h_c))>rq^k$. 
    Then there exist a positive integer $m$ and holomorphic sections $\tau_1,\cdots,\tau_N\in H^0(X_c,L^{\otimes mp}\otimes\scr{I}(h_c^{mp}))$ such that 
    \begin{align}
        \{x_1,\cdots,x_r\}\subset\bigcap^N_{\ell=1}\tau_{\ell}^{-1}(0)\subsetneq (Y\cup Z_c)\cap X_c \quad\rom{and}\quad\rom{dim}\,\bigcap^N_{\ell=1}\tau_{\ell}^{-1}(0)|_Y<k,
    \end{align}
    $(2)$ after replacing $\Delta$ with a smaller disk if necessary, there exist a small coordinate neighborhood $B_j$ centered at $x_j$ in $X_c\setminus Z_c$, which is biholomorphic to the unit ball with $\sigma_j(\Delta)\subset Y_j:=Y\cap B_j\,(1\leq j\leq r)$
    and local holomorphic sections $\tau_\ell(j)\in H^0(B_j\times\Delta,pr^*_j L^{\otimes mp})$ $(1\leq j\leq r,\,1\leq \ell\leq N)$, where $pr_j:B_j\times\Delta\longrightarrow B_j$ is the first projection, such that $\tau_\ell(j)|_{B_j\times0}=\tau_\ell|_{B_j}$ for any $j$ and $\ell$, and that 
    \begin{align*}
        \tau_\ell(j)|_{Y_j\times u}\in H^0(Y_j\times u,pr^*_jL^{\otimes mp}\otimes\fra{m}^{mq}_{Y\times u,\sigma_j(u)\times u})
    \end{align*}
    for any $u\in\Delta\setminus\{0\}$ and for any $j$ and $\ell$.
\end{lemma}

\begin{proof}
    We use the notation from Setting \ref{Setting of each sublevel set} for $X_c$.
    Here, $\pi_c:\widetilde{X}_{c+\varepsilon}\longrightarrow X_{c+\varepsilon}$, $\widetilde{L}_{m_c}:=\pi_c^*L^{\otimes m_c}\otimes\cal{O}_{\widetilde{X}_{c+\varepsilon}}(-D_c)$ and $\scr{L}:=\widetilde{L}_{m_c}^{\otimes t}\otimes\cal{O}_{\widetilde{X}_{c+\varepsilon}}(-D_b)$ for some integer $t\geq t_c$ by Remark \ref{Remark of Setting on each sublevel set}.
    Let $\widetilde{x}_j:=\pi_c^{-1}(x_j)$, $\Delta''_j:=\pi^{-1}_c(\Delta'_j)$ and $\widetilde{\sigma}_j:\Delta\longrightarrow\Delta''_j$ by $\widetilde{\sigma}_j:=\sigma_j\circ\pi_c^{-1}$.
 
    Let $\mu:\widehat{Y}_{st}\longrightarrow Y_{st}$ be a desingularization. 
    By taking $t$ sufficiently large, we obtain
    \begin{align*}
        (p-1)^k(\scr{L}^k\cdot Y_{st})
        &
        =(p-1)^k c_1(\mu^*(\widetilde{L}_{m_c}^{\otimes t}\otimes\cal{O}_{\widetilde{X}_{c+\varepsilon}}(-\widetilde{D}_b))\otimes\cal{O}_{\widehat{Y}_{st}})^k\\
        &>(p-2)^k t^k c_1(\mu^*\widetilde{L}_{m_c})^k=(p-2)^kt^k(\widetilde{L}_{m_c}^k\cdot Y_{st})\\
        &\geq(p-2)^k m_c^k (\widetilde{L}_{m_c}^k\cdot Y_{st})=(p-2)^k \widetilde{\rom{Vol}}_Y(L\otimes\scr{I}(h_c)).
    \end{align*}

    Here, $\scr{L}$ is a positive line bundle on $\widetilde{X}_{c+\varepsilon}$. By \cite[Lemma\,3.10]{Tak98}, there exist a positive integer $\widetilde{m}$ and holomorphic sections $\widetilde{\tau}_1,\ldots,\widetilde{\tau}_N\in H^0(X_c,\scr{L}^{\otimes \widetilde{m}p})$ such that 
    \begin{align*}
        \{\widetilde{x}_1,\ldots,\widetilde{x}_N\}\subset\bigcap^N_{\ell=1}\widetilde{\tau}_\ell^{-1}(0)\subsetneq Y_{st}\cap\widetilde{X}_c \quad \text{and} \quad \rom{dim}\,\bigcap^N_{\ell=1}\widetilde{\tau}_\ell^{-1}(0)<k,
    \end{align*}
    and that these sections have the property of $(2)$ in \cite[Lemma\,3.10]{Tak98}.
    Define the effective divisor by $D_t:=tD_c+D_b$. By taking the canonical section $s_t$ of $\cal{O}_{\widetilde{X}_c}(D_t)$, the sections $\widetilde{\tau}_\ell s_t^{\widetilde{m}p}\in H^0(\widetilde{X}_c,\pi^*_cL^{\otimes m_c\widetilde{m}p})$ descend to sections $\tau_j \in H^0(X_c,L^{\otimes mp})$, where $m=m_c\widetilde{m}$.
    These sections $\{\tau_\ell\}^N_{\ell=1}$ clearly satisfy the condition $(1)$, and since $\pi_c$ is biholomorphic on a neighborhood of each $x_j$, the condition $(2)$ is also satisfied.
\end{proof}

\section{Points separation theorem on each sublevel set}\label{Section 7: Points separation on Xc}

In this section, we construct singular Hermitian metrics with properties related to points separation of adjoint bundles, based on \cite[Section\,4]{Tak98}.
For any integer \( c\in\bb{N} \), let \( h_c \) and \( Z_c \) be the singular positive Hermitian metric and the analytic subset, respectively, as in Theorem \ref{Thm characterizations of canonical sHm} (or Section \ref{Setting of each sublevel set}).

\subsection{Singular Hermitian metrics inducing points separation}\label{subsection: sHm inducing points separation}

First, we introduce several notations to establish the following theorem, which is the main result of this section.
For every $x\in X_c$ and for any positive integer $d$ with $1\leq d\leq n-1$, we set 
\begin{align*}
    V_{X_c,d}(x):=&\{Y\mid d\text{-dimensional} \,\,\textit{compact}\,\,\text{subvariety of } X_c \text{ passing through } x\},\\
    I_d(L,h_c,x):=&
    \begin{cases}
        \inf\{\widetilde{\rom{Vol}}_Y(L\otimes\scr{I}(h_c))^{1/d}\mid Y\in V_{X_c,d}(x)\} & \rom{if}\,\,V_{X_c,d}(x)\,\,\text{is non-empty}, \\
        +\infty                 & \rom{otherwise},
      \end{cases}
\end{align*}
and $I_n(L,h_c,x):=\rom{Vol}_{X_c}(L\otimes\scr{I}(h_c))^{1/n}$. 
Furthermore, for $r$ distinct points $x_1,\ldots,x_r$ on $X_c$, we set $I_d(L,h_c,\{x_j\}^r_{j=1}):=\min_{1\leq j\leq r} I_d(L,h_c,x_j)$ and 
\begin{align*}
    I_{\sum}(L,h_c,\{x_j\}^r_{j=1}):=&\sum^n_{d=1}\frac{d\sqrt[d]{r}}{I_d(L,h_c,\{x_j\}^r_{j=1})},\\
    m_0(L,h_c,\{x_j\}^r_{j=1}):=&\min\{m\in\bb{N}\mid m>I_{\sum}(L,h_c,\{x_j\}^r_{j=1})\}.
\end{align*}

\begin{theorem}\label{Ext [Tak98, Thm 4.1]}
    Let $x_1,\ldots,x_r$ be $r$ distinct points on a subset $X_c\setminus Z_c$. Then for any positive integer $m>I_{\sum}(L,h_c,\{x_j\}^r_{j=1})$,
    there exist a positive rational number $\vartheta_c\in\bb{Q}_{>0}$ smaller than $m_0(L,h_c,\{x_j\}^r_{j=1})$ and a singular positive Hermitian metric $H_c$ of $L|_{X_c}^{\otimes m}$ 
    such that $\scr{I}(H_c)\subset\scr{I}(h_c^{m-\vartheta_c})$ and $\{x_j\}^r_{j=1}\subset V(\scr{I}(H_c))$, and that some point $x_j$ is isolated in $V(\scr{I}(H_c))$.
\end{theorem}

We take pairs of positive integers $\{(P_d,Q_d)\}^n_{d=1}$ such that 
\begin{align*}
    P_d>2\quad \text{and} \quad (P_d-2)^d I_d(L,h_c,\{x_j\}^r_{j=1})^d>rQ^d_d.
\end{align*}
Note that the pairs $\{(P_d,Q_d)\}^n_{d=1}$ can be chosen large enough such that 
\begin{align*}
    m>I_{\sum}(L,h_c,\{x_j\}^r_{j=1})=\sum^n_{d=1}\frac{d\sqrt[d]{r}}{I_d(L,h_c,\{x_j\}^r_{j=1})} \quad \text{implies} \quad m>\sum^n_{d=1}\frac{dP_d}{Q_d}.
\end{align*}
Using such a pair $\{(P_d,Q_d)\}^n_{d=1}$, Theorem \ref{Ext [Tak98, Thm 4.1]} can be proven.

\vspace*{4mm}

The desired singular metric is derived by appropriately constructing multiple regular sections. 
Let us first examine the \textit{basic construction}. We take a real number $c'\in(c,c+\varepsilon_c)$. 
Theorem \ref{Non-vanishing thm} can also be applied to $\rom{Vol}_{X_c}(L\otimes\scr{I}(h_c))$ by Remark \ref{Remark of Setting on each sublevel set} and its proof.
Thus, there exist a positive integer $m$ and a section $\tau\in H^0(X_{c'},L^{\otimes mP_n}\otimes\fra{m}^{mQ_n}_{x_1}\cdots\fra{m}^{mQ_n}_{x_r})$ such that $\rom{dim}\,\tau^{-1}(0)<n$. We set as in $\S\ref{subsection: multivalued section and natural sHm}$ 
\begin{align*}
    \scr{I}(t):=&\scr{I}\Big((|\tau|^2)^{-t/(mQ_n)}\Big) \quad\text{for any }\quad t\geq0,\\
    \alpha(x_j):=&\sup\{t\geq0\mid\scr{I}(t)_{x_j}=\cal{O}_{X,x_j}\},
    \quad \alpha:=\max_{1\leq j\leq r} \alpha(x_j).
\end{align*}
We have $\{x_j\}^r_{j=1}\subset V(\scr{I}(\alpha))\subset \tau^{-1}(0)$ and an effective estimate $0<\alpha\leq n$.
If some point $x_j$ is isolated in $V(\scr{I}(\alpha))$, we  get the desired singular Hermitian metric
\begin{align*}
    H_c:=h_0^{m-\alpha P_n/Q_n}\frac{h_0^{\alpha P_n/Q_n}}{|s|^2_{h_0^{\alpha P_n/Q_n}}}e^{-2(m-\alpha P_n/Q_n)\varphi_c}=h_c^{m-\alpha P_n/Q_n}\frac{1}{|s|^2}
\end{align*}
for any integer $m>\alpha P_n/Q_n$, where $s:=\tau^{\alpha/(mQ_n)}$ is a multivalued holomorphic section of $L^{\alpha P_n/Q_n}$.

\vspace*{4mm}

For any $d$ with $0\leq d\leq n-1$, we set $w_d=\sum^n_{j=d+1}jP_j/Q_j\in\bb{Q}_{>0}$.
We take real numbers $0<c(0)<c(1)<\cdots<c(n)<c+\varepsilon_c$ and rational numbers $0<\varepsilon(n)<\varepsilon(n-1)<\cdots<\varepsilon(0)<m_0-w_0$, where $m_0$ is the smallest integer then $w_0$.
Here, the pair $\{(P_d,Q_d)\}^n_{d=1}$ can be chosen such that $m_0=m_0(L,h_c,\{x_j\}^r_{j=1})$.

\vspace*{4mm}

\noindent
\textbf{Induction Statement $\mathbf{A_d}$}. 
There exist a real number $c_d\in (c(d),c(d+1))$, a rational number $\varepsilon_d\in(\varepsilon(d+1),\varepsilon(d))$, a rational number $\eta_d\leq w_d$, a finite number of multivalued holomorphic sections $s_{d,1},\ldots,s_{d,k_d}$ of $L^{\otimes \eta_d+\varepsilon_d}$ on $X_{c_d}$
and a non-empty subset $J_d$ of $\{1,\ldots,r\}$ such that the following four conditions are satisfied. Let 
\begin{align*}
    \scr{I}_d(t):=\scr{I}\Big(\big(\sum^{k_d}_{j=1}|s_{d,j}|^2\big)^{-t}\Big) \quad\text{and}\quad Y_d(t):=V(\scr{I}_d(t))
\end{align*}
for any $t\geq0$, then
\begin{itemize}
    \item [$(i)$] $x_j\in Y_d(1)$ for any $j\in J_d$,
    \item [$(ii)$] $x_j\notin Y_d(t)$ for any $j\in J_d$ and any $t<1$,
    \item [$(iii)$] there exists a rational number $\beta_d\in (0,1)$ such that $x_j\in Y_d(t)$ for any $j\in\{1,\ldots,r\}\setminus J_d$ and any $t\geq\beta_d$,
    \item [$(iv)$] the dimension of $Y_d(1)$ at $x_j$ is at most $d$ for any $j\in J_d$.
\end{itemize}

\vspace*{4mm}

Let $\alpha_d(x_j):=\sup\{t\geq0\mid\scr{I}_d(t)=\cal{O}_{X,x_j}\}$ for any $j$ and $\alpha_d:=\max_{1\leq j\leq r}\alpha_d(x_j)$. 
Then we remark that the conditions $\mathbf{A_d}(i)$ and $\mathbf{A_d}(ii)$ are equivalent to $\alpha_d=1$ and $J_d=\{j\in\{1,\cdots,r\}\mid\alpha_d(x_j)=\alpha_d\}$.
Thus, it suffices to choose the rational number \( \beta_d \) such that \(\max\{\alpha_d(x_j)\mid\alpha_d(x_j)<1\}<\beta_d<1\).

\begin{lemma}\label{Lem basic construcion An-1}
    The first statement $\mathbf{A_{n-1}}$ is verified by the above basic construcion.
\end{lemma}

\begin{proof}
    We retake a real number $c'=c_{n-1}\in (c(n-1),c(n))$ and take a rational number $\varepsilon_{n-1}\in(\varepsilon(n),\varepsilon(n-1))$.
    Set $J_{n-1}:=\{j\in\{1,\ldots,r\}\mid\alpha(x_j)=\alpha\}$ for $\alpha$ within the basic construcion.
    We use the notation from Setting \ref{Setting of each sublevel set} for $X_{c_{n-1}}$. 
    By the vanishing theorem (= Theorem \ref{Vanishing thm for positive line bdl}) for the positive line bundle $\scr{L}$ on $\widetilde{X}_{c(n)}$, there exist positive integers $\ell$ and $K$ and holomorphic sections $\{s_j\}^K_{j=1}$ of $\scr{L}^{\otimes \ell}$ on without common zeros.
    For some effective divisor $D$ with $\pi_{c(n)}(D)=Z_c$ and positive integer $t$, we can write $\scr{L}=\pi^*_{c(n)}L^{\otimes t}\otimes\cal{O}_{\widetilde{X}_{c(n)}}(-D)$.
    Let $s_D$ be a canonical section of $\cal{O}_{\widetilde{X}_{c(n)}}(D)$, then the sections $s_j\otimes s_D^\ell\in H^0(\widetilde{X}_{c(n)},\pi^*_{c(n)}L^{\otimes t\ell})$ descend to sections $S_j\in H^0(X_{c(n)},L^{\otimes t\ell})$. 
    We let $S_j:=s_j^{1/t\ell}$ $(1\leq j\leq K)$ be multivalued holomorphic sections of $L$ on $X_{c(n)}$ without common zeros on $X_{c(n)}\setminus Z_c$.
    Then for any $\eta_{n-1}\in [\alpha P_n/Q_n-\varepsilon_{n-1},w_{n-1}]$, we set multivalued holomorphic sections 
    \begin{align*}
        \{s_{n-1,j}\}^{k_{n-1}}_{j=1}:=\{\tau^{\alpha/(mQ_n)}\times S_j^{\eta_{n-1}+\varepsilon_{n-1}-\alpha P_n/Q_n}\}^K_{j=1}
    \end{align*}
    of $L^{\otimes (\eta_{n-1}+\varepsilon_{n-1})}$ on $X_{c_{n-1}}$. Since 
    \begin{align*}
        \scr{I}_{n-1}(t):=\scr{I}\Big(\big(\sum^{k_{n-1}}_{j=1}|s_{n-1,j}|^2\big)^{-t}\Big),
    \end{align*}
    we can see $\mathbf{A_{n-1}}$ holds, where $\scr{I}_{n-1}(t)=\scr{I}(t\alpha)$ on $X_{c_{n-1}}\setminus Z_c$.
\end{proof}

\subsection{Kawamata-Shokurov's concentration method}

We assume $\mathbf{A_d}$ with $d\geq0$. Let $\kappa$ be a positive integer such that $\kappa(\eta_d+\varepsilon_d)$ is a positive integer and that every $s_{d,j}^\kappa$ is a holomorphic section of the line bundle $L^{\otimes \kappa(\eta_d+\varepsilon_d)}$ on $X_{c_d}$.
We take a real number $c_\sharp\in(c(d),c_d)$. Let $\scr{J}_d$ be the sheaf of ideal of $\cal{O}_{X_{c_\sharp}}$ generated locally by $\{s_{d,j}^\kappa\}^{k_d}_{j=1}$ and let $\scr{J}^\sharp_d$ be the restriction of $\scr{J}_d$ on $X_{c_\sharp}$. 
We take a proper modification $\pi_d:\widetilde{X}_\sharp\longrightarrow X_{c_\sharp}$ as in $\S\ref{subsection: multivalued section and natural sHm}$ by a finite number of blow-ups with non-singular centers and 
a finite number of smooth exceptional divisors $E_\mu$ in $\widetilde{X}_\sharp$ with only simple normal crossing, satisfying that
the sheaf $\pi^{-1}_d\!\scr{J}^\sharp_d\cdot\cal{O}_{\widetilde{X}_\sharp}=\rom{im}(\pi^*_d\scr{J}^\sharp_d\longrightarrow\cal{O}_{\widetilde{X}_\sharp})$ is equal to the ideal sheaf $\cal{O}_{\widetilde{X}_\sharp}(-\sum^J_{\mu=1}\xi^\sharp_\mu E_\mu)$ for some $\xi_\mu\in\bb{N}\cup\{0\}$,
and that $K_{\widetilde{X}_\sharp}=\pi^*K_{X_\sharp}\otimes\cal{O}_{\widetilde{X}_\sharp}(\sum^J_{\mu=1}\zeta_\mu E_\mu)$ for some $\zeta_\mu\in\bb{N}\cup\{0\}$.
Let $\xi_\mu:=\xi_\mu^\sharp/\kappa$ and 
\begin{align*}
    M_j:=\{\mu\mid E_\mu \text{ intersects } \pi_d^{-1}(x_j) \text{ and } \xi_\mu-\zeta_\mu\geq1\}
\end{align*}
for any $1\leq j\leq r$. 
Note that $\pi_d(E_\mu)\subset Y_d(1)$ if and only if $\xi_\mu-\zeta_\mu\geq1$.
When calculating the \( L^2 \)-integrability on \( X_{c_\sharp} \), the integrability conditions $\mathbf{A_d}(i)$-$(iii)$ are equivalent to the following three conditions $(i)'$-$(iii)'$, respectively.
\begin{itemize}
    \item [$(i)'$] the set $M_j$ is non-empty, for any $j\in J_d$.
    \item [$(ii)'$] for any $j\in J_d$, $\mu\in M_j$ if and only if $\xi_\mu-\zeta_\mu=1$.
    \item [$(iii)'$] for any $j\in\{1,\ldots,r\}\setminus J_d$, there exists some $\mu\in M_j$ with $\xi_\mu-\zeta_\mu>1$.
\end{itemize}
The condition on the dimension $\mathbf{A_d}(iv)$ is rewritten as
\begin{itemize}
    \item [$(iv)'$] $\rom{dim}\,\pi_d(E_\mu)\leq d$, for any $j\in J_d$ and any $\mu\in M_j$.
\end{itemize}

A technique of Kawamata-Shokurov's method (see \cite[Chapter\,2-3]{KMM87}) is to reduce the problem on the complicated space $\bigcup_\mu E_\mu$ to that of only one component. 
By appropriately using blow-ups, we obtain the following generalization of \cite[Lemma-Definition\,4.4]{Tak98} for singular Hermitian metrics.

\begin{lemma}\label{Ext [Lemma-Definition 4.4, Tak98]: singular positive var}
    Assume $\mathbf{A_d}$ with $d\geq0$ and let $\pi_d:\widetilde{X}_\sharp\longrightarrow X_{c_\sharp}$ be a proper modification and $\{E_\mu\}$ be the finite number of smooth irreducible divisor in $\widetilde{X}_\sharp$ as above. 
    Then the following statement $\mathbf{B_d}$ holds: there exist a real number $c'_d\in(c(d),c_\sharp)$, a rational number $\varepsilon'_d\in(\varepsilon(d+1),\varepsilon(d))$, a finite number of multivalued holomorphic sections $s'_{d,1},\ldots,s'_{d,k'_d}$ of $L^{\otimes(\eta_d+\varepsilon'_d)}$ on $X_{c'_d}$ 
    and a non-empty subset $J'_d$ of $\{1,\ldots,r\}$ such that the following conditions $(o)'$-$(iv)'$ are satisfied. 
    Let $\kappa'$ be a positive integer such that $\kappa'(\eta_d+\varepsilon'_d)$ is a positive integer ans that every $(s'_{d,j})^{\kappa'}$ is a holomorphic section of the line bundle $L^{\otimes \kappa'(\eta_d+\varepsilon'_d)}$ on $X_{c'_d}$. 
    Let $\scr{J}'$ be the sheaf of ideal of $\cal{O}_{X_{c'_d}}$ generated locally by $\{(s'_{d,j})^{\kappa'}\}^{k'_d}_{j=1}$. 
    \begin{itemize}
        \item [$(o)'$] there exist non-negative integers $\xi^*_\mu$ such that $\pi_d^{-1}\!\scr{J}'\cdot\cal{O}_{\widetilde{X}_{c'_d}}\!\!=\cal{O}_{\widetilde{X}_{c'_d}}(-\sum_\mu\xi^*_\mu E_\mu|_{\widetilde{X}_{c'_d}})$, where $\widetilde{X}_{c'_d}:=\pi_d^{-1}(X_{c'_d})$.
    \end{itemize}
    Let $\xi'_\mu:=\xi^*_\mu/\kappa'$ and 
    \begin{align*}
        M'_j:=\{\mu\mid E_\mu \text{intersects } \pi_d^{-1}(x_j) \text{ and } \xi'_\mu-\zeta_\mu\geq1\},
    \end{align*}
    for any $1\leq j\leq r$. Then after relabelling $\{E_\mu\}$, 
    \begin{itemize}
        \item [$(i)'$] $M'_j$ is non-empty and consists of the single element $\mu=0$, for any $j\in J'_d$.
        \item [$(ii)'$] the difference $\xi'_0-\zeta_0=1$.
        \item [$(iii)'$] for any $j\in\{1,\ldots,r\}\setminus J'_d$, there exist some $\mu\in M'_j$ with $\xi'_\mu-\zeta_\mu>1$.
        \item [$(iv)'$] the dimension of a (reduced and irreducible) closed subvariety $Y_d:=\pi_d(E_0)$ of $X_{c_\sharp}$ is less then or equal to $d$.
    \end{itemize}
\end{lemma}

\begin{proof}
    Let us assume $\mathbf{A_d}(i)'$-$(iv)'$ above. We take real numbers $c(d)<c'_d<c_\flat<c_\dagger<c_\ddag<c_\sharp$. 
    The singular Hermitian metric $\pi_d^*h_c$ on $\pi_d^*L$ is singular semi-positive on $\widetilde{X}_\sharp$ and singular positive on $\widetilde{X}_\sharp\setminus\bigcup_\mu E_\mu$.
    By blow-ups Theorem \ref{Blow ups of Dem appro with log poles and ideal sheaves}, this weight function $\varphi_c$ of $h_c$ with algebraic singularities is locally expressed as 
    \begin{align*}
        \varphi_c=\frac{1}{2m_c}\log\sum_{k\in\bb{N}}|\sigma_{c,k}|^2+\psi_c,
    \end{align*}
    where $m_c\in\bb{Q}_{>0}$ and $\{\sigma_{c,k}\}_{k\in\bb{N}}$ is an orthonormal basis of a Hilbert space. 
    As in the proof of blow-ups Theorem \ref{Blow ups of Dem appro with log poles and ideal sheaves}, there exists a positive integer $\ell_1$ and positive integers $a_\mu$ such that a line bundle $\pi^*_dL^{\otimes m}\otimes\cal{O}_{\widetilde{X}_\sharp}(-\sum_\mu a_\mu E_\mu)$ has a singular positive Hermitian metric $\pi_d^*h_c^m\otimes h_a^*$ on $\pi_d^{-1}(X_{c_\ddag})=:\!\widetilde{X}_\ddag$
    for any $m\geq\ell_1$ and such that $\pi^*_dL^{\otimes m}\otimes\cal{O}_{\widetilde{X}_\sharp}(-\sum_\mu a'_\mu E_\mu)$ has a singular positive Hermitian metric $\pi_d^*h_c^m\otimes h_{a'}^*$ on $\widetilde{X}_\ddag$ for any positive integers $a'_\mu$ such as $|a_\mu-a'_\mu|\ll m$. 
    Here, $h_a$ and $h'_a$ are smooth Hermitian metrics on $\cal{O}_{\widetilde{X}_\sharp}(\sum_\mu a_\mu E_\mu)$ and $\cal{O}_{\widetilde{X}_\sharp}(\sum_\mu a'_\mu E_\mu)$, respectively.
    Hence we can choose positive integers $\ell_1, \ell_2$ and $a_\mu$ such that the following conditions are satisfied.
    \begin{itemize}
        \item [$(a1)$] the line bundle $\widetilde{L}:=\pi_d^*L^{\otimes \,\ell_1}\otimes\cal{O}_{\widetilde{X}_\sharp}(-\sum_\mu a_\mu E_\mu)$ has a singular positive Hermitian metric $\widetilde{h}:=\pi_d^*h_c^{\ell_1}\otimes h_a^*$ on $\widetilde{X}_{c_\ddag}$.
        \item [$(a2)$] $0<\ell_1/\ell_2<\varepsilon(d)-\varepsilon_d$.
        \item [$(a3)$] if $(1+\zeta_\mu)/\xi_\mu<(1+\zeta_\nu)/\xi_\nu$, then $0<(1+\zeta_\mu-a_\mu/\ell_2)/\xi_\mu<(1+\zeta_\nu-a_\nu/\ell_2)/\xi_\nu$.
        \item [$(a4)$] if $\mu\ne\nu$, then $(1+\zeta_\mu-a_\mu/\ell_2)/\xi_\mu\ne(1+\zeta_\nu-a_\nu/\ell_2)/\xi_\nu$.
    \end{itemize}

    The weight $\widetilde{\varphi}$ of $\widetilde{h}$ is locally expressed as 
    \begin{align*}
        \widetilde{\varphi}=\frac{\ell_1}{2m_c}\log\sum_{k\in\bb{N}}|\widetilde{\sigma}_k|^2+\widetilde{\psi},
    \end{align*}
    here let $\widetilde{\sigma}_k:=\pi^*_d\sigma_{c,k}$. 
    We introduce the ideal $\scr{J}_\ell$ of germs of holomorphic function $f$ such that $|f|\leq Ce^{m_c\widetilde{\varphi}/\ell_1}$ for some constant $C$. 
    For a singular positive line bundle $(\widetilde{L},\widetilde{h})$, by blow-ups Theorem \ref{Blow ups of Dem appro with log poles and ideal sheaves}, there exist a proper modification $\pi_\dagger:\widehat{X}_\dagger\longrightarrow\widetilde{X}_\dagger:=\pi_d^{-1}(X_{c_\dagger})$ obtained by a finite sequence of blow-ups along the ideal $\scr{J}_\ell$, 
    a large integer $\ell_3$ and an effective divisor $D$ with only simple normal crossing such that the holomorphic line bundle $\scr{L}:=\pi_\dagger^*\widetilde{L}^{\otimes\,\ell_3}\otimes\cal{O}_{\widehat{X}_\dagger}(-D)$ is positive on $\widehat{X}_\dagger$.

    By the vanishing theorem (= Theorem \ref{Vanishing thm for positive line bdl}), there exists a positive integer $\ell_4$ and a finite number of holomorphic sections $\{u_j\}^{K'}_{j=1}\subset H^0((\pi_\dagger^*\pi_d)^{-1}(X_{c_\flat}),\scr{L}^{\otimes\,\ell_4})$ without common zeros on $(\pi_\dagger^*\pi_d)^{-1}(X_{c'_d})$.
    Let $v_D$ be the canonical section of $\cal{O}_{\widehat{X}_\dagger}(D)$, then the sections $u_jv_D^{\ell_4}\in H^0((\pi_\dagger^*\pi_d)^{-1}(X_{c_\flat}),\pi_\dagger^*\widetilde{L}^{\otimes\,\ell_3\ell_4})$ descend to sections $\widetilde{u}_j\in H^0(\pi_d^{-1}(X_{c'_d}),\widetilde{L}^{\otimes\,\ell_3\ell_4})$. This sections $\{\widetilde{u}_j\}^{K'}_{j=1}$ has no common zero on $\pi_d^{-1}(X_{c'_d}\setminus Z_c)$.
    Let $v_E$ be the canonical section of $\cal{O}_{\widetilde{X}_\sharp}(\sum_\mu a_\mu E_\mu)$, then the sections $\widetilde{u}_jv_E^{\otimes\,\ell_3\ell_4}\in H^0(\pi_d^{-1}(X_{c'_d}),\pi_d^*L^{\otimes\,\ell_1\ell_3\ell_4})$ descend to sections $\widetilde{v}_j\in H^0(X_{c'_d},L^{\otimes\,\ell_1\ell_3\ell_4})$. 
    We let $v_j:=\widetilde{v}_j^{1/(\ell_2\ell_3\ell_4)}$ $(1\leq j\leq K')$ be multivalued holomorphic sections of $L^{\otimes\,\ell_1/\ell_2}$ on $X_{c'_d}$ without common zeros on $X_{c'_d}\setminus(Z_c\cup\bigcup_\mu\pi_d(E_\mu))$.

    The remaining part follows similarly to the proof of \cite[Lemma-Definition\,4.4]{Tak98}, as described below:
    For any $x_j\in X_c\setminus Z_c$ $(1\leq j\leq r)$, we set 
    \begin{align*}
        \alpha(x_j):=&\min\{(1+\zeta_\mu-a_\mu/\ell_2)/\xi_\mu\mid\mu\in M_j\}\\
        =&\min\{(1+\zeta_\mu-a_\mu/\ell_2)/\xi_\mu\mid\mu \text{ such that } E_\mu \text{ intersects } \pi^{-1}_d(x_j)\}
    \end{align*}
    by $(a3)$, then we have  
    \begin{itemize}
        \item [$(b1)$] the minimum $\alpha(x_j)$ is attained only at a unique index $\mu_j\in M_j$ by $(a4)$,
        \item [$(b2)$] for every $j\in J_d$, $\alpha(x_j)\in(0,1)$ by $\mathbf{A_d}(ii)'$.
    \end{itemize}
    By setting $\alpha:=\max\{\alpha(x_j)\mid 1\leq j\leq r\}$, the following is obtained.
    \begin{itemize}
        \item [$(c1)$] $\alpha=\max\{\alpha(x_j)\mid j\in J_d\}$ by $\mathbf{A_d}(ii)', \mathbf{A_d}(iii)'$ and $(a3)$,
        \item [$(c2)$] $\alpha\in(0,1)$ by $(b2)$ and $(c1)$, 
        \item [$(c3)$] if $\alpha(x_j)=\alpha(x_k)=\alpha$, then $\mu_j=\mu_k$ by $(a4)$ and $(b1)$.
    \end{itemize}

    We consider the set of indices which attain the maximum
    \begin{align*}
        J'_d:=\{j\in\{1,\ldots,r\}\mid\alpha(x_j)=\alpha\},
    \end{align*}
    then $J'_d$ is a non-empty subset of $J_d$ by $(c1)$. 
    After relabelling $\{E_\mu\}$, we may assume that, the index $\mu_j\in M_j$ in $(b1)$, $\mu_j=0$ for any $j\in J'_d$ by $(c3)$.
    We set 
    \begin{align*}
        \{s'_{d,j}\}^{k'_d}_{j=1}:=\{s^\alpha_{d,j}\times v_k\times S_\ell^{(1-\alpha)(\eta_d+\varepsilon_d)}\}_{j,k,\ell}
    \end{align*}
    multivalued holomorphic sections of $L^{\otimes(\eta_d+\varepsilon'_d)}$ on $X_{c'_d}$, where $\varepsilon'_d:=\varepsilon_d+\ell_1/\ell_2\in(\varepsilon(d+1),\varepsilon(d))$ by $(a2)$. 
    For any $t\geq0$, we set 
    \begin{align*}
        \scr{I}_d'(t):=\scr{I}\Big(\big(\sum_j|s'_{d,j}|^2\big)^{-t}\Big).
    \end{align*}
    In the neighborhood of each point $x_j$ $(1\leq j\leq r)$, we obtain 
    \begin{align*}
        \scr{I}_d'(t):=\scr{I}\Big(\big((\sum_j|s'_{d,j}|^2)^\alpha\sum_j|v_j|^2\big)^{-t}\Big),
    \end{align*}
    where the set $\{S_\ell\}^K_{j=1}$ has no common zeros on $X_{c(n)}\setminus Z_c$. 
    Here, $\xi'_\mu:=\alpha\xi_\mu+a_\mu/\ell_2$ plays the role of $\xi_\mu$ in $\mathbf{A_d}$.
    We see that $\mathbf{B_d}(i)'$ and $\mathbf{B_d}(ii)'$ by the definition and relabeling, and that $\mathbf{B_d}(iv)'$ by $\mathbf{A_d}(iv)'$. 
    For any $j\in\{1,\ldots,r\}\setminus J'_d$, i.e., $\alpha(x_j)<\alpha$, there exists $\mu\in M_j$ such that $(1+\zeta_\mu-a_\mu/\ell_2)/\xi_\mu<\alpha$, i.e., $\xi'_\mu-\zeta_\mu>1$. 
    This represents $\mathbf{B_d}(iii)'$.
\end{proof}

\subsection{Points separation on each sublevel set $X_c$}

\begin{proof}[Proof of Theorem \ref{Ext [Tak98, Thm 4.1]}]
    By using Lemma \ref{Ext [Lemma 3.10, Tak98]: see [Lemma 4.1, AS95]}, we obtain the following inductive statement on dimension, similarly to the proof of \cite[Lemma\,4.6]{Tak98}. 
    Here, \cite[Lemma\,4.8]{Tak98} (= an effective estimate), used in this proof, is a local statement around $x_j$ for any $j\in J'_d$, and the situation is the same as in the case of a positive line bundle. 
    As in \cite{Tak98}, it follows from the Ohsawa-Takegoshi $L^2$-extension theorem (see \cite{OT87}).

\begin{lemma}$($\textnormal{cf.\,\cite[Lemma\,4.6]{Tak98}}$)$
    Assume $\mathbf{A_d}$ with $d>0$ and let $\pi_d:\widetilde{X}_\sharp\longrightarrow X_{c_\sharp}$ be the proper modification, $\{E_\mu\}$ be the finite number of smooth irreducible divisors in $\widetilde{X}_\sharp$ and $Y_d$ be the closed subvariety of $X_{c_\sharp}$ with dimension at most $d$ as in Lemma \ref{Ext [Lemma-Definition 4.4, Tak98]: singular positive var}. 
    If the dimension of $Y_d$ is less than $d$, then the statement $\mathbf{A_{d-1}}$ hold with $\eta_{d-1}=\eta_d$. 
    If the dimension of $Y_d$ equals to $d$, for every pair $(p,q)$ of positive integers $(p>2)$ such that $p/q\leq P_d/Q_d$ and that $(p-2)^d\widetilde{\rom{Vol}}_Y(L\otimes\scr{I}(h_c))>|J'_d|q^d$ where $|J'_d|$ is the number of elements in $J'_d$, the statement $\mathbf{A_{d-1}}$ holds with $\eta_{d-1}=\eta_d+dp/q$.
\end{lemma}

    By the basic construction, i.e., Lemma \ref{Lem basic construcion An-1}, and Lemma \ref{Ext [Lemma-Definition 4.4, Tak98]: singular positive var} and \ref{Ext [Lemma 3.10, Tak98]: see [Lemma 4.1, AS95]}, we see that the statement $\mathbf{A_0}$ holds. 
    Hence, there exist a positive rational number $\eta_0\leq w_0$, a rational number $\varepsilon_0$ with $0 < \varepsilon_0 < m_0-w_0$ and a finite number of multivalued holomorphic sections $\{s_{0,j}\}^{k_0}_{j=1}$ of $L^{\otimes(\eta_0+\varepsilon_0)}$ on $X_{c_0}$ such that $\{x_j\}^r_{j=1}\subset V\big(\scr{I}\big((\sum_{j=1}^{k_0}|s_{0,j}|^2)^{-1}\big)\big)$ and that some $x_j$ is isolated in $V\big(\scr{I}\big((\sum_{j=1}^{k_0}|s_{0,j}|^2)^{-1}\big)\big)$.
    Thus, we obtain the desired singular Hermitian metric
    \begin{align*}
        H_c:=h_0^{m-(\eta_0+\varepsilon_0)}\frac{h_0^{\eta_0+\varepsilon_0}}{\sum^{k_0}_{j=1}|s_{0,j}|^2_{h_0^{\eta_0+\varepsilon_0}}}e^{-2(m-(\eta_0+\varepsilon_0))\varphi_c}=h_c^{m-(\eta_0+\varepsilon_0)}\frac{1}{\sum^{k_0}_{j=1}|s_{0,j}|^2}
    \end{align*}
    on $L|_{X_{c_0}}^{\otimes m}$ for any integer $m\geq m_0>\vartheta_c:=\eta_0+\varepsilon_0$, where $X_c\Subset X_{c_0}$.
\end{proof}

Using Theorem \ref{Ext [Tak98, Thm 4.1]}, along with induction on the number of points $r$ and vanishing Theorem \ref{Thm Nadel type vanishing without Kahler}, which is of the Nadel type without requiring a \kah metric, the following result can be derived (see \cite[Section\,10]{AS95}).

\begin{theorem}\label{Ext [Tak98, Thm 4.2]}
    Let $x_1,\ldots,x_r$ be $r$ distinct points on a subset $X_c\setminus Z_c$. 
    Then for any positive integer $m>I_{\sum}(L,h_c,\{x_j\}^r_{j=1})$,
    there exists a positive rational number $\theta_c\in\bb{Q}_{>0}$ smaller than $m_0(L,h_c,\{x_j\}^r_{j=1})$ such that the space $H^0(X_c,K_X\otimes L^{\otimes m}\otimes\scr{I}(h_c^{m-\theta_c}))$ separates $\{x_j\}^r_{j=1}$, 
    i.e., the following restriction map is surjective. 
    \begin{align*}
        H^0(X_c,K_X\otimes L^{\otimes m}\otimes\scr{I}(h_c^{m-\theta_c}))\longrightarrow\bigoplus^r_{j=1}\cal{O}_X/\fra{m}_{X,x_j}.
    \end{align*}
\end{theorem}

\begin{proof}
    Taking the metric \( H_c \) constructed in Theorem \ref{Ext [Tak98, Thm 4.1]}, it is defined on \( X_{c_0} \) for some $c<c_0$. Here, there exists a rational positive number $\vartheta_r<m_0(L,h_c,\{x_j\}^r_{j=1})$ which satisfies $\scr{I}(H_c)\subseteq\scr{I}(h_c^{m-\vartheta_r})$.
    Thus, we can apply Nadel-type vanishing Theorem \ref{Thm Nadel type vanishing without Kahler}, and the $1$-th cohomology $H^1(X_c,K_X\otimes L^{\otimes m}\otimes\scr{I}(H_c))$ vanishes. 
    By relabeling the $r$ points, we can assume that $J_0=\{1,\ldots,\ell\}$ for some $1\leq\ell\leq r$. In other words, the set $V(\scr{I}(H_c))$ is isolated at $x_j$ for $1\leq j\leq\ell$. 
    Let $\scr{J}$ be the ideal sheaf on $X_c$ which agrees with $\scr{I}(H_c)$ on $X_c\setminus\{x_1,\ldots,x_\ell\}$ and which agrees with $\cal{O}_X$ on the set $\{x_1,\ldots,x_\ell\}$. 
    From the exact sequence 
    \begin{align*}
        0\longrightarrow K_X\otimes L^{\otimes m}\otimes\scr{I}(H_c)\longrightarrow K_X\otimes L^{\otimes m}\otimes\scr{J}\longrightarrow K_X\otimes L^{\otimes m}\otimes(\scr{J}/\scr{I}(H_c))\longrightarrow 0,
    \end{align*}
    it follows that the restriction map 
    \begin{align*}
        H^0(X_c,K_X\otimes L^{\otimes m}\otimes\scr{J}) \longrightarrow \bigoplus^\ell_{j=1}\cal{O}_X/\fra{m}_{X,x_j}
    \end{align*}
    is surjective. Thus, by $\{x_{\ell+1},\ldots,x_r\}\subset V(\scr{J})$, we can find a holomorphic section of $K_X\otimes L^{\otimes m}\otimes\scr{J}$ over $X_c$ which vanishes on $x_{\ell+1},\ldots,x_r$ and has prescribed values at $x_1,\ldots,x_\ell$.
    From the definition of the sheaf $\scr{J}$ and $\scr{I}(H_c)\subseteq\scr{I}(h_c^{m-\vartheta_r})$, it is clear that $\scr{J}\subseteq\scr{I}(h_c^{m-\vartheta_r})$.
    In particular, there exists a holomorphic section $u_1$ of $K_X\otimes L^{\otimes m}\otimes\scr{I}(h^{m-\vartheta_r})$ over $X_c$ which vanishes on $x_2,\ldots,x_r$ and nonzero at $x_1$. 
    
    Similarly, for the integer $m>I_{\sum}(L,h_c,\{x_j\}^r_{j=1})>I_{\sum}(L,h_c,\{x_j\}^r_{j=2})$, by applying Theorem \ref{Ext [Tak98, Thm 4.1]}, there exist a positive rational number $\vartheta_{r-1}<m_0(L,h_c,\{x_j\}^r_{j=2})\leq m_0(L,h_c,\{x_j\}^r_{j=1})$ and a holomorphic section $u_2$ of $K_X\otimes L^{\otimes m}\otimes\scr{I}(h^{m-\vartheta_{r-1}})$ which vanishes on $x_3,\ldots,x_r$ and nonzero at $x_2$.
    For $2\leq k\leq r$, by applying the induction hypothesis on $r$ to the points $x_k,\ldots,x_r$, we conclude that there exists a holomorphic section $u_k$ of $K_X\otimes L^{\otimes m}\otimes\scr{I}(h^{m-\vartheta_{r-k+1}})$ over $X_c$ which vanishes on $x_{k+1},\ldots,x_r$ and nonzero at $x_k$, here $0<\vartheta_\ell<m_0(L,h_c,\{x_j\}^r_{j=1})$ for any $1\leq\ell\leq r$. 
    Let $\theta_c:=\max_{1\leq\ell\leq r}\vartheta_\ell$, then we have $\scr{I}(h_c^{m-\vartheta_\ell})\subseteq\scr{I}(h_c^{m-\theta_c})$ for all $1\leq\ell\leq r$, and the $r$ holomorphic sections $u_1,\ldots,u_r$ restricted to the $r$ points $\{x_1,\ldots,x_r\}$ are linearly independent. Hence, the holomorphic sections of $K_X\otimes L^{\otimes m}\otimes\scr{I}(h_c^{m-\theta_c})$ over $X_c$ can separate the $r$ points $x_1,\ldots,x_r$. 
\end{proof}

Let $h_\natural=h_\natural(\ell)$ be the singular Hermitian metric constructed via approximation of $h$ for some integer $\ell\in\bb{N}$.
Regarding the notation \( I_d \), \( I_{\sum} \) and \( m_0 \) defined in $\S\ref{subsection: sHm inducing points separation}$, we replace \( (X_c, h_c) \) with \( (X, h_\natural) \) and similarly define \( I_d(L,h_\natural,\{x_j\}^r_{j=1}) \), \( I_{\sum}(L,h_\natural,\{x_j\}^r_{j=1}) \), and \( m_0(L,h_\natural,\{x_j\}^r_{j=1}) \). 
Furthermore, for $r$ distinct points $x_1,\ldots,x_r$ on $X$, we set 
\begin{align*}
    m_r(L,h_\natural):=\max\{m_0(L,h_\natural,\{x_j\}^r_{j=1}) \text{ for } r \text{ distinct points on } X\}.
\end{align*}
Note that $m_0(L,h_\natural,\{x_j\}^r_{j=1})\leq n(n+2r-1)/2+1$. Thus, if $X$ is non-compact then  
\begin{align*}
    m_0(L,h_\natural,\{x_j\}^r_{j=1})\leq \frac{n(n+2r-3)}{2}+2-r \quad \text{and} \quad
    m_2(L,h_\natural)\leq \frac{n(n+1)}{2}.
\end{align*}
In fact, due to non-compactness, it always holds that \( I_n = +\infty \).
By similarly defining the above notations for each $X_c$ and $h_c$, we have $m_0(L,h_\natural,\{x_j\}^r_{j=1})\geq m_0(L,h_c,\{x_j\}^r_{j=1})$ for any $r$ distinct points $x_1, \dots, x_r$ on $X_c$, and a similar inequality for \( h_c \) as above.

By singular positivity of $h$, $X$ carries a Hermitian metric $\omega_X$ that satisfies $\iO{L,h}\geq\omega_X$ on $X$ in the sense of currents.
Here, for each $X_c$, there exists a constant $\scr{C}_c>0$ such that $\scr{C}_c\omega_X\geq\iO{L,h_0}$ on $X_c$. Using this constant $\scr{C}_c$, the following corollary is obtained.

\begin{theorem}\label{Thm of Ext [Tak98, Thm 4.1 and 4.2]}
    Let $x_1,\ldots,x_r$ be $r$ distinct points on a subset $X_c\setminus\bigcup_{p\in\bb{N}}V(\scr{I}(h^p))$. 
    Then for any positive integer $m\geq\scr{C}_c\cdot m_0(L,h_\natural,\{x_j\}^r_{j=1})$, there exists a singular positive Hermitian metric $\cal{H}_c$ of $L|_{X_c}^{\otimes m}$ 
    such that $\scr{I}(\cal{H}_c)\subset\scr{I}(h_\natural^m)$ and $\{x_j\}^r_{j=1}\subset V(\scr{I}(\cal{H}_c))$, and that some point $x_j$ is isolated in $V(\scr{I}(\cal{H}_c))$. 
    
    Furthermore, the space $H^0(X_c,K_X\otimes L^{\otimes m}\otimes\scr{I}(h_\natural^m))$ separates $\{x_j\}^r_{j=1}$, 
    i.e., the restriction map $H^0(X_c,K_X\otimes L^{\otimes m}\otimes\scr{I}(h_\natural^m))\longrightarrow\bigoplus^r_{j=1}\cal{O}_X/\fra{m}_{X,x_j}$ is surjective.
\end{theorem}

\begin{proof}
    Similarly to the proof of Theorem \ref{Ext [Tak98, Thm 4.1]}, we see that the statement $\mathbf{A_0}$ holds.
    Using similar notation, there exist a rational number $\eta_0\leq w_0$, a rational number $\varepsilon_0<m_0-w_0$ and multivalued holomorphic sections $\{s_{0,j}\}^{k_0}_{j=1}$ of $L^{\otimes(\eta_0+\varepsilon_0)}$ on $X_{c_0}$ 
    such that $\{x_j\}^r_{j=1}\subset V\big(\scr{I}\big((\sum_{j=1}^{k_0}|s_{0,j}|^2)^{-1}\big)\big)$ and that some $x_j$ is isolated in $V\big(\scr{I}\big((\sum_{j=1}^{k_0}|s_{0,j}|^2)^{-1}\big)\big)$. 
    We define the desired singular Hermitian metric $\cal{H}_c$ on $L|_{X_{c_0}}^{\otimes m}$ by 
    \begin{align*}
        \cal{H}_c:=h_0^{m-(\eta_0+\varepsilon_0)}\frac{h_0^{\eta_0+\varepsilon_0}}{\sum^{k_0}_{j=1}|s_{0,j}|^2_{h_0^{\eta_0+\varepsilon_0}}}e^{-2m\varphi_c}=h_c^m\cdot h_0^{-(\eta_0+\varepsilon_0)}\frac{1}{\sum^{k_0}_{j=1}|s_{0,j}|^2}.
    \end{align*}

    We prove that this metric $\cal{H}_c$ is singular positive on $X_{c_0}$.
    Here, $m_0=m_0(L,h_c,\{x_j\}^r_{j=1})$ and $\vartheta_c:=\eta_0+\varepsilon_0<m_0$. 
    By the construction of $h_c$ using Theorems \ref{Blow ups of Dem appro with log poles and ideal sheaves} and \ref{Thm characterizations of canonical sHm}, there exists a sufficiently small number $\delta_c>0$ such that $\iO{L,h_c}\geq(1-\delta_c)\omega_X$ in the sense of currents on $X_{c+\varepsilon_c}$, where $c_0\in (c,c+\varepsilon_c)$.
    Note that 
    $\bigcup_{p\in\bb{N}}V(\scr{I}(h^p))=\bigcup_{c\in\bb{N}}Z_c$ by Proposition \ref{Prop E0(h)=cup Zj}.
    Since this metric $h_c$ can be replaced as $\delta_c$ approaching zero by blow-ups Theorem \ref{Blow ups of Dem appro with log poles and ideal sheaves}, we may assume that $\delta_c<1-\vartheta_c/m_0$. 
    Thus, we obtain 
    \begin{align*}
        \iO{L^{\otimes m},\cal{H}_c}
        &\geq m\,\iO{L,h_c}-\vartheta_c\,\iO{L,h_0}
        =(m-\scr{C}_cm_0)\iO{L,h_c}+\scr{C}_cm_0\iO{L,h_c}-\vartheta_c\iO{L,h_0}\\
        &\geq(m-\scr{C}_cm_0)(1-\delta_c)\omega_X+\scr{C}_c m_0\Big(1-\delta_c-\frac{\vartheta_c}{m_0}\Big)\omega_X
        >0.
    \end{align*}
    Similarly to the proof of Theorem \ref{Ext [Tak98, Thm 4.2]}, the proof is complete.
\end{proof}

\begin{remark}\label{Remark of points separation thm on Xc}
    For each \( X_c \) and \( h_c \), there exists a sufficiently small \( \delta_c>0 \) such that \( \iO{L,h_c}\geq(1-\delta_c)\omega_X \) by blow-ups Theorem \ref{Blow ups of Dem appro with log poles and ideal sheaves}. 
    If the assumption in Theorem \ref{Thm of Ext [Tak98, Thm 4.1 and 4.2]} is changed to $m \geq \frac{\scr{C}_c}{1-\delta_c}m_0(L,h_c,\{x_j\}^r_{j=1})$, then $r$ discrete points \( x_1, \dots, x_r \) can be chosen from \( X_c \setminus Z_c \), and a similar statement can be obtained regarding $h_c$ instead of $h_\natural$.
\end{remark}

The difference between Theorem \ref{Ext [Tak98, Thm 4.2]} and \ref{Thm of Ext [Tak98, Thm 4.1 and 4.2]} (or Remark \ref{Remark of points separation thm on Xc}) is ideal sheaves $\scr{I}(h_c^{m-\vartheta_c})$ and $\scr{I}(h_c^m)$. 
The latter is effective in global discussions because Approximation Theorem \ref{Approximation thm of hol section with ideal sheaves} can be applied to it.

\section{Global points separation, embedding and bigness theorem \\ on weakly pseudoconvex manifolds}\label{Section 8: Global Main results}

In this section, we provide global theorems as main results. 
Recall that $X$ is an $n$-dimensional weakly pseudoconvex manifold and $(L,h)$ be a singular positive line bundle on $X$. 
Here, $X$ carries a Hermitian metric $\omega_X$ that satisfies $\iO{L,h}\geq\omega_X$ on $X$ in the sense of currents.
First, we can derive the following result 

\begin{theorem}$($\textnormal{\textbf{Global points separation}}$)$\label{Thm global points separation}
    We assume that there exist a constant $C>0$ and a smooth Hermitian metric $h_0$ such that $C\omega_X\geq\iO{L,h_0}$ on $X$. 
    Let $x_1,\ldots,x_r$ be $r$ distinct points on $X\setminus \bigcup_{p\in\bb{N}}V(\scr{I}(h^p))$.  
    Then for any positive integer 
    \begin{align*}
        m\geq C\Big(\frac{n(n+2r-1)}{2}+1\Big),
    \end{align*}
    the space $H^0(X,K_X\otimes L^{\otimes m}\otimes\scr{I}(h^m))$ separates $\{x_j\}^r_{j=1}$,
    i.e., the restriction map $H^0(X,K_X\otimes L^{\otimes m}\otimes\scr{I}(h^m))\longrightarrow \bigoplus^r_{j=1}\cal{O}_X/\fra{m}_{X,x_j}$ is surjective. 

    In particular, if $m\geq C\big(n(n+1)/2+1\big)$ then the adjoint bundle $K_X\otimes L^{\otimes m}$ is generated by global sections on $X\setminus\bigcup_{k\in\bb{N}}V(\scr{I}(h^k))$. 
\end{theorem}

Theorem \ref{Thm global points separation} can be proven by using Theorem \ref{Thm of Ext [Tak98, Thm 4.1 and 4.2]} and Approximation Theorem \ref{Approximation thm of hol section with ideal sheaves}, 
and taking a singular Hermitian metric $h_\natural:=h_\natural(m)$ constructed via approximation for each \( m \) above, leading to $\scr{I}(h_\natural^m)=\scr{I}(h^m)$ from Theorem \ref{Thm characterizations of canonical sHm}.
Recall that the set $E_{+}(h)$ is defined as the non-zero locus of the Lelong number, and it is known that
\begin{center}
    $E_{+}(h):=\{x\in X\mid\nu(-\log\,h,x)>0\}=\bigcup_{p\in\bb{N}}V(\scr{I}(h^p))=\bigcup_{c\in\bb{N}}Z_c$,
\end{center}
by Proposition \ref{Prop E0(h)=cup Zj}.
Furthermore, if \( X \) is non-compact, the condition of Theorem \ref{Thm global points separation} can be made weaker, namely \(m\geq C\big(n(n+2r-3)/2+2-r\big)\).

In what follows, we assume that \( X \) is \textit{non}-\textit{compact} in this section.

\begin{theorem}$($\textnormal{\textbf{Global embedding}}$)$\label{Thm global embedding}
    We assume that there exist a constant $C>0$ and a smooth Hermitian metric $h_0$ such that $C\omega_X\geq\iO{L,h_0}$ on $X$. Let $m$ be a positive integer with $m\geq C\cdot n(n+1)/2$. 
    Then, any open subset of $X\setminus\bigcup_{k\in\bb{N}}V(\scr{I}(h^k))$ 
    is holomorphically embeddable in to $\bb{P}^{2n+1}$ by a linear subsystem of $|(K_X\otimes L^{\otimes m})^{\otimes n+1+\ell}\otimes L^{\otimes p}\otimes\scr{I}(h^{p+1})|$ for any $\ell\geq1$ and $p\geq0$.
    In other words, the adjoint bundle $K_X\otimes L^{\otimes m}$ is ample over any open subset of $X\setminus\bigcup_{k\in\bb{N}}V(\scr{I}(h^k))$, and 
    the line bundle $(K_X\otimes L^{\otimes m})^{\otimes n+1+\ell}\otimes L^{\otimes p}$ is very ample over any open subset of $X\setminus\bigcup_{k\in\bb{N}}V(\scr{I}(h^k))$ for any $\ell\geq1$ and $p\geq0$.
\end{theorem}

Here, if $h$ is smooth positive then we can take $C=1$, $h=h_0$ and $\omega_X=\iO{L,h}=\iO{L,h_0}$.

\begin{proof}
    We take $\scr{X}$ as an arbitrary open subset of $X\setminus\bigcup_{k\in\bb{N}}V(\scr{I}(h^k))$.
    For the integer \( q = \max\{m,p+1\} \), by taking the singular Hermitian metric \( h_\natural:=h_\natural(q) \) constructed via approximation of $h$, it follows from Theorem \ref{Thm characterizations of canonical sHm} that \( \scr{I}(h_\natural^k)=\scr{I}(h^k) \) for any $k\leq q$.
    By Theorem \ref{Thm global points separation}, for every relatively compact open subset $K$ of $\scr{X}$, 
    there exists a finite dimensional subspace $V_K\subset H^0(X,K_X \otimes L^{\otimes m}\otimes\scr{I}(h^m))$ such that the Kodaira map $\varPhi_{V_K}:X\dashrightarrow \bb{P}(V_K)$ is holomorphic and one to one over $K$, where $K\subset X\setminus\rom{Bs}_{|V_K|}$.

    Let $N:=\rom{dim}\,V_K-1$ and $\cal{L}:=K_X\otimes L^{\otimes m}$, and let $h_{\cal{L}}$ be a smooth Hermitian metric on $\cal{L}$. For any point $x\in K$, there exists a singular Hermitian metric $\cal{H}_x$ on $\cal{L}^{\otimes n+1}$ over $X$ such that the curvature current is semi-positive, $x$ is isolated in $V(\scr{I}(\cal{H}_x))$, and that $\scr{I}(\cal{H}_x)_x\subset\fra{m}^2_{X,x}$ (see \cite[Lemma\,5.1]{Tak98}). 
    This metric \( \cal{H}_x \) can be constructed as follows. Let $(w_0:\cdots:w_N)$ be a homogenuous coordinate of $\bb{P}^N=\bb{P}(V_K)$ such that $\varPhi_{V_K}(x)=\{w_1=\cdots=w_N=0\}$. In other words, \( x \) is an isolated common zero of $\{\varPhi^*_{V_K}w_j\}^N_{j=1}$.
    For a local coordinate $(z_1,\ldots,z_n)$ centered at $x$, the inequality $\sum^N_{j=1}|\varPhi^*_{V_K}w_j|^2_{h_{\cal{L}}}\leq\cal{C}_x\sum^n_{j=1}|z_j|^2$ holds, where $\cal{C}_x>0$ is a constant. 
    Thus, the desired singular metric \( \cal{H}_x \) on \( \cal{L} \) is obtained by 
    \begin{align*}
        \cal{H}_x:=\frac{h_{\cal{L}}^{\otimes n+1}}{\big(\sum^N_{j=1}|\varPhi^*_{V_K}w_j|^2_{h_{\cal{L}}}\big)^{n+1}}=\frac{1}{\big(\sum^N_{j=1}|\varPhi^*_{V_K}w_j|^2\big)^{n+1}},
    \end{align*}
    which is defined on \( X\setminus\rom{Bs}_{|V_K|} \), is singular semi-positive, and satisfies the desired properties (see Skoda's result \cite[Lemma\,5.6\,(b)]{Dem12}).
    Locally, for some open neighborhood $U$ of each point in $\rom{Bs}_{|V_K|}$, the weight function $\varphi_{\cal{H}_x}:=(n+1)\log\sum^N_{j=1}|\varPhi^*_{V_K}w_j|^2$ of $\cal{H}_x$ is plurisubharmonic on $U\setminus\rom{Bs}_{|V_K|}$, and it tends to $-\infty$ as the point approaches $\rom{Bs}_{|V_K|}$. 
    Since $\varphi_{\cal{H}_x}$ is bounded above, it extends uniquely as a plurisubharmonic function. 
    Hence, by defining the value of $\cal{H}_x$ at $\rom{Bs}_{|V_K|}$ to be $+\infty$, the metric $\cal{H}_x$ can be redefined as a singular Hermitian metric of $\cal{L}$ over the whole $X$, and it is singular semi-positive on $X$.

    Let $h_{V_K}$ be a natural singular Hermitian metric on $\cal{L}$ defined by $V_K$, then this metric is singular semi-positive.
    Here, there exists an integer $c\in\bb{N}$ such that $K\subset X_c\setminus Z_c$.
    We define a singular Hermitian metric $\hbar_x$ on $\cal{L}^{\otimes n+\ell}\otimes L^{\otimes p+1}\longrightarrow X$ by $\hbar_x:=\cal{H}_x\otimes h_{V_K}^{\otimes\ell-1}\otimes h^{\otimes p+1}$, then this metric is singular positive and satisfies $\scr{I}(\hbar_x)\subset\scr{I}(h^{p+1})$. 
    Applying the Nadel-type vanishing Theorem \ref{Thm Nadel type vanishing without Kahler} to $(\cal{L}^{\otimes n+\ell}\otimes L^{\otimes p+1},\hbar_x)$ on $X_c$, for any $x\in K$ the following restriction map is surjective.
    \begin{align*}
        H^0(X_c,(K_X\otimes L^{\otimes m})^{\otimes n+1+\ell}\otimes L^{\otimes p}\otimes\scr{I}(\hbar_x))\longrightarrow\cal{O}_X/\fra{m}^2_{X,x}
    \end{align*}

    By \cite[Lemma\,5.1]{Tak98} and the above extension to a natural singular Hermitian metric, for any distinct two points $x$ and $y$ on $K$, there exists a singular Hermitian metric $\cal{H}_{x,y}$ on $\cal{L}^{\otimes n+1}$ 
    such that the curvature current is semi-positive, $\{x,y\}\subset V(\scr{I}(\cal{H}_{x,y}))$ and $x$ is isolated in $V(\scr{I}(\cal{H}_{x,y}))$. 
    By defining the singular Hermitian metric \( \hbar_{x,y} \) on $\cal{L}^{\otimes n+\ell}\otimes L^{\otimes p+1}$ as \( \hbar_{x,y}:=\cal{H}_{x,y}\otimes h_{V_K}^{\otimes\ell-1}\otimes h^{\otimes p+1} \), and similarly applying the Nadel-type vanishing Theorem \ref{Thm Nadel type vanishing without Kahler}, 
    there exists a holomorphic section \( \sigma\in H^0(X_c,(K_X\otimes L^{\otimes m})^{\otimes n+1+\ell}\otimes L^{\otimes p}\otimes\scr{I}(\hbar_{x,y})) \) that satisfies \( \sigma(x)\ne\sigma(y) \). Here, $\scr{I}(\hbar_{x,y})\subset\scr{I}(h^{p+1})$.

    Since the singular Hermitian metrics \( \hbar_x \) and \( \hbar_{x,y} \) are defined on \( X \), by applying the Approximation Theorem \ref{Approximation thm of hol section with ideal sheaves}, each section obtained on \( X_c \) can be obtained as a section on the entire \( X \).
    Hence, every relatively open subset $K$ of $\scr{X}$ can be embedded into a projective space by some finite dimensional subspace of $H^0(X,(K_X\otimes L^{\otimes m})^{\otimes n+1+\ell}\otimes L^{\otimes p}\otimes\scr{I}(h^{p+1}))$.
    Here, $\scr{X}$ is countable at infinity.
    Then this theorem follows from well-known argument of Whitney (cf. \cite[Chapter\,V]{Hor90}). 
\end{proof}

\begin{remark}\label{remark of only analytic singularities} 
    If \( h \) has only analytic singularities, then \( E_{+}(h) = \bigcup_{k\in\bb{N}}V(\scr{I}(h^k)) \) is an analytic subset, and under the same assumptions as Theorem \ref{Thm global embedding}, the complement  \( X\setminus E_{+}(h) \) is holomorphically embeddable in to \( \bb{P}^{2n+1} \).
\end{remark}

By using Theorem \ref{Ext [Tak98, Thm 4.2]}, we obtain the following in the same way as the proof above.

\begin{corollary}
    Let $m$ be a positive integer with $m\geq n(n+1)/2$, then the holomorphic line bundle $K_X\otimes L^{\otimes m}$ is singular semi-positive and big on $X_c$, and 
    the open subset $X_c\setminus Z_c$ is holomorphically embeddable in to $\bb{P}^{2n+1}$ by a linear subsystem of $|(K_X\otimes L^{\otimes m})^{\otimes n+1+\ell}\otimes L^{\otimes p}\otimes\scr{I}(h_c^{p+1})|$ for any $\ell\geq1$ and $p\geq0$.
\end{corollary}

Note that the \( p_c \) that leads to bigness in Theorem \ref{Thm bigness with ideal sheaves on each sublevel sets} can be estimated as \( p_c \geq n(n+1)/2 \) without considering the multiplier ideal sheaf $\scr{I}(h_c^{p_c})$.

For any $c>\inf_X\varPsi$ and any smooth Hermitian metric $h_0$ on $L$ over $X$, we define $\scr{C}_{c,h_0}:=\min\{C>0\mid C\omega_X\geq\iO{L,h_0} \text{ on } \overline{X}_c\}$.
Let \( \scr{C}_{c,h_0} \) be taken as small as possible by varying $c>\inf_X\varPsi$ and \( h_0 \), and denote it as \( \fra{C}_c \). 
Using this $\fra{C}_c$, we obtain the following.

\begin{theorem}$($\textnormal{\textbf{Global bigness}}$)$\label{Thm global bigness theorem}
    There exists a positive integer $m_0$ such that for any integer $m\geq m_0$, the adjoint bundle $K_X\otimes L^{\otimes m}$ is singular semi-positive and big on $X$.
    Here, the integer \( m_0 \) can be taken as 
    \begin{align*}
        m_0=\lfloor\fra{C}_c\frac{n(n+1)}{2}\rfloor+1.
    \end{align*}
\end{theorem}

Clealy, the adjoint bundle $K_X\otimes L^{\otimes m+1}$ is also singular positive on $X$.

\begin{proof}
    By the construction of \( h_c \), we have \( \iO{L,h_c}\geq(1-\delta_c)\omega_X \) on $X_c$ for some \( \delta_c>0 \). 
    Since this metric $h_c$ can be replaced while preserving $\scr{I}(h_c^m)=\scr{I}(h^m)$ on $X_c$ as $\delta_c$ approaching zero by blow-ups Theorem \ref{Blow ups of Dem appro with log poles and ideal sheaves}, we may assume that $m_0>\frac{\fra{C}_c}{1-\delta_c}n(n+1)/2$. 
    From Remark \ref{Remark of points separation thm on Xc}, for any relatively compact subset $K$ of $X_c\setminus Z_c$, there exists a finite dimensional subspace $W_K\subset H^0(X_c,K_X\otimes L^{\otimes m}\otimes\scr{I}(h_c^m))$ such that the Kodaira map $\varPhi_{W_K}:X_c\dashrightarrow \bb{P}(W_K)$ is holomorphic and one to one over $K$. 
    Using the equation $\scr{I}(h_c^m)=\scr{I}(h^m)$ and 
    Approximation Theorem \ref{Approximation thm of hol section with ideal sheaves}, there exists a finite dimensional subspace $V_K\subset H^0(X,K_X\otimes L^{\otimes m}\otimes\scr{I}(h^m))$ such that the Kodaira map $\varPhi_{V_K}:X_c\dashrightarrow \bb{P}(V_K)$ is also holomorphic and one to one over $K$.
    Hence, $K_X\otimes L^{\otimes m}$ has a singular semi-positive Hermitian metric $h_{V_K}$ and is big on $X$.
    In fact, similarly to the proof of Theorem \ref{Thm global embedding}, the linear subsystem of \( |(K_X\otimes L^{\otimes m})^{\otimes n+2}\otimes\scr{I}(h)| \) has a finite dimensional subspace $\scr{V}_K$ that induces a Kodaira map $\varPhi_{\scr{V}_K}$ embedding the set \( K \).
\end{proof}

\begin{theorem}\label{Thm global bigness and Vol if Zc is empty}
    If there exists $c>\inf_X\varPsi$ such that $E_{+}(h)\bigcap X_c=\emptyset$, then  
    the adjoint bundle $K_X\otimes L^{\otimes m}$ is big for any integer $m\geq n(n+1)/2$. 
    Furthermore, if $m\geq n(n-1)/2+1$, then the adjoint bundle $K_X\otimes L^{\otimes m}$ is singular semi-positive, and we have the inequality $\rom{Vol}_X(K_X\otimes L^{\otimes m+1})>0$.
\end{theorem}

\begin{proof}
    We take the singular positive Hermitian metric \( h_c \) on $L|_{X_c}$ from blow-ups Theorem \ref{Blow ups of Dem appro with log poles and ideal sheaves} as in $\S\ref{Setting of each sublevel set}$, and let \( Z_c \) be a singular locus of $h_c$.
    It follows from the assumption that $Z_c=\emptyset$ and $\scr{I}(h_c^p)=\scr{I}(h^p)=\cal{O}_X$ on $X_c$ for any $p\in\bb{N}$. In other words, \( h_c \) is simply smooth positive. In particular, in Theorem \ref{Thm of Ext [Tak98, Thm 4.1 and 4.2]} we may take $\scr{C}_c=1$.
    We take an integer $m\geq n(n-1)/2+1$. By Theorem \ref{Ext [Tak98, Thm 4.2]}, there exists a finite dimensional subspace $V\subset H^0(X_c,K_X\otimes L^{\otimes m})$ such that $\rom{Bs}_{|V|}=\emptyset$, 
    and by applying Approximation Theorem \ref{Approximation thm of hol section with ideal sheaves}, there exists a finite dimensional subspace $\cal{V}\subset H^0(X,K_X\otimes L^{\otimes m}\otimes\scr{I}(h^m))$ induced by $V$ such that $\rom{Bs}_{|\cal{V}|}\bigcap X_c=\emptyset$.
    Thus, \( K_X\otimes L^{\otimes m} \) carries a natural singular semi-positive Hermitian metric $h_{\cal{V}}$ on $X$ induced by $\cal{V}$, which is smooth semi-positive on $X_c$.
    If $m\geq n(n+1)/2$, then similarly to the proof of Theorem \ref{Thm global embedding}, the linear subsystem of \( |(K_X\otimes L^{\otimes m})^{\otimes n+2}\otimes\scr{I}(h)| \) has a finite dimensional subspace $\scr{V}_{X_c}$ that induces a Kodaira map $\varPhi_{\scr{V}_{X_c}}$ embedding the set \( X_c \).
    Hence, the adjoint bundle $K_X\otimes L^{\otimes m}$ is big.

    Here, $\cal{L}:=K_X\otimes L^{\otimes m+1}$ admits a singular positive Hermitian metric $\cal{H}:=h_{\cal{V}}\otimes h$ over $X$, where $m\geq n(n-1)/2+1$.
    From Theorem \ref{singular holomorphic Morse inequality on w.p.c} and its proof, 
    for a suitable number $b\in(\inf_X\varPsi,c)$ and a smooth positive Hermitian metric $\cal{H}_b:=h_{\cal{V}}\otimes h_c$ on $\cal{L}|_{X_b}$, we obtain the following singular holomorphic Morse inequality 
    \begin{align*}
        \rom{dim}\,H^0(X_b,\cal{L}^p)\geq\frac{p^n}{n!}\int_{X_b(\leq1)}\Big(\frac{i}{2\pi}\Theta_{\cal{L},\cal{H}_b}\Big)^n+o(p^n).
    \end{align*}
    For any $p\in\bb{N}$, let $H_p:=h^{m+1}\otimes\cal{H}^p=h^{m+p+1}\otimes h_{\cal{V}}^p$ be a singular Hermitian metric on $L^{\otimes m+1}\otimes\cal{L}^{\otimes p}$ over $X$, then $H_p$ is singular positive and we have $\scr{I}(H_p)=\scr{I}(h^{m+p+1})=\cal{O}_X$ on $X_b$. 
    By the applying Approximation Theorem \ref{Approximation thm in Introduction} and Corollary \ref{Corollary equation of dim H0}, the following inclusion relation is derived.
    \begin{align*}
        \rom{dim}\,H^0(X,\cal{L}^{\otimes p+1})&\geq\rom{dim}\,H^0(X,K_X\otimes L^{\otimes m+1}\otimes\cal{L}^{\otimes p}\otimes\scr{I}(H_p))\\
        &=\rom{dim}\,H^0(X_b,K_X\otimes L^{\otimes m+1}\otimes\cal{L}^{\otimes p}\otimes\scr{I}(H_p))=\rom{dim}\,H^0(X_b,\cal{L}^{\otimes p+1}).
    \end{align*}
    Therefore, from the above discussion, we obtain \( \rom{Vol}_X(K_X\otimes L^{\otimes m+1}) > 0 \).
\end{proof}


Here, let us consider the case where we additionally assume that $L$ is semi-positive.
In this situation, we first observe that in Theorem \ref{Thm global embedding in Introduction} the constant $C$ may be chosen as any number greater than $1$.
In fact, let $h_{\geq0}$ be a smooth semi-positive Hermitian metric on $L$, then for any $\varepsilon\in (0,1)$, the newly constructed singular Hermitian metric $\widetilde{h}:=h^{1-\varepsilon}\cdot h^{\varepsilon}_{\geq0}$ on $L$ is also singular positive and satisfies 
$\iO{L,\widetilde{h}}\geq(1-\varepsilon)\omega_X+\varepsilon\iO{L,h_{\geq0}}=:\varpi_{X}>0$ in the sense of currents and $\varpi_X\geq\varepsilon \iO{L,h_{\geq0}}$, i.e., $(1/\varepsilon)\varpi_X\geq\iO{L,h_{\geq0}}$, here $h$ is singular positive and satisfies $\iO{L,h}\geq\omega_X$ in the sense of currents. 
Therefore, we may take $C=1/\varepsilon$. The following theorem provides that $C=1$ can be taken.

\begin{theorem}\label{Thm global embedding with semi-positive}
    Let $X$ be a non-compact weakly pseudoconvex manifold and $(L,h)$ be a singular positive line bundle.
    Furthermore, if $L$ is semi-positive, we have the following. 
    \begin{itemize}
        \item [] \hspace{-10mm} $($\textnormal{\textbf{Global points separation}}$)$ Let $x_1,\ldots,x_r$ be $r$ distinct points on $X\setminus\bigcup_{k\in\bb{N}}\!V(\scr{I}(h^k))$. 
        Then for any positive integer $m\geq \big(n(n+2r-3)/2+2-r\big)$, there exists a positive rational number $\vartheta_r\in\bb{Q}_{>0}$ smaller than $m$ such that the space $H^0(X,K_X\otimes L^{\otimes m}\otimes\scr{I}(h^{m-\vartheta_r}))$ separates $\{x_j\}^r_{j=1}$, 
        i.e., the restriction map $H^0(X,K_X\otimes L^{\otimes m}\otimes\scr{I}(h^{m-\vartheta_r}))\longrightarrow \bigoplus^r_{j=1}\cal{O}_X/\fra{m}_{X,x_j}$ is surjective. 

        In particular, if $m\geq \big(n(n-1)/2+1\big)$ then the adjoint bundle $K_X\otimes L^{\otimes m}$ is generated by global sections on $X\setminus\bigcup_{k\in\bb{N}}V(\scr{I}(h^k))$. 
        \item [] \hspace{-10mm} $($\textnormal{\textbf{Global embedding}}$)$ Let $m$ be a positive integer with $m\geq n(n+1)/2$. 
        Then, any open subset of $X\setminus\bigcup_{k\in\bb{N}}V(\scr{I}(h^k))$ 
        is holomorphically embeddable in to $\bb{P}^{2n+1}$ by a linear subsystem of $|(K_X\otimes L^{\otimes m})^{\otimes n+1+\ell}\otimes L^{\otimes p}\otimes\scr{I}(h^{p+1})|$ for any $\ell\geq1$ and $p\geq0$.
        In particular, 
        the adjoint bundle $K_X\otimes L^{\otimes m}$ is ample over any open subset of $X\setminus\bigcup_{k\in\bb{N}}V(\scr{I}(h^k))$, and 
        the tensor power of the adjoint bundle $(K_X\otimes L^{\otimes m})^{\otimes n+2}$ is very ample over any open subset of $X\setminus\bigcup_{k\in\bb{N}}V(\scr{I}(h^k))$.
    \end{itemize}
\end{theorem}

\begin{proof}
    We take an integer $m\geq \big(n(n+2r-3)/2+2-r\big)$. 
    There exist $c>\inf_X\varPsi$, a singular Hermitian metric $h_c$ on $L|_{X_c}$ and an analytic subset $Z_c\subset X_c$ as in Setting \ref{Setting of each sublevel set}
    such that $x_1,\ldots,x_r\in X_c\setminus Z_c\subset X\setminus\bigcup_{k\in\bb{N}}V(\scr{I}(h^k))$. By Theorem \ref{Blow ups of Dem appro with log poles and ideal sheaves} or (\ref{subsection: canonical sHm}\,$d$), the singular metric $h_c$ can be chosen such that $\scr{I}(h_c^t)=\scr{I}(h^t)$ on $X_c$ for any $0<t\leq m$. 

    By Theorem \ref{Ext [Tak98, Thm 4.2]}, there exists a rational number $\vartheta_r\in \bb{Q}_{>0}$ smaller than $m$ such that the space $H^0(X_c,K_X\otimes L^{\otimes m}\otimes\scr{I}(h^{m-\vartheta_r}))$ separates $\{x_j\}^r_{j=1}$, here $\scr{I}(h_c^{m-\vartheta_r})=\scr{I}(h^{m-\vartheta_r})$ on $X_c$. 
    Let $h_{\geq0}$ be a smooth semi-positive Hermitian metric on $L$, then the newly constructed singular Hermitian metric $\widetilde{h}:=h^{\frac{m-\vartheta_r}{m}}\cdot h^{\frac{\vartheta_r}{m}}_{\geq0}$ on $L$ is also singular positive and satisfies $\scr{I}(\widetilde{h}^m)=\scr{I}(h^{m-\vartheta_r})$ on $X$. 
    Hence, by applying the Approximation Theorem \ref{Approximation thm in Introduction} to $(L^{\otimes m},\widetilde{h}^m)$, we obtain global points separation. 
    The global embedding is proved using this global points separation, similarly to the proof of Theorem \ref{Thm global embedding}. 
\end{proof}

\begin{theorem}
    If a weakly pseudoconvex manifold $X$ has a singular positive line bundle, 
    then $X$ is generalized Moishezon.
    Conversely, if a complex manifold $X$ is generalized Moishezon, 
    then there exist a singular semi-positive Hermitian metric $h$ on a holomorphic line bundle and analytic subsets $A$ and $B$ of $X$ 
    such that $h$ is smooth on $X\setminus B$ and positive on $X\setminus (A\cup B)$.
\end{theorem}

\begin{proof}
    Here, Theorem \ref{Ext [MM07, Theorem 2.2.15]} shows that generalized Moishezon-ness is equivalent to the existence of a big line bundle.
    The first claim follows directly from Global bigness Theorem \ref{Thm global bigness theorem}.
    We prove the converse. 
    Let $L$ be a big line bundle on $X$, then there exist a positive integer \( m\in\bb{N} \) and a finite dimensional subspace $V\subset H^0(X,L^{\otimes m})$ such that the Kodaira map \( \varPhi_{V}:X\dashrightarrow \bb{P}^N \) satisfies \( \rom{dim}\,\overline{\varPhi_{V}(X)} = n \), where $N:=\rom{dim}\,V-1$.
    By taking $\sigma_0,\ldots,\sigma_N$ as a basis of \( V \), and defining the natural singular Hermitian metric $h_\sigma$ on $L^{\otimes m}$ induced by this basis as \( h_\sigma = 1/(\sum^N_{j=0}|\sigma_j|^2) \), we obtain $L^{\otimes m}\cong\varPhi_V^*\cal{O}_{\bb{P}^N}(1)$ and $\iO{L^{\otimes m},h_\sigma}=\varPhi_{V}^*\omega_{FS}$ on $X\setminus\rom{Bs}_{|V|}$.
    We take this \( \text{Bs}_{|V|} \) as the analytic set \( B \), then 
    $h_\sigma$ is smooth and semi-positive on $X\setminus B$. 
    As in the proof of Theorem \ref{Thm global embedding}, by the extension of the plurisubharmonic function, the metric $h_\sigma$ can be defined on the whole of $X$ as a singular semi-positive Hermitian metric of $L^{\otimes m}$.
    Here, let \( A \) denote the locus where the rank of the Kodaira map $\varPhi_{V}$ is not maximal, which then becomes an analytic subset of \( X\setminus B \).
    Since $\varPhi_V$ is an immersion on $X\setminus(A\cup B)$, we have that $h_\sigma$ is positive, i.e., $\iO{L^{\otimes m},h_\sigma}=\varPhi_{V}^*\omega_{FS}>0$, on $X\setminus(A\cup B)$.
\end{proof}


\noindent
{\bf Acknowledgement.} 
The author would like to thank Taro Yoshino, a Ph.D. student, for the helpful discussion about algebraic geometry. 
The author would also like to thank Professor Shigeharu Takayama, Shin-ichi Matsumura and Takeo Ohsawa for their useful advice.
The author is supported by Grant-in-Aid for Research Activity Start-up $\sharp$24K22837 from the Japan Society for the Promotion of Science (JSPS).



\end{document}